\documentclass[letterpaper,11pt,reqno]{amsart} 
\RequirePackage[utf8]{inputenc}
\usepackage[portrait,margin=3cm]{geometry}
\usepackage{color}              
\usepackage[usenames,dvipsnames]{xcolor}

\usepackage{mathrsfs} 
\usepackage{hyperref}
\usepackage[foot]{amsaddr}
\usepackage{amssymb,amsthm,amsmath,amsfonts,amsbsy,latexsym} 
\usepackage{graphicx}
\usepackage[numeric,initials,nobysame]{amsrefs}

\definecolor{BrickRed}{rgb}{0.65,0.08,0}

\usepackage{enumerate}

\numberwithin{equation}{section}
\numberwithin{figure}{section}
\numberwithin{table}{section}

\newtheorem{Lemma}{Lemma}[section]
\newtheorem{Proposition}[Lemma]{Proposition}

\newtheorem{Theorem}[Lemma]{Theorem}

\newtheorem{Corollary}[Lemma]{Corollary}

\theoremstyle{definition}
\newtheorem{Example}[Lemma]{Example}

\newtheorem{Condition}[Lemma]{Condition}
\newtheorem{Construction}[Lemma]{Construction}
\newtheorem{Remark}[Lemma]{Remark}

\newcommand{\myroot}{{\text{\o}}}
\newcommand{\Erdos}{Erd\H{o}s-R\'enyi }

\newcommand{\lan}{\langle}
\newcommand{\ran}{\rangle}

\def\SQ{S^{\sqcup}}

\newcommand{\Emb}{{\mathbb{E}}}

\newcommand{\Nmb}{{\mathbb{N}}}

\newcommand{\Pmb}{{\mathbb{P}}}

\newcommand{\Rmb}{{\mathbb{R}}}

\newcommand{\Umb}{{\mathbb{U}}}
\newcommand{\Vmb}{{\mathbb{V}}}

\newcommand{\Gmc}{{\mathcal{G}}}

\newcommand{\Lmc}{{\mathcal{L}}}

\newcommand{\Ebf}{{\mathbf{E}}}

\newcommand{\zero}{{\boldsymbol{0}}}

\newcommand{\xbd}{{\boldsymbol{x}}}\newcommand{\Xbd}{{\boldsymbol{X}}}
\newcommand{\xibd}{{\boldsymbol{\xi}}}
\newcommand{\Ybd}{{\boldsymbol{Y}}}

\newcommand{\Xbar}{{\bar{X}}}
\newcommand{\xibar}{{\bar{\xi}}}

\newcommand{\Ybar}{{\bar{Y}}}
\newcommand{\Zbar}{{\bar{Z}}}

\newcommand{\Nhat}{{\hat{N}}}

\newcommand{\Atil}{{\tilde{A}}}
\newcommand{\Btil}{{\tilde{B}}}
\newcommand{\Ctil}{{\tilde{C}}}

\newcommand{\ftil}{{\tilde{f}}}

\newcommand{\phitil}{{\tilde{\phi}}}

\newcommand{\Ztil}{{\tilde{Z}}}

\newcommand{\PolishX}{{\mathcal{X}}}

\newcommand{\Polish}{{\mathcal{Z}}}
\newcommand{\polish}{{z}}
\newcommand{\local}{{local-field equations}}
\newcommand{\neighborhood}{{marginal dynamics}}

\renewcommand{\iff}{\Leftrightarrow}

\def\qed{ \hfill $\blacksquare$}  

\newcommand{\N}{\mathbb{N}}

\newcommand{\R}{\mathbb{R}}

\newcommand{\E}{\mathbb{E}}

\def\tree{\mathcal{T}}
\def\mom{\pi}
\def\cmeas{\gamma}
\def\V{\mathbb{V}}
\def\X{\mathcal{X}}
\def\P{\mathcal{P}}
\def\G{\mathcal{G}}
\def\Z{\mathbb{Z}}
\def\PP{\mathbb{P}}
\def\Nspace{\Xi}

\def\Var{\mathrm{Var}}
\def\Cov{\mathrm{Cov}}

\DeclareFontFamily{U}{skulls}{}
\DeclareFontShape{U}{skulls}{m}{n}{ <-> skull }{}

\def\cem{\small{\varpi}}

\newcommand{\be}{\begin{equation}}
\newcommand{\ee}{\end{equation}}
\newcommand{\bes}{\begin{equation*}}
\newcommand{\ees}{\end{equation*}}
\newcommand{\beqn}{\begin{eqnarray}}
\newcommand{\eeqn}{\end{eqnarray}}
\newcommand{\beq}{\begin{eqnarray*}}
\newcommand{\eeq}{\end{eqnarray*}}
\newcommand{\ba}{\begin{aligned}}
\newcommand{\ea}{\end{aligned}}

\numberwithin{equation}{section}

\begin{document}
	
\title[Marginal dynamics of probabilistic cellular automata on trees]{Marginal dynamics of probabilistic cellular automata on trees}

\date{\today}
\author[Lacker]{Daniel Lacker}
    \address{Columbia University, New York, New York} 
    \thanks{D.\ Lacker acknowledges support from an Alfred P.\ Sloan Fellowship and the NSF CAREER award DMS-2045328.}
\author[Ramanan]{Kavita Ramanan}
    \thanks{K.\ Ramanan was supported in part by NSF-DMS Grants 1713032 and NSF DMS-1954351, ARO Grant W911NF2010133 and the VBFF Faculty Fellowship ONR-N0014-21-1-2887} 
    \address{Division of Applied Mathematics, Brown University, 182 George Street, Providence, RI 02912} 
\author[Wu]{Ruoyu Wu}
    \thanks{R.\ Wu was supported in part by NSF DMS-2308120 and Simons Foundation travel grant MP-TSM-00002346}
    \address{Department of Mathematics, Iowa State University, 411 Morrill Rd, Ames, IA 50011} 
\email{daniel.lacker@columbia.edu, kavita\_ramanan@brown.edu, ruoyu@iastate.edu}

\begin{abstract}
    We study locally interacting processes in discrete time, often called probabilistic cellular automata, indexed by locally finite graphs. For infinite regular trees and certain generalized Galton-Watson trees, we show that the marginal evolution at a single vertex and its neighborhood can be characterized by an autonomous stochastic recursion referred to as the local-field equation. This evolution can be viewed as a nonlinear or measure-dependent chain, but the measure dependence arises from the symmetries of the underlying tree rather than from any mean field interactions. We discuss applications to simulation of marginal dynamics and approximations of empirical measures of interacting chains on several generic classes of large-scale finite graphs that are locally tree-like. In addition to the symmetries of the tree, a key role is played by a second-order Markov random field property, which we establish for general graphs along with some other novel Gibbs measure properties.
\end{abstract}
	 	 	
\maketitle

\tableofcontents

\section{Introduction}

We study discrete-time systems of interacting stochastic processes, indexed by the vertices of a tree, with the state transitions of each  vertex depending only on its nearest neighbors. More precisely, for a connected locally finite graph $G=(V,E)$ with a finite or countable vertex set, we consider the following dynamics:
\begin{equation}
    \label{intro:mainsystem}
    X^G_v(k+1) = F(X^G_v(k),\mu^G_v(k),\xi_v(k+1)), \qquad k \in \N_0, \ \ v \in V, 
\end{equation}
where $\mu^G_v(k)$ is the neighborhood empirical measure given by
\begin{equation*}
    \mu^G_v(k) = \frac{1}{|N_v(G)|}\sum_{u \in N_v(G)}\delta_{X^G_u(k)}, 
\end{equation*}
$N_v(G) := \{u \in V : (u,v) \in E\}$ denotes the neighborhood of $v$, $F$ is a suitable mapping that governs the dynamics, and $\N_0 := \N \cup \{0\}$.  
The noises $(\xi_v(k))_{v \in V,\, k \in \N}$ are independent and identically distributed (i.i.d.), and the process $X = (X_v(k))_{k \in \N_0}$ take values in a common Polish space $\PolishX$ at each time step.  
Processes of the form \eqref{intro:mainsystem} are referred to by many names, including  \emph{synchronous} \emph{locally interacting Markov chains}, or perhaps most commonly \emph{stochastic} or \emph{probabilistic cellular automata}, and have been the subject of significant research in both theoretical and applied domains, but we postpone our discussion of the vast related literature to Section \ref{se:background} below.
In fact, our most general results stated in Section \ref{se:results} will also encompass non-Markovian (history-dependent) dynamics.

The primary goal of the present article is to derive what we call the \emph{\local} associated with \eqref{intro:mainsystem}, when the underlying graph $G$ is 
either an infinite regular tree or a (generalized) Galton-Watson tree.
The \local\ provide an autonomous dynamic characterization of the marginal  law 
of the stochastic process $X$ at a single site $v$ and its neighborhood $N_v(G)$ in a self-contained fashion, without reference to the rest of the graph. They also prescribe an  algorithm for simulating this marginal distribution.   
The  marginal distribution at a vertex also captures 
the limit, as $n \rightarrow \infty,$ of macroscopic observables captured by the global empirical measure 
\begin{equation}
    \label{def-mugn}
    \mu^{G_n} = \frac{1}{|G_n|} \sum_{u \in G_n} \delta_{X_u(k)}, 
\end{equation} 
whenever $\{G_n\}_{n \in \mathbb{N}}$ is a sequence of finite graphs that converge in a suitable local topology to $G$ (see \cite{LacRamWu-convergence,Ramanan2023interacting,GangulyRamanan2022hydrodynamic} for related results and also further discussion in Section \ref{se:finapprox}). Thus, the local-field equations also serve to approximate macroscopic observables of the dynamics on several classes of locally tree-like graphs, including random regular graphs, configuration models and \Erdos random graphs (e.g., see Section 2.2.4 of \cite{LacRamWu-convergence} and Section \ref{se:finapprox}). Additionally, they play a role in understanding rare events or large deviations of such particle systems \cite{RamYas23}. In the following section, we 
illustrate the essential ideas behind the derivation of the local-field equation in a simplified setting.

\subsection{Informal derivation of the \local\ on $\Z$} \label{se:Zlocdyn}

Let us focus on the simple case where the graph $G$ is the $2$-regular tree, that is,  the integer lattice $\Z$. Changing notation slightly, the particle system \eqref{intro:mainsystem} can be written as
\begin{equation}
	\label{intro:Zsystem}
	X_i(k+1) = F(X_i(k),X_{i+1}(k),X_{i-1}(k),\xi_i(k+1)), \quad i \in \Z, 
\end{equation}
where we crucially assume that $F$ is symmetric with respect to the neighboring states in the sense that $F(x,y,z,\xi)=F(x,z,y,\xi)$. It is then clear that  the particle system inherits the symmetries of the underlying graph, so that for each $k \in \N$ the relation $(X_i(k))_{i \in \Z} \stackrel{d}{=} (X_{j+i}(k))_{i \in \Z}$ and $(X_i(k))_{i \in \Z} \stackrel{d}{=} (X_{j-i}(k))_{i \in \Z}$ is satisfied for all $j \in \Z$, as long as this property is true at the initial time $k=0$. In fact, the same symmetries are valid for the trajectories $X_i[k] := (X_i(0),X_i(1),\ldots,X_i(k))$, not just the time-$k$ marginals.

In addition to symmetry, the other key ingredient in our \local\ is a conditional independence structure. We prove in Section \ref{sec:MC_graph} that the particle system \eqref{intro:mainsystem} forms a \emph{second-order Markov random field} over $G$, which in the present case of $G=\Z$ means that $(X_j[k])_{j < i}$ and $(X_j[k])_{j > i+1}$ are conditionally independent given $(X_i[k],X_{i+1}[k])$, for each $i \in \Z$ and $k \in \N$, again as long as we assume this conditional independence to be true at time $k=0$. This conditional independence property, interesting in its own right, is generic for dynamics of the form \eqref{intro:mainsystem}, in the sense that it is true for any underlying graph, not just trees. Moreover, it cannot be strengthened to a first-order Markov property, nor does it hold for the time-marginals; it is valid only at second order, and only at the level of trajectories. This is discussed in detail in Section \ref{se:counterexamples}.

Combining the symmetries and conditional independence lets us understand the \neighborhood\ as follows. 
Let $\nu[k]$ denote the joint law of $(X_1[k],X_0[k],X_{-1}[k])$, for $k \in \N_0$.
From $\nu[k]$ we can construct the (regular) conditional law of $X_1[k]$ given $(X_0[k],X_{-1}[k])$, denoted $\cmeas[k](\cdot \,|\, x_0,x_{-1}):= \Lmc(X_1[k] \,|\, X_0[k]=x_0,\,X_{-1}[k]=x_{-1})$.
Then the joint law of $(X_i[k])_{i \in \Z}$ is precisely
\begin{equation*}
	\nu[k](dx_1,dx_0,dx_{-1})\prod_{i=1}^\infty \gamma[k](dx_{i+1}\,|\,x_i,x_{i-1})\,\gamma[k](dx_{-i-1}\,|\,x_{-i},x_{-i+1}).
\end{equation*}
(Fixing $k$ and interpreting $i \in \Z$ as a time index, this is essentially saying that $(X_i[k])_{i \in \Z}$ is a stationary second-order Markov chain.)

Once we understand that the entire joint law of $(X_i[k])_{i \in \Z}$ is characterized by the marginal law $\nu[k]$ of $(X_i[k])_{i=-1,0,1}$, this strongly suggests that the dynamics of the whole infinite system can be obtained from the dynamics of sites $(-1,0,1)$. Indeed, suppose we know $(X_i[k])_{i=-1,0,1}$ and $\nu[k]$ at some time $k$. Construct the conditional measure $\cmeas[k]$ as above. We can then generate two samples $Z_2[k]$ and $Z_{-2}[k]$, which we call \emph{phantom particles}, which are conditionally independent given $(X_i[k])_{i=-1,0,1}$ with conditional law given by
\begin{align*}
	\Lmc\big(Z_2[k] &\in dz_2, \, Z_{-2}[k] \in dz_{-2} \,|\, (X_i[k])_{i=-1,0,1}\big) \\
	&= \cmeas[k]\big(dz_{-2} \,|\, X_{-1}[k],X_0[k]\big) \ \cmeas[k]\big(dz_{2} \,|\, X_{1}[k],X_0[k]\big).
\end{align*}
We then update $(X_i)_{i=-1,0,1}$ by setting
\begin{align*}
	X_1(k+1) &= F(X_1(k),Z_2(k),X_0(k),\xi_1(k+1)), \\
	X_0(k+1) &= F(X_0(k),X_1(k),X_{-1}(k),\xi_0(k+1)), \\
	X_{-1}(k+1) &= F(X_{-1}(k),X_0(k),Z_{-2}(k),\xi_{-1}(k+1)).
\end{align*}
This is consistent in distribution with the dynamics \eqref{intro:Zsystem}, because $(Z_{-2},X_{-1},X_0,X_1,Z_2)[k]$ and $(X_{-2},X_{-1},X_0,X_1,X_2)[k]$ have the same distribution, by construction (and because the $\xi_i(k+1)$'s are independent).

To summarize the \emph{\local} governing $(X_i)_{i=-1,0,1}$: At each time step $k$, we sample the two relevant missing particles at vertices $i=\pm 2$ using the conditional distribution $\cmeas[k]$ constructed from the histories of the middle three particles, use these phantom particles to update the three middle particles, discard the phantom particles, and repeat.

The interesting feature of these dynamics (which motivates the name) is that they are governed solely by the middle three particles, with no reference to particles $(X_i)_{|i| \ge 2}$. 
This is the key advantage of the \local, now three-dimensional compared to the original infinite-dimensional system \eqref{intro:Zsystem}.  The price one has to pay for this drastic reduction in dimension  is that the update procedure from time $k$ to $k+1$ depends not only on the realizations $(X_i(k))_{i=-1,0,1}$ but also on the joint law of the entire history up to time $k$. This feature renders the \local\ non-Markovian.  
Nonetheless, we demonstrate in Appendix \ref{sec:char_regular_tree} how the \local\ can be tractable in some cases. 
Moreover, for other classes of processes relevant in applications including SIR-type models, the \local\ in fact reduce to simpler tractable  Markov recursions (see \cite{Wortsman18} for numerical evidence and \cite{CocomelloRamanan2023exact,Cocomello-thesis} for results for continuous time jump processes).

The \local\ described above are a (non-standard) example of a \emph{nonlinear Markov chains} (cf.\ \cite{del2004feynman}). A nonlinear Markov chain $M$ is described by dynamics of the form
\begin{equation}
	\label{intro:nonlinMC}
	M(k+1)=F(M(k),\Lmc(M(k)),\xi(k+1)), 
\end{equation}
where $\Lmc(M(k))$ denotes the law of $M(k)$. Such dynamics arise as  $n\to\infty$ (mean field) limits  of interacting Markov chains of the form \eqref{intro:mainsystem}, where now the graph $G = G_n$ is the complete graph on $n$ vertices, $F$ is a suitable mapping that is continuous with respect to weak convergence in the second variable and $(X^{G_n}_v(0))_{v \in G_n}$ are i.i.d. In this case both $\mu^{G_n}$ and $X^{G_n}_{v_n}$ converge in law to a nonlinear Markov chain $M$ (where the vertex $v_n$ is an arbitrary vertex in $G_n$).  
In fact, the \local\ described above can be written in the same form (albeit history-dependent) as \eqref{intro:nonlinMC} with $M=(X_{-1},X_0,X_1)$ representing the trajectories and with $F$ and $\xi$ redefined appropriately. But the resulting function $F$ will rarely be continuous, because its dependence on the measure involves several conditional distributions  associated with the law $\Lmc(M(k))$.

\subsection{Additional features of the \local\ on trees}

We will save the details for the body of the paper, but we highlight here some of the key points in deriving the \local\ on more general trees. It is not too difficult to adapt the discussion of Section \ref{se:Zlocdyn} to $d$-regular trees for $d > 2$. Again, the symmetries of the underlying tree (now richer than the simple shifts and reflections of $\Z$) and a conditional independence structure play a central role. On the $d$-regular tree $G$, we show that if one removes two adjacent vertices, then the particles indexed by the two remaining (disjoint) subtrees are conditionally independent given the particles indexed by these two removed vertices. (To be precise, this is true on the trajectory-level up to any time $k$, as long as it is true at time $k=0$; see Theorem \ref{thm:cond_independence}.) The \local\ now describe the evolution of a single particle and its neighbors, resulting in a $(d+1)$-dimensional process. At each time step, each of the $d$ neighbors is updated by sampling $d-1$ phantom particles, for a total of $d(d-1)$ phantom particles. See Section \ref{se:regresults} for precise statements.

The case of a random tree $G$, while following similar principles, is much more subtle; full details  are given in Section \ref{se:GWresults}. We focus on Galton-Watson trees, or the generalization thereof denoted by $\mathrm{GW}(\rho_{\myroot},\rho)$ which is governed by two offspring distributions; the number of offspring of the root is distributed according to $\rho_{\myroot}$, while those 
of subsequent generations independently have distribution  $\rho$. The \local\ now require tracking the joint law of the particle histories and tree structure of the first two generations, with a random number of phantom particles appearing at each time step.
Interestingly, if $\rho_{\myroot}$ and $\rho$ relate in precisely the right manner to render the random tree $\mathrm{GW}(\rho_{\myroot},\rho)$ \emph{unimodular}, then \local\ consisting of one instead of two generations suffices, at the cost of a slightly more complicated evolution; see Section \ref{se:UGWresults} for details. 
It is worth emphasizing that even when 
the tree is unimodular, the slightly simpler form of the  local-field equations on the larger $2$-neighborhood of the root described in Section \ref{se:GWresults} may turn out to be more useful. For example, it was used in \cite{Cocomello-thesis} to characterize the hydrodynamic limits of  certain hybrid interacting systems.

The family of $\mathrm{GW}(\rho_{\myroot},\rho)$ trees possesses convenient spatial homogeneity and conditional independence structures, and this lets us adapt the ideas described above for $\Z$ and for $d$-regular trees. It is conceptually clear, however, that a form of \local\ should be available on any suitably homogeneous random tree with some conditional independence to work with. A simple example would be a (deterministic) periodic tree.
We do not pursue this generalization here.

\subsection{Conditional independence and Gibbs measures} \label{se:Markovfield}

As mentioned above, a crucial second-order Markov random field property, developed in detail in Section \ref{sec:MC_graph}, underlies our derivation of the \local. This property, valid for the particle system \eqref{intro:mainsystem} set on a general graph $G=(V,E)$, is interesting in its own right and appears to be new in the literature on probabilistic cellular automata. We study two other related properties in Section \ref{se:consist-Gibbs}. 
First, we prove a Gibbs-type uniqueness property (Theorem \ref{thm:Gibbs_uniqueness}), which states that (under suitable conditions) the joint law of $(X^G_v[k])_{v \in V}$ is uniquely determined by the joint law of the initial states $(X^G_v(0))_{v \in V}$ and the \emph{specifications}, that is, the family of conditional laws of $(X^G_v[k])_{v \in A}$ given $(X^G_v[k])_{v \in {\partial A}}$, ranging over finite sets $A \subset V$.
The second is a consistency property (Proposition \ref{thm:cond_consistency}), which states that for a finite set of vertices $A \subset V$ the conditional law of $(X^G_v[k])_{v \in A}$ given $(X^G_v[k])_{v \in \partial^2 A}$ does not depend on the structure of the graph $G$ outside of $A \cup \partial^2A$.
Here we use the following notation for the first and second boundaries:
\begin{equation}
	\label{def:boundaries}
	\partial A = \{u \in V \setminus A : (u,v) \in E \text{ for some }v \in A\}, \qquad
	\partial^2 A = \partial A \cup \partial (A \cup \partial A).
\end{equation}

\subsection{Background and motivation} \label{se:background}

The  class of particle systems of the form \eqref{intro:mainsystem} might well be named \emph{symmetric probabilistic} (or \emph{probabilistic}) \emph{cellular automata}.
The key features are that the state $X_v(k)$ at each site $v \in V$ is updated \emph{simultaneously} from time $k$ to $k+1$, and these updates are done conditionally independently given the time-$k$ value of the configuration $(X_v(k))_{v \in V}$, with the update at a site $v$ depending in a symmetric fashion on only the neighboring sites $N_v(G)$.

Probabilistic cellular automata (PCA) have seen substantial application in a wide range of physical and social sciences. The recent book \cite{LouisNardi} contains a thorough account of recent developments and a wonderful introductory chapter surveying the history and applications. 
Modern probabilists are perhaps better acquainted with PCA's cousins, the asynchronous  models of interacting particle systems and Glauber dynamics in which sites are updated one at a time in a Poissonian fashion, such as those treated in \cite{Liggett}.
In statistical mechanics, PCA provide an intriguing alternative for modeling non-equilibrium phenomena like metastability  \cite{grinstein1985statistical,lebowitz1990statistical,bigelis1999critical}.
PCA are also a natural and popular framework for modeling the evolution of various processes on or through a network, including but not limited to epidemics such as SIS/SIR-type models \cite{grassberger1983critical,boccara1993critical}, spatial patterns and dynamics in ecological systems \cite{DurrettLevin}, economic models of price fluctuations and social interactions \cite{follmer1994stock,horst2005financial,cont2010social,FraimanLinOlvera2024opinion}, and default cascades in financial networks \cite{hurd2016contagion}. See also \cite[Section 1]{LouisNardi} for additional discussion of applications including computer science, neuroscience, and other natural sciences.

Our primary interest is in large-scale models of the form \eqref{intro:mainsystem}. 
In a companion paper \cite{LacRamWu-convergence} (see also \cite{OliReiSto19} and \cite{GangulyRamanan2022hydrodynamic} 
for related results for diffusion models and continuous time Markov chains, respectively), we showed that the particle system behaves well with respect to \emph{local convergence} of graphs. Many popular sparse random graph models converge locally to trees of the kind considered in this paper; for example, the \Erdos graph with constant edge density and the configuration model are known converge to $\mathrm{GW}(\rho_{\myroot},\rho)$ trees. 
Our \local\ thus characterize the behavior of a typical neighborhood in these sparse large-scale models, as explained in detail in Section \ref{se:finapprox} below, addressing a question raised in \cite{Ram22} and providing a kind of surrogate for the mean field approximation, which is valid only when $G$ is a complete or sufficiently dense graph, whereas many real-world networks are sparse.

Much of the prior literature on PCA is focused on stationarity, ergodicity, and Gibbs measure properties \cite{dawson1975synchronous,goldstein1989pca,lebowitz1990statistical,follmer2001convergence,dai2002stationary,louis2004ergodicity,dai2012sampling,BLMWY2024probabilistic}, which are remarkably delicate and present a host of challenges quite distinct from asynchronous continuous time models. Because of these challenges, much of the prior literature is limited to finite or even binary state spaces $\PolishX$. 
On the other hand, the idea of \local\ we develop in this paper, as well as the Markov random field properties discussed in Section \ref{se:Markovfield}, are robust enough to apply to models with general (Polish) state space.

While we do not in this paper explore any problems of statistical mechanics, large deviations, ergodicity, etc., we are optimistic that this new perspective of \local\ will prove useful in a variety of applications. Indeed, large deviations principles for the discrete-time systems studied here have been established in \cite{RamYas23}, and analogous \local\ obtained for diffusions in \cite{LacRamWu-diffusion} have been shown to be useful for the study of stationary distributions and long-time behavior in \cite{HuRam24,HuRam25,LacZha23}. 
For now, we limit our discussion of applications to Section \ref{se:finapprox} which, as mentioned above, explains how our \local\ characterize the limiting behavior of large finite-graph systems of the form \eqref{intro:mainsystem}.

\subsection{Organization of the paper}

In Section \ref{se:notation} we fix some terminology and notation to be used throughout the paper. 
Section \ref{se:results} states the main results on \local\ in full detail, including an application in Section \ref{se:finapprox} to large-scale limits of models on finite graphs.
Then Section \ref{sec:MC_graph} develops the key conditional independence structure for systems governed by general deterministic graphs. 
This is used to give a proof of the \local\ for (deterministic) regular trees, and the rest of Section \ref{sec:MC_graph} describes some counterexamples pertaining to other natural conditional independence structures. 
Finally Section \ref{sec:GW} gives the full proofs of the \local\ for Galton-Watson trees.

\section{Statements of main results} \label{se:results}

This section contains the main results on the \local, with all proofs deferred to later sections. 
We begin with regular trees in Section \ref{se:regresults} before moving to Galton-Watson trees in Section \ref{se:GWresults} and finally unimodular Galton-Watson trees in Section \ref{se:UGWresults}. 
Lastly, Section \ref{se:finapprox} describes applications to finite particle systems. 
First, we fix some basic notation in Section \ref{se:notation}.

\subsection{Common notation and terminology} \label{se:notation}

Throughout the paper, the main processes $X_v(k)$ will take values in a Polish space $\PolishX$, and the noise processes $\xi_v(k)$ will take values in another Polish space $\Nspace$. For any Polish space $\Polish$, we write $\P(\Polish)$ for the set of Borel probability measures on $\Polish$, endowed always with the topology of weak convergence. Denote by $\Lmc(Z)$ the law of a random variable $Z$. For two random variables $Z,Y$ taking values in Polish spaces, we write $\Lmc(Z|Y)$ or $\Lmc(Z|Y=y)$ for a version of the regular conditional law of $Z$ given $Y$. We write $Z \stackrel{d}{=} Y$ to mean that the random variables $Z$ and $Y$ have the same distribution.
We write $Z \perp Y$ to mean that $Z$ and $Y$ are independent, and similarly,
\[
	Z \perp Y \,|\, W
\]
means that $Z$ and $Y$ are conditionally independent given $W$, where $W$ is an additional random variable.
Lastly, we let $\N_0 := \N \cup \{0\}$ and write $|A|$ for the cardinality of a set $A$.

For an index set $I$ and a Polish space $\Polish$, we write $\Polish^I$ for the configuration space. We make use of a standard notation for configurations on subsets of vertices: For $\polish=(\polish_i)_{i \in I} \in \Polish^I$ and $J \subset I$, we write $\polish_J$ for the element $\polish_J:=(\polish_i)_{i \in J}$ of $\Polish^J$.

\subsubsection{Space of unordered terminating sequences} \label{se:symmetricsequencespace}
When working with random trees, the update rule of our processes will depend on an unspecified number of neighbors.
Hence, we need a formalism for specifying a single ``updating function" which takes as input finite sequences of arbitrary length from some space and are insensitive to the order of these elements. To this end, for a Polish space $\Polish$, 
we define in this paragraph a space $\SQ(\Polish)$ of finite unordered $\Polish$-valued sequences of arbitrary length (possibly zero). First, for $k \in \N$ we define the symmetric power (or unordered Cartesian product) $S^k(\Polish)$ as the quotient of $\Polish^k$ by the natural action of the symmetric group on $k$ letters. For convenience, let $S^0(\Polish) = \{\emptyset\}$. Define $\SQ(\Polish)$ as the disjoint union,
\[
\SQ(\Polish) = \bigsqcup_{k=0}^\infty S^k(\Polish).
\] 
A typical element of $\SQ(\Polish)$ will be denoted $(\polish_v)_{v \in A}$, for a finite (possibly empty) set $A$; if the set is empty, then by convention $(\polish_v)_{v \in A} = \emptyset \in S^0(\Polish)$. 
It must be stressed that the vector $(\polish_v)_{v \in A}$ has no order. 
Endow $\SQ(\Polish)$ with the disjoint union topology, i.e., the finest topology on $\SQ(\Polish)$ for which the injection $S^k(\Polish) \hookrightarrow \SQ(\Polish)$ is continuous for each $k \in \N$. 
This makes $\SQ(\Polish)$ a Polish space.

We make use of two operations on the space $\SQ(\Polish)$.
First, we write $|\polish|$ to denote the length of a vector $\polish \in \SQ(\Polish)$. That is, for $\polish=(\polish_v)_{v \in A}$, then $|z|=|A|$ is just the cardinality of the index set.
Second, we will sometimes use the symbol $\lan \cdot \ran$ to emphasize when we are working with an unordered vector; for a vector $\polish=(\polish_v)_{v \in A}$ which may be viewed as an element of the (ordered) Cartesian product $\Polish^A$, we write $\lan z \ran$ for the corresponding (unordered) element of $\SQ(\Polish)$, to avoid any ambiguity.

\subsection{The local-field equation on a regular tree} \label{se:regresults}

In this section we work with the infinite $\kappa$-regular tree $G=(V,E)$. 
For ease of notation in the regular tree setup, we denote the root $\myroot$ by $0$.
Write $N_v=N_v(G)$ for the neighborhood of a vertex $v$. We study the following process:
\begin{equation}
	\label{eq:regular_tree}
	X_v(k+1) = F^k(X_v[k],X_{N_v}[k],\xi_v(k+1)), \qquad k \in \N_0, \ \ v \in V. 
\end{equation}
We write $X_v[k]:=(X_v(i))_{i=0}^k$ for $k \in \N_0$ and $X_v:=(X_v(i))_{i=0}^\infty$ for the finite and infinite trajectories, which are random variables with values in $\PolishX^{k+1}$ and $\PolishX^\infty$, respectively.
We make the following assumptions:

\begin{Condition}
	\label{cond:regular_tree}
	\phantomsection
	{\ }
	\begin{enumerate}
	\item
		The graph $G=(V,E)$ is an infinite $\kappa$-regular tree, for $\kappa \ge 2$, meaning that it is a tree with a countably infinite vertex set such that every vertex $v$ has precisely $\kappa = |N_v|$ neighbors.
	\item
		$F^k \colon \PolishX^{k+1} \times (\PolishX^{k+1})^\kappa \times \Nspace \to \PolishX$ is measurable, and it is symmetric with respect to its $(\PolishX^{k+1})^\kappa$ argument. That is, $F^k(x,(y_i)_{i=1}^{\kappa},\xi) = F^k(x,(y_{\pi(i)})_{i=1}^{\kappa},\xi)$ for any permutation $\pi$ of $\{1,\ldots,\kappa\}$ and any $x,y_1,\ldots,y_\kappa \in \PolishX^{k+1}$ and $\xi \in \Nspace$.
	\item 
		The $\Nspace$-valued random variables $\{\xi_v(k) : v \in V, k \in \N\}$ are i.i.d.\ and independent of $(X_v(0))_{v \in V}$.
	\item 
		The collection $(X_v(0))_{v \in V}$ satisfies the following: 
		\begin{enumerate}
		\item 
			For any edge $(u,v)$ in $G$, if $T_1,T_2 \subset V$ denote the disjoint subtrees obtained by removing $u$ and $v$, then $X_{T_1}(0) \perp X_{T_2}(0) \,|\, X_{\{u,v\}}(0)$.
		\item 
			$(X_v(0))_{v \in V}$ is invariant under automorphisms of $G$. That is, if $\varphi : V \to V$ is a bijection such that $(u,v) \in E \iff (\varphi(u),\varphi(v)) \in E$, then $(X_{\varphi(v)}(0))_{v \in V} \stackrel{d}{=} (X_v(0))_{v \in V}$.
		\end{enumerate}
	\end{enumerate}
\end{Condition}

Unlike in \eqref{intro:mainsystem} we choose here to write $F^k$ as depending on the vector of neighbors $X_{N_v}[k]$ instead of the empirical measure. Even though we require  $F^k$ to depend symmetrically on $X_{N_v}[k]$, the vector formulation is more general as it allows for dependence on the numbers of neighbors in different states, and  is also more convenient for expressing continuity conditions that capture models that arise in practice.
Note that, of course, Condition \ref{cond:regular_tree}(4) holds if $(X_v(0))_{v \in V}$ are i.i.d.
The following construction describes the \local\ for the $\kappa$-regular tree:

\begin{Construction}[Local-field equation on a regular tree] 
	\label{constr:regulartree_local}
	Suppose Condition \ref{cond:regular_tree} holds.	
	Construct an $\PolishX^{\kappa+1}$-valued process $\Xbar$ as follows:
	\begin{enumerate}[(i)]
	\item 
		Initialize by setting $(\Xbar_v(0))_{v=0}^{\kappa} \stackrel{d}{=} (X_v(0))_{v =0}^{\kappa}$.
	\item 
		Proceeding recursively, for $k \ge 0$ we generate $\Xbar(k+1)$ from $\Xbar(k)$ as follows:
		\begin{itemize}
		\item 
			Define a kernel $\bar{\cmeas}[k]$ by
			\begin{equation*}
				\bar{\cmeas}[k](\cdot \,|\, x_0,x_1) := \Lmc\big((\Xbar_v[k])_{v=2}^{\kappa} \,|\, \Xbar_0[k] = x_0, \, \Xbar_1[k]=x_1 \big).
			\end{equation*}
			This is a $\Lmc(\Xbar_1[k],\Xbar_0[k])$-a.e.\ well-defined random measure on $(\PolishX^{k+1})^{\kappa-1}$.
		\item 
			Generate i.i.d.\ random variables $(\xibar_v(k+1))_{v=0}^{\kappa}$ with the same law as $\xi_v(k)$.
		\item 
			For each $v = 1,\dotsc,\kappa$, given $(\Xbar_u[k])_{u=0}^{\kappa}$, generate conditionally independent random variables $(Z^k_{vj})_{j=1}^{\kappa-1}$ with
			\begin{equation*}
				(\Zbar^k_{vj})_{j=1}^{\kappa-1} \sim \bar{\cmeas}[k](\cdot\, | \Xbar_v[k], \Xbar_0[k]).
			\end{equation*}
			In other words, for Borel sets $B_1,\ldots,B_\kappa \subset (\PolishX^{k+1})^{\kappa-1}$, we have
			\begin{equation}
				\PP\big((\Zbar^k_{vj})_{j=1}^{\kappa-1} \in B_v, \ v=1,\ldots,\kappa \,\big|\, (\Xbar_u[k])_{u=0}^{\kappa}\big) = \prod_{v=1}^{\kappa}\bar{\cmeas}[k](B_v \,|\, \Xbar_v[k], \Xbar_0[k]). \label{def:update-reg}
			\end{equation}
			Define also $\Zbar^k_{v0} := \Xbar_0[k]$.
		\item 
			Finally, set
			\begin{align*}
				\Xbar_0(k+1) & = F^k\left(\Xbar_0[k],(\Xbar_v[k])_{v=1}^{\kappa},\xibar_0(k+1)\right) \\
				\Xbar_v(k+1) & = F^k\left(\Xbar_v[k],(\Zbar^k_{vj})_{j=0}^{\kappa-1},\xibar_v(k+1)\right), \quad v=1,\dotsc,\kappa.
			\end{align*}
		\end{itemize}
	\end{enumerate}
\end{Construction}

The first main result is the following.  

\begin{Theorem}[Local-field equations characterize marginals on regular trees] 
	\label{thm:localequations-reg}	
	Suppose Condition \ref{cond:regular_tree} holds, and let $X$ be as in \eqref{eq:regular_tree}. Let $\Xbar$ be as in Construction \ref{constr:regulartree_local}. Then $(X_v)_{v =0}^{\kappa} \stackrel{d}{=} (\Xbar_v)_{v=0}^{\kappa}$.
\end{Theorem}

The proof is deferred to Section \ref{se:regtreeprf}. 
We emphasize again that, as was explained in Section \ref{se:Zlocdyn} for the $\kappa=2$-regular tree, the point of Theorem \ref{thm:localequations-reg} is to show that the $(\kappa+1)$-dimensional process given by Construction \ref{constr:regulartree_local} agrees in law with any neighborhood of particles in the infinite-dimensional process given by \eqref{eq:regular_tree}. 

\begin{Remark}
    The \local\ in Construction \ref{constr:regulartree_local} and Theorem \ref{thm:localequations-reg}, although involving the conditional distribution $\bar{\cmeas}[k]$, turn out to be quite useful and perform much better than a naive simulation of the system.
    We illustrate this point via two examples with Gaussian affine dynamics in Appendix \ref{sec:char_regular_tree}. 
    We also note that for 
    a certain class of dynamics, the \local\ in fact reduce to a tractable Markov process (e.g., see \cite{Wortsman18} for numerics and \cite{Cocomello-thesis,CocomelloRamanan2023exact} for a proof for continuous-time analogs). 
\end{Remark}

\begin{Remark}
    We now show how to write a recursive description of the joint probability of $(X_v[k])_{v=0}^\kappa$ using Theorem \ref{thm:localequations-reg}.
    For ease of writing, suppose $\PolishX$ is a countable state.
    Denote by $\Xbd:= (X_v)_{v=0}^\kappa$ and $\Xbd_{c(i)}:=(X_{ij})_{j=1}^{\kappa-1}$, $i \in \{1,\dotsc,\kappa\}$, for random variables.
    Similarly, denote by $\xbd:= (x_v)_{v=0}^\kappa$ and $\xbd_{c(i)}:=(x_{ij})_{j=1}^{\kappa-1}$, $i \in \{1,\dotsc,\kappa\}$, for deterministic values.
    Write
    $$\nu_k(\xbd[k]) := \Pmb(\Xbd[k]=\xbd[k]).$$
    As in Construction \ref{constr:regulartree_local}, denote the conditional probability
    \begin{align*}
        \cmeas_k((x_v[k])_{v=2}^\kappa \,|\, x_0[k], x_1[k]) & := \Pmb((X_v[k])_{v=2}^\kappa = (x_v[k])_{v=2}^\kappa \,|\, X_0[k]=x_0[k], X_1[k]=x_1[k]) \\
        & = \nu_k(\xbd[k]) / \Pmb(X_0[k]=x_0[k], X_1[k]=x_1[k]).
    \end{align*}    
    Also, define  the transition probability at time $k$ to be 
    $$p_k(y \,|\, x_0[k], (x_v[k])_{v=1}^\kappa) := \Pmb(y=F^k(x_0[k], (x_v[k])_{v=1}^\kappa,\xi(k+1))).$$ 
    Then from Theorem \ref{thm:localequations-reg} and Construction \ref{constr:regulartree_local} we can write
    \begin{align*}
        & \nu_{k+1}(\xbd[k+1]) \\
        & = \sum_{\xbd_{c(1)}[k],\dotsc,\xbd_{c(\kappa)}[k]} \Pmb(\Xbd[k+1]=\xbd[k+1], \Xbd_{c(v)}[k]=\xbd_{c(v)}[k], v\in\{1,\dotsc,\kappa\}) \\
        & = \sum_{\xbd_{c(1)}[k],\dotsc,\xbd_{c(\kappa)}[k]} \nu_k(\xbd[k]) \Pmb(\Xbd_{c(v)}[k]=\xbd_{c(v)}[k], v\in\{1,\dotsc,\kappa\} \,|\, \Xbd[k]=\xbd[k]) \\
        & \qquad \cdot \Pmb(\Xbd(k+1)=\xbd(k+1) \,|\, \Xbd[k]=\xbd[k], \Xbd_{c(v)}[k]=\xbd_{c(v)}[k], v\in\{1,\dotsc,\kappa\}) \\
        & = \sum_{\xbd_{c(1)}[k],\dotsc,\xbd_{c(\kappa)}[k]} \nu_k(\xbd[k]) \prod_{v=1}^\kappa \cmeas_k(\xbd_{c(v)}[k] \,|\, x_v[k],x_0[k]) \prod_{v=0}^\kappa p_k(x_v(k+1) \,|\, x_v[k], x_{N_v}[k]).
    \end{align*}
    If the transition kernel $F^k$ does not depend on the past trajectory, that is, $p_k(y \,|\, x_0[k], (x_v[k])_{v=1}^\kappa) = p_k(y \,|\, x_0(k), (x_v(k))_{v=1}^\kappa)$, then the last line can be simplified to conclude that  
    \begin{align*} 
        & \nu_{k+1}(\xbd[k+1]) \\
        & \quad =     
        \sum_{\xbd_{c(1)}(k),\dotsc,\xbd_{c(\kappa)}(k)} \nu_k(\xbd[k]) \prod_{v=1}^\kappa \cmeas_k(\xbd_{c(v)}(k) \,|\, x_v[k],x_0[k]) \prod_{v=0}^\kappa p_k(x_v(k+1) \,|\, x_v(k), x_{N_v}(k)). 
    \end{align*}
\end{Remark}

One might be tempted to compare the local-field equation of Theorem \ref{thm:localequations-reg} to the \emph{dynamic cavity method} introduced in \cite{KanoriaMontanari2011majority}. Both are essentially recursions based on the homogeneity of the regular tree,  but there are important distinctions. While the dynamic cavity method of 
\cite{KanoriaMontanari2011majority} was formulated for finite spaces on regular trees, the local-field equation allows for more general  Polish state spaces, can be defined on random GW trees (see Section \ref{se:GWresults} and Section \ref{se:UGWresults} below), and can accommodate more general, continuous-time dynamics \cite{LacRamWu-diffusion,CocomelloRamanan2023exact}.  
More importantly, while the dynamic cavity recursion 
\cite[Lemma 2.1]{KanoriaMontanari2011majority} 
provides a nice algorithmic method to obtain bounds for particle systems with finite spaces, it provides only an implicit description of the dynamics via a fixed point relation.   
In contrast, 
the local-field equation, in addition to providing an algorithm for simulating the marginals, also provides an explicit autonomous characterization of the marginal dynamics. Such a characterization has already been shown in other contexts to be useful for studying other properties of the system, including long-time behavior (see \cite{HuRam24,HuRam25} in the context of diffusions) and large deviations \cite{RamYas23}.

\subsection{Local-field equations on Galton-Watson Trees} \label{se:GWresults}

We now focus on a Galton-Watson tree $\tree \sim \mathrm{GW}(\rho_{\text{\myroot}},\rho)$ governed by two offspring distributions $\rho_{\myroot},\rho \in \P(\N_0)$. The root vertex has offspring distribution $\rho_{\myroot}$, and all subsequent generations have offspring distribution $\rho$.  In Section \ref{subs:ulam-harris-neveu}, we first 
introduce the Ulam-Harris-Neveu labeling of trees, 
which will allow us to adopt the convenient perspective of the tree  $\tree$ as a random subset of a larger deterministic set  $\V.$
Then in Section \ref{subs:GWresults} we introduce 
the dynamics and state the form of the \local. 

\subsubsection{Ulam-Harris-Neveu labeling of trees} \label{subs:ulam-harris-neveu}
We work with a standard labeling scheme for
trees known as the Ulam-Harris-Neveu labeling (see, e.g.,  \cite[Section VI.2]{harris-book} or \cite{neveu1986arbres}),
which we summarize here.
Define the vertex set 
\[
\V := \{\myroot \} \cup \bigcup_{k=1}^\infty\N^k, \qquad \V^0 := \V \setminus \{\myroot\}.
\]
For $u, v \in \V$, let $uv$ denote the concatenation, that is, if $u=(u_1,\ldots,u_k) \in \N^k$ and $v=(v_1,\ldots,v_j) \in \N^j$, then $uv = (u_1,\ldots,u_k,v_1,\ldots,v_j) \in \N^{k+j}$. The root $\myroot$ is the identity element, so $\myroot u = u \myroot = u$ for all $u \in \V$.  
For $v \in \V^0$, we write $\mom_v$ 
for the parent of $v$; precisely, $\mom_v$ is the unique element of $\V$ such that there exists
$k \in \N$ satisfying $v=\mom_{v}k$.  

There is a natural partial order on $\V$. We say $u \le v$ if there exists (a necessarily unique) $w \in \V$ such that $uw=v$, and say $u < v$ if the unique vertex $w$ is not $\myroot$. 
A tree is a subset $\tree \subset \V$ satisfying: 
\begin{enumerate}
\item $\myroot \in \tree$; 
\item If $v \in \tree$ and $u \le v$, then $u \in \tree$; 
\item For each $v \in \tree$ there exists an integer $c_v(\tree)$ such that, for $k \in \N$, we have $vk \in \tree$ if and only if $1 \le k \le c_v(\tree)$. 
\end{enumerate}
For a tree $\tree \subset \V$, we use the same symbol $\tree$ to refer to the corresponding graph, which has vertex set $\tree$ and edge set $\{(\mom_v,v) : v \in \tree \cap \V^0\}$.
For $\tree \subset \V$ and $v \in \V$, let $N_v(\tree)$ denote the set of neighbors of $v$ in $\tree$ if $v \in \tree$,  and set $N_v(\tree) = \emptyset$ if $v \notin \tree$. Formally, $N_{\myroot}(\tree) = \tree \cap \N$, and $N_v(\tree) = \{\mom_v\} \cap \{vk \in \tree : k \in \N\}$ for $v \in \tree \setminus \{\myroot\}$. By convention, set $N_v(\tree)=\emptyset$ for $v \notin \tree$.

It is convenient to also define 
$\V_n$ to be  the labels of the first $n$ generations: 
\begin{equation*}
	\V_n := \{\myroot \} \cup \bigcup_{k=1}^n\N^k.
\end{equation*}
Lastly, we will make use of the following notation. For $v \in \V^0$ we define $\V_{v+}$ and $\V_{v-}$ to be the disjoint sets of labels on opposite ``sides" of $(v,\mom_v)$, or more precisely, set 
\begin{equation}
	\label{def:V+V-}
	\V_{v+} := \{vu : u \in \V^0\} = \{u \in \V : u > v\}, \qquad
	\V_{v-} := \V \setminus (\V_{v+} \cup \{v,\mom_v\}). 
\end{equation}

\subsubsection{Dynamics and Results on Galton-Watson Trees}
\label{subs:GWresults} 

We now introduce the dynamics. 
We could of course first generate the tree $\tree$ and then construct the particle system $(X_v)_{v \in \tree}$, but it is convenient to work with a particle system $(X_v)_{v \in \V}$ defined on the entire set of labels. To facilitate this, we assume our state space $\PolishX$ contains an isolated point $\cem$,  which acts as a cemetery state. This point will be reserved exclusively for the states of vertices with labels $v \in \V \setminus \tree$. We also adopt the notation
\begin{equation*}
	1x := x, \qquad 0x := \cem, \quad \text{for } x \in \X.
\end{equation*}
With some abuse of notation, we also write $\cem$ in place of $(\cem,\ldots,\cem) \in \PolishX^k$ for any $k \in \N$. 
We can then define the particle system by
\begin{equation}
	X_v(k+1) = 1_{\{v \in \tree\}} F^k\left(X_v[k],X_{N_v(\tree)}[k],\xi_v(k+1)\right), \qquad k \in \N_0, \ \ v \in \V. \label{eq:GW}
\end{equation}
We again write $X_v[k]:=(X_v(i))_{i=0}^k$ and $X_v:=(X_v(i))_{i=0}^\infty$ for the finite and infinite trajectories.
It is often convenient to augment the state process by the random tree itself, by setting
\[
Y_v(k) = (1_{\{v \in \tree\}}, X_v(k)), \qquad \ \ Y_v[k] := (1_{\{v \in \tree\}}, X_v[k]).
\]
To be absolutely clear about the role of the cemetery state, we have $X_v(k+1)=\cem$ if $v \notin \tree$, and otherwise $X_v(k+1)=F^k(\dotsm)$ is updated according to $F^k$.
We make the following assumptions on the dynamics.

\begin{Condition}
\phantomsection
\label{cond:GW} {\ }
\begin{enumerate}
\item $\tree \sim \mathrm{GW}(\rho_{\myroot},\rho)$ and $X_\V(0)$ satisfy
\begin{equation} 
	\label{eq:GW_assumption_1}
	Y_{\V_{v+}}(0) \perp Y_{\V_{v-}}(0) \, | \, Y_{\{v,\mom_v\}}(0), \quad v \in \V^0,
\end{equation} 
with $\V_{v-}$ and $\V_{v+}$ as defined in \eqref{def:V+V-}, 
and the following invariance property holds: $\Lmc((Y_{vu}(0))_{u \in \V^0} \, | \, Y_v(0),Y_{\mom_v}(0))$ does not depend on the choice of $v \in \V^0$.
	\item The $\Nspace$-valued random variables $\{\xi_v(k) : v \in \V, k \in \N\}$ are i.i.d.
	\item $(\tree,X_{\V}(0))$ is independent of $\{\xi_v(k) : v \in \V, k \in \N\}$.
	\item $F^k \colon \PolishX^{k+1} \times \SQ(\PolishX^{k+1}) \times \Nspace \to \PolishX \setminus \{\cem\}$ is measurable for each $k \in \N_0$.
	\item For each $v \in \V$ we have almost surely $\{X_v(0) = \cem\} = \{v \notin \tree\}$.
	\end{enumerate} 	
\end{Condition}

\begin{Remark} \label{re:treemeasurable} 
Properties 4 and 5 of 
Condition \ref{cond:GW} ensure that in fact $\{X_v(k)=\cem\}=\{v\notin\tree\}$ almost surely for each $k \in \N$.
\end{Remark}

\begin{Example}
The most important part of Condition \ref{cond:GW} is (1), so we discuss a few special cases.
\begin{enumerate}[(a)]
\item Suppose $X_v(0)=1_{\{v \in \tree\}}x$ for all $v \in \V$, for some $x \in \PolishX$. Then Condition \ref{cond:GW}(1) becomes
\begin{equation*} 
\tree \cap \V_{v+} \perp \tree \cap \V_{v-}  \, | \, (1_{\{v \in \tree\}},1_{\{\mom_v \in \tree\}}), \quad v \in \V^0,
\end{equation*}
which is clearly true by the conditional independence structure of a Galton-Watson tree.
\item Suppose $\widetilde{X}_{\V}(0)$ is independent of $\tree$ and satisfies 
\begin{equation*}
\widetilde{X}_{\V_{v+}}(0) \perp \widetilde{X}_{\V_{v+}}(0) \, | \, \widetilde{X}_{\{v,\mom_v\}}(0), \quad v \in \V^0.
\end{equation*}
Let $X_v(0)=1_{\{v \in \tree\}}\widetilde{X}_v(0)$. Then \eqref{eq:GW_assumption_1} holds. In particular, Condition \ref{cond:GW}(1) holds if $(X_v(0))_{v \in \V}$ are i.i.d.\ and independent of $\tree$.
\item We can also achieve Condition \ref{cond:GW}(1) by jointly generating the initial states along with the tree, with the same conditional independence structure as the $\mathrm{GW}(\rho_{\myroot},\rho)$ tree itself. 
Let $\mu_{\myroot}^* \in \P(\X)$ and $\mu_{\myroot},\mu \in \P(\cup_{n \in \N_0} (\{n\} \times \PolishX^n))$, recalling that $\PolishX^0 :=\{\emptyset\}$. 
Suppose that $\mu_{\myroot}(\{n\}\times \PolishX^n)=\rho_{\myroot}(n)$ and $\mu(\{n\} \times \PolishX^n) = \rho(n)$ for each $n \in \N_0$.
First, we generate $X_{\myroot}(0) \sim \mu_{\myroot}^*$. 
Then, generate $(N,(Z_i)_{i=1}^N) \sim \mu_{\myroot}$, declare $i \in \tree$ for $i=1,\ldots,N$, and set $X_i(0)=Z_i$ for these $i$. 
Then, for each $i \in \tree$, generate independently $(N_i,(Z_{ij})_{j=1}^{N_i}) \sim \mu$, declare $ij \in \tree$ for $j=1,\ldots,N_i$, and set $X_{ij}(0)=Z_{ij}$. 
Continue in this fashion, using $\mu$ to generate jointly the number of offspring and the corresponding initial states.
\end{enumerate}
\end{Example}

The following construction describes the \local\ for the $\mathrm{GW}(\rho_{\myroot},\rho)$ tree:

\begin{Construction}[Local-field equations on GW trees] \label{constr:GWlocal}
Suppose Condition \ref{cond:GW} holds. We construct a random tree $\bar{\tree}$ and an $\PolishX^{\V_2}$-valued process $\Xbar_{\V_2}$ as follows.
\begin{enumerate}[(i)]
\item Initialize by letting $\bar{\tree}$ be a random tree and $\Xbar_{\V_2}(0)$ a $\PolishX^{\V_2}$-valued random variable with $(\bar{\tree},\Xbar_{\V_2}(0)) \stackrel{d}{=} (\tree \cap \V_2,X_{\V_2}(0))$.
\item Proceeding recursively, for $k \ge 0$ we generate $\Xbar_{\V_2}(k+1)$ from
$\Xbar_{\V_2}[k]$ and $\bar{\tree}$ as follows:
\begin{itemize}
\item Define a kernel $\bar{\cmeas}[k]$ by 
\[
\bar{\cmeas}[k](\cdot \,|\, x_1,x_{\myroot}) := \Lmc(\lan \Xbar_{N_1(\bar{\tree})}[k]\ran \,|\, \Xbar_1[k]=x_1,\Xbar_{\myroot}[k]=x_{\myroot}).
\]
This is a $\Lmc(\Xbar_1[k],\Xbar_{\myroot}[k])$-a.e.\ well-defined random measure on $\SQ(\PolishX^{k+1})$.
\item Generate i.i.d.\ random variables $(\bar{\xi}_{v}(k+1))_{v \in \V_2}$ with the same law as $\xi_v(k)$.
\item For each $v \in \bar{\tree} \setminus \V_1$ generate $\SQ(\PolishX^{k+1})$-valued random variables $\Zbar^k_v$  conditionally independent given $\Ybar_{\V_2}[k]$, with conditional distribution given by
\begin{equation*}
\Zbar^k_v \sim \bar{\cmeas}[k](\cdot \,|\, \Xbar_v[k], \Xbar_{\mom_v}[k]).
\end{equation*}
In other words, for Borel sets $B_v \subset \SQ(\PolishX^{k+1})$, $v \in \V_2 \setminus \V_1$, we have
\begin{equation}
\PP\Big(\Zbar^k_v \in B_v, \, v \in  \bar{\tree} \setminus \V_1 \,\Big|\, \bar{\tree}, \, \Xbar_{\V_2}[k]\Big) = \prod_{v \in  \bar{\tree} \setminus \V_1}\bar{\cmeas}[k](B_v \,|\, \Xbar_v[k], \Xbar_{\mom_v}[k]). \label{def:updaterule}
\end{equation}
\item Finally, set
\begin{align}
\begin{split}
\Xbar_v(k+1) &= 1_{\{v \in \bar{\tree}\}} F^k\big(\Xbar_v[k],\Xbar_{N_v(\bar{\tree})}[k],\xibar_v(k+1)\big), \quad v \in \V_1, \\
\Xbar_v(k+1) & = 1_{\{v \in \bar{\tree}\}} F^k\big(\Xbar_v[k],\Zbar^k_v,\xibar_v(k+1)\big), \qquad \qquad v \in  \V_2 \setminus \V_1.
\end{split} \label{eq:Xbar-GW}
\end{align}
\end{itemize}
\end{enumerate}
\end{Construction}

\begin{Remark}
Note that $\bar{\cmeas}[k]$ is defined as the conditional law of the $\SQ(\PolishX^{k+1})$-valued random variable $\lan \Xbar_{N_1(\bar{\tree})}[k]\ran$, which includes the variable $\Xbar_{\myroot}[k]$ on which we are conditioning since $\myroot \in N_1(\bar{\tree})$. Hence, the random (unordered) vector $\Zbar^k_v$ always includes $\Xbar_{\mom_v}[k]$ as one element. This redundancy only serves to streamline the presentation. One could instead define $\bar{\cmeas}$ to be the conditional law of $\lan \Xbar_{N_1(\bar{\tree}) \setminus \{\myroot\}}[k]\ran$, as long as one concatenates the resulting $\Zbar^k_v$ with $\Xbar_{\mom_v}[k]$.
\end{Remark}

The second main result of the paper is the following, which shows that the local-field equations in the above construction characterize the dynamics of the first two generations of the particle system set on the $\mathrm{GW}(\rho_{\myroot},\rho)$ tree.
Section \ref{sec:GW} is devoted to the proof.

\begin{Theorem}[Local-field equations characterize marginals on GW trees]  \label{thm:localequations-GW}
Suppose Condition \ref{cond:GW} holds. Let $\tree \sim \mathrm{GW}(\rho_{\myroot},\rho)$, and let $X$ be as in \eqref{eq:GW}. Let $(\bar{\tree},\Xbar)$ be as in Construction \ref{constr:GWlocal}. Then the $(\{0,1\} \times \PolishX^\infty)^{\V_2}$-valued random variables $(1_{\{v \in \tree\}},X_v)_{v \in \V_2}$ and $(1_{\{v \in \bar{\tree}\}},\Xbar_v)_{v \in \V_2}$ have the same distribution.
\end{Theorem}

\subsection{The local-field equation on a unimodular Galton-Watson tree} \label{se:UGWresults}

For our last variety of \local, we show how to refine Theorem \ref{thm:localequations-GW} when the tree is \emph{unimodular}.
For a $\mathrm{GW}(\rho_{\myroot},\rho)$ tree, this means that $\rho_{\myroot}$ and $\rho$ relate in the following manner:
\begin{equation}
	\rho(k) = \frac{(k+1)\rho_{\myroot}(k+1)}{\sum_{n\in\N}n\rho_{\myroot}(n)}, \quad k \in \N_0. \label{def:rho-unimod}
\end{equation}
Let us write $\mathrm{UGW}(\rho_{\myroot})$ for $\mathrm{GW}(\rho_{\myroot},\rho)$, where $\rho$ is given by \eqref{def:rho-unimod}, and we assume throughout that $\rho_{\myroot}$ has finite first moment so that this makes sense.

To define unimodularity properly, we need some terminology for spaces of marked graphs: A \emph{rooted graph} $(G,o)$ is a a connected locally finite graph (with finite or countable vertex set) together with a distinguished vertex $o$, and two rooted graphs are \emph{isomorphic} if there exists an isomorphism between the graphs which maps the root to the root. For a Polish space $\Polish$, a \emph{rooted $\Polish$-marked graph} is a triple $(G,\polish,o)$, where $(G,o)$ is a rooted graph and $\polish=(\polish_v)_{v\in V} \in \Polish^V$ are the \emph{marks}. Two rooted $\Polish$-marked graphs are \emph{isomorphic} if there exists an isomorphism between the rooted graphs which maps the marks $\polish^1$ to $\polish^2$.
Let $\G_*$ (resp.\ $\G_*[\Polish]$) denote the set of isomorphism classes of rooted graphs (resp.\ $\Polish$-marked rooted graphs). There is a natural (Polish) topology on $\Gmc_*[\Polish]$ corresponding to \emph{local convergence}: Let $B_r(G,o)$ denote the set of vertices within (graph distance) $r > 0$ around the root $o$. Fixing some metric on $\Polish$, the distance between $(G_1,\polish^1,o_1)$ and $(G_2,\polish^2,o_2)$ is defined as $1/(1+r)$, where $r$ is the supremum of those $r > 0$ such that there is an isomorphism $\varphi$ between $(B_r(G_1,o_1),o_1)$ and $(B_r(G_2,o_2),o_2)$ such that the distance between the marks $\polish^1_v$ and $\polish^2_{\varphi(v)}$ is at most $1/r$  for all $v \in B_r(G_1,o_1)$. See \cite{aldous-lyons} or the appendix of \cite{LacRamWu-convergence} for further details on this space.

Similarly, we define the space $\Gmc_{**}[\Polish]$ of isomorphism classes of \emph{doubly-rooted} $\Polish$-marked graphs $(G,\polish,o,\tilde o)$, where now $o,\tilde o$ are two (possibly equal) distinguished vertices of $G$, and ``isomorphism" must respect both roots. This space too admits a natural (Polish) local convergence topology, metrized as above but now with isomorphisms considered between $(B_r(G_i,o_i) \cup B_r(G_i,\tilde o_i),o_i,\tilde o_i)$, for $i=1,2$.

We may finally define unimodularity as follows. Again let $\Polish$ be a Polish space. We say that a $\G_*[\Polish]$-valued random variable $(G,Z,o)$ is \emph{unimodular} if, for any bounded nonnegative Borel-measurable function $F$ on $\G_{**}[\Polish]$, we have
\begin{equation}
\E\left[\sum_{\tilde o \in G}F(G,Z,o,\tilde o)\right] = \E\left[\sum_{\tilde o \in G}F(G,Z,\tilde o,o)\right] \label{def:unimodularity}
\end{equation}
Similarly, a $\G_*$-valued random variable $(G,o)$ is unimodular if the same is true but with $Z$ omitted (equivalently, specializing the above definition to the trivial one-point mark space).

It is will known that if $\tree \sim \mathrm{UGW}(\rho_{\myroot})$, then the random rooted tree $(\tree,\myroot)$ is unimodular. Clearly if we equip a unimodular random graph with i.i.d.\ marks then the resulting marked graph remains unimodular. We will show that the dynamics studied in this paper, such as  \eqref{eq:GW}, propagate unimodularity over time.

We make the following assumptions on the dynamics, once again using the Ulam-Harris-Neveu labeling introduced in Section \ref{subs:ulam-harris-neveu}.

\begin{Condition}
	\phantomsection
	\label{cond:UGW}
	{\ }
	\begin{enumerate}
	\item $\tree \sim \mathrm{UGW}(\rho_{\myroot})$, where $\rho_{\myroot}$ has finite first moment, and $X_\V(0)$ together satisfy the following, with $Y_v(0)=(1_{\{v \in \tree\}},X_v(0))$:
	\begin{enumerate}
	\item For each $v \in \V^0$, $Y_{\V_{v+}}(0) \perp Y_{\V_{v-}}(0) \,|\, Y_{\{v,\mom_v\}}(0)$.
	\item The conditional law $\Lmc((Y_{vu}(0))_{u \in \V^0} \,|\, (Y_v(0),Y_{\mom_v}(0)))$ does not depend on the choice of $v \in \V^0$.
	\item For each $n\in\N$ satisfying $\PP(|N_{\myroot}(\tree)|=n) > 0$, and each permutation $\pi$ of $\{1,\ldots,n\}$, the conditional law of $(Y_{\myroot}(0),(Y_{\pi(1)v}(0))_{v \in \V},\dotsc,(Y_{\pi(n)v}(0))_{v \in \V})$ given $|N_{\myroot}(\tree)|=n$ is the same as that of $(Y_{\myroot}(0),(Y_{1v}(0))_{v \in \V},\dotsc,(Y_{nv}(0))_{v \in \V})$.
	\item The $\G_*[\X]$-valued random variable $(\tree,X_\tree(0))$ is unimodular.
	\end{enumerate}
	\item The $\Nspace$-valued random variables $\{\xi_v(k) : v \in \V, k \in \N\}$ are i.i.d.
	\item $(\tree,X_\V(0))$ is independent of $\{\xi_v(k)\}_{v \in \V, k \in \N}$.
	\item $F^k \colon \PolishX^{k+1} \times \SQ(\PolishX^{k+1}) \times \Nspace \to \PolishX \setminus \{\cem\}$ is measurable for each $k \in \N_0$.
	\item For each $v \in \V$ we have almost surely $\{X_v(0) = \cem\} = \{v \notin \tree\}$.
	\end{enumerate} 	
\end{Condition}

Note that Condition \ref{cond:UGW} is stronger than Condition \ref{cond:GW}; in order to simplify the form of the \local, we need to assume not only that the tree is unimodular but also that the initial conditions satisfy additional symmetry properties, namely (1c) and (1d).
Note also that condition (5) implies that $\{v \in \tree\}$ is a.s.\ $X_v(0)$-measurable for each $v$, and so the conditions (1a--c) could be equivalently written with $X$ in place of $Y$; we prefer to write it with $Y$ to stress that these are really assumptions on the joint distribution of $(\tree,X_\tree(0))$, not just on $X_{\V}(0)$.

The following construction describes the \local\ for the $\mathrm{UGW}(\rho_{\myroot})$ tree.

\begin{Construction}[Local-field equations on UGW trees] \label{constr:UGWlocal}
Suppose Condition \ref{cond:UGW} holds. We construct a random tree $\bar{\tree}$, an $\PolishX^{\V_1}$-valued process $\Xbar_{\V_1}$, and an $\N_0$-valued process $\Nhat$ as follows.
\begin{enumerate}
                                                                                                                                             	\item Initialize by letting $\bar{\tree}$ be a random tree, $\Xbar_{\V_1}(0)$ an $\PolishX^{\V_1}$-valued random variable, and $\Nhat(0)$ an $\N_0$-valued random variable, with $(\bar{\tree},\Xbar_{\V_1}(0),\Nhat(0)) \sim (\tree \cap \V_1,X_{\V_1}(0),|N_1(\tree)|)$. Note that necessarily $\{\Nhat(0)=0\}=\{|N_{\myroot}(\bar{\tree})|=0\}$ a.s., since by convention $N_1(\tree)=\emptyset$ if $1 \notin \tree$.
	\item If $\bar{\tree} = \{\myroot\}$, then we simply set $\Xbar_{\myroot}(k+1)=F^k(\Xbar_{\myroot}[k],\emptyset,\xi_v(k+1))$, $\Nhat(k)=0$, and $\Xbar_v(k+1) = \cem$ for $v \in \N$ and $k \in \N_0$.
	\item If $\bar{\tree}\neq \{\myroot\}$, we proceeding recursively as follows. For $k \ge 0$, suppose we know  $\Xbar_{\V_1}[k]$, $\Nhat(k)$, and $\bar{\tree}$, with $\Nhat(k) \ge 1$ a.s. Generate $\Xbar_{\V_1}(k+1)$ and $\Nhat(k+1)$ as follows:
	\begin{itemize}
	\item Define a kernel $\bar{\cmeas}[k]$ by setting, for Borel sets $B \subset \SQ(\PolishX^{k+1})$,
	\begin{equation*}
	 \bar{\cmeas}[k](B \,|\, x_{\myroot},x_1) := \frac{\Emb\left[\frac{|N_{\myroot}(\bar{\tree})|}{\Nhat(k)}1_{\{\lan \Xbar_{N_{\myroot}(\bar{\tree})}[k] \ran \in B\}} \,|\, \Xbar_{\myroot}[k]=x_{\myroot}, \, \Xbar_1[k]=x_1\right]}{\Emb\left[\frac{|N_{\myroot}(\bar{\tree})|}{\Nhat(k)} \,|\, \Xbar_{\myroot}[k]=x_{\myroot},\,\Xbar_1[k]=x_1\right]},
	\end{equation*}
	with the convention that $0/0:=1$.
	\item Generate i.i.d.\ random variables $(\xibar_v(k+1))_{v \in \V_1}$ with the same law as $\xi_v(k)$.
	\item For each $v \in N_{\myroot}(\bar{\tree})$ generate $\SQ(\PolishX^{k+1})$-valued random variables $\Zbar^k_v$ which are conditionally independent, with conditional distribution given by
	\begin{equation*}
	\Zbar^k_v \sim \bar{\cmeas}[k](\cdot \,|\, X_v[k],X_{\myroot}[k]).
	\end{equation*}
	In other words, for Borel sets  $B_v \subset \SQ(\PolishX^{k+1})$, $v \in \N$, we have
	\begin{equation}
	\PP\big( \Zbar^k_v \in B_v, \, v \in N_{\myroot}(\tree) \,\big|\, \bar{\tree}, \Xbar_{\V_1}[k]\big) = \prod_{v \in N_{\myroot}(\tree)}\bar{\cmeas}[k](B_v \,|\, X_v[k],X_{\myroot}[k]). \label{def:ZbardistUGW}
	\end{equation}
	Set $\Nhat(k+1)=|\Zbar^k_1|$, where we recall that $|\cdot|$ denotes the length of a vector.
	\item Finally, set 
	\begin{align*}
		\Xbar_{\myroot}(k+1) &= F^k\big(\Xbar_{\myroot}[k],\Xbar_{N_{\myroot}(\bar{\tree})}[k],\xibar_{\myroot}(k+1)\big) \\
		\Xbar_v(k+1) &= 1_{\{v \in \bar{\tree}\}}F^k\big(\Xbar_v[k], \Zbar^k_v,\xibar_v(k+1)\big), \quad v \in \N.
	\end{align*}
	\end{itemize}
	\end{enumerate}
\end{Construction}

\begin{Remark}
To check that the recursion is well defined, note that for $\bar{\tree} \neq \{\myroot\}$ the random measure $\bar{\cmeas}[k](\cdot \,|\, \Xbar_{\myroot}[k],\Xbar_1[k])$ is a.s.\ supported on the set of vectors in $\SQ(\PolishX^{k+1})$ of nonzero length, and it follows that $\Nhat(k+1)=|\Zbar^k_1| \ge 1$ a.s.
\end{Remark}

The following is the third main result of the paper, and its proof is given in Section \ref{sec:pf-UGW}.

\begin{Theorem}[Reduced \local\ for UGW trees]
 \label{thm:localequations-UGW}
Suppose Condition \ref{cond:UGW} holds. Let $\tree \sim \mathrm{UGW}(\rho_{\myroot})$, and let $X$ be as in \eqref{eq:GW}, and also let $(\bar{\tree},\Xbar)$ be as in Construction \ref{constr:UGWlocal}. 
Then for each $k \in \N_0$, we have
\begin{equation*}
\big((1_{\{v \in \tree\}},X_v[k])_{v \in \V_1}, \, |N_1(\tree))|\big) \stackrel{d}{=} \big((1_{\{v \in \bar{\tree}\}},\Xbar_v[k])_{v \in \V_1}, \, \Nhat(k)\big).
\end{equation*}
In particular, the $(\{0,1\} \times \PolishX^\infty)^{\V_1}$-valued random variables $(1_{\{v \in \tree\}},X_v)_{v \in \V_1}$ and $(1_{\{v \in \bar{\tree}\}},\Xbar_v)_{v \in \V_1}$ have the same law.
\end{Theorem}

\subsection{Approximation of dynamics on finite graphs} \label{se:finapprox}

We briefly discuss some ramifications of the results on the \local\ presented above, for approximation of finite graphs. To state these results we set up some notation for particle systems defined on arbitrary locally finite graphs.

Suppose we are given a continuous transition function $F^k \colon \PolishX^{k+1} \times \SQ(\PolishX^{k+1}) \times \Nspace \rightarrow \PolishX$, as well as a noise distribution $\theta \in \P(\Nspace)$. Then for any locally finite graph $G=(V,E)$ with finite or countable vertex set (not necessarily connected), and for any initial distribution of $X_V(0)$, we let $\{\xi_v(k) : v \in V, k \in \N\}$ be i.i.d.\ and independent of $X_V(0)$, and we define
\begin{equation*}
X^G_v(k+1) = F^k(X^G_v[k],X^G_{N_v(G)}[k],\xi_v(k+1)), \quad v \in V, \ \ k \in \N_0.
\end{equation*}

Recall the notion of local convergence of graphs and the space $\Gmc_*$ introduced in Section \ref{se:UGWresults}.
It is shown in \cite[Section 3.1]{LacRamWu-convergence} that for any $\Gmc_*$-valued random variable $G$, the above particle system $(G,X^G)$ is a well-defined $\Gmc_*[\PolishX^\infty]$-valued random variable. 
Moreover, its distribution varies continuously with that of the $\Gmc_*[\X]$-valued random variable $(G,X^G(0))$ (see Theorem 3.2 in \cite{LacRamWu-convergence}). 
Hence, the $n\to\infty$ limiting behavior of $(G_n,X^{G_n})$ can be characterized in terms of the \local, for any sequence of random graphs $G_n$ converging locally to either a regular tree, a $\mathrm{GW}(\rho_{\myroot},\rho)$ tree, or a $\mathrm{UGW}(\rho_{\myroot})$ tree, as long as the initial conditions converge as well. 
Here we provide three noteworthy examples of random graph models and associated graph sequences.

In the following, if $G_n$ is a (random) finite graph which is not necessarily connected, let $\myroot_n$ denote a uniformly random vertex in $G_n$, and let $\mu^{G_n}$ 
 be the global empirical measure of the particle system on that graph, as defined in  \eqref{def-mugn}.

\begin{Theorem}[Characterization of limits of global empirical measures]
\label{thm:cvg-eg}
Let any of the following three cases hold for $\{G_n\}_{n \in \N}$:
\begin{enumerate}[(i)]
    \item Let $\kappa \ge 2$ be an integer. For each $n \ge \kappa+1$ with $n\kappa$ even, there exists at least one $\kappa$-regular graph on $n$ vertices, and we let $G_n$ be chosen uniformly at random.
    \item Let $\lambda > 0$, and let $G_n$ be the \Erdos graph $G(n,\lambda/n)$, on the vertex set $\{1,\ldots,n\}$.
    \item Let $G_n$ be drawn from the configuration model with vertex set $\{1,\ldots,n\}$ and with degree distribution converging weakly to $\rho_{\myroot}$.
    Assume $\rho_{\myroot}$ has finite first and second moments, that is, $\sum_{k=0}^\infty k^2 \rho_{\myroot}(k) < \infty$.
\end{enumerate}
Suppose the initial states $\{X^{G_n}_v(0) : v \in G_n, n \in \N\}$ are i.i.d.\ with law $\mu_0$.
Let $\Xbar$ and $G$ be given as follows, for each case above:
\begin{enumerate}[(i)]
    \item Let $\Xbar$ be as in Construction \ref{constr:regulartree_local}, with $(\Xbar_v(0))_{v=0}^{\kappa}$ i.i.d.\ with law $\mu_0$.
    \item Let $\Xbar$ be as in Construction \ref{constr:UGWlocal}, with $\rho_{\myroot}=\mathrm{Poisson}(\lambda)$ and with $(\Xbar_v(0))_{v\in\V_1}$ i.i.d.\ with law $\mu_0$.
    \item Let $\Xbar$ be as in Construction \ref{constr:UGWlocal}, with $(\Xbar_v(0))_{v\in\V_1}$ i.i.d.\ with law $\mu_0$.
\end{enumerate}
Then the $\PolishX^\infty$-valued random trajectory 
$X^{G_n}_{\myroot_n}$ converges in law  to $\Xbar_{\myroot}$ (or $\Xbar_0$ in case (i)). Moreover, 
the $\P(\PolishX^\infty)$-valued random measure  
$\mu^{G_n}$ converges in law  to $\Lmc(\Xbar_{\myroot})$ (or  $\Lmc(\Xbar_0)$ in case (i)).
\end{Theorem}
\begin{proof}
Case (i) follows on combining Theorem 3.7 of \cite{LacRamWu-convergence} with Theorem \ref{thm:localequations-reg}.
Cases (ii) and (iii) follow on combining Theorem 3.7 of \cite{LacRamWu-convergence} with Theorem \ref{thm:localequations-UGW}.
\end{proof}

\begin{Remark}
    To avoid giving a full definition of local weak convergence of (marked) graphs (which can be found in \cite{LacRamWu-convergence} in Section 2.2.4, including Definitions 2.8 and 2.10 therein, and Appendix A), we did not state the most general result possible here.
    In particular, the assumption of i.i.d.\ initial states in Theorem \ref{thm:cvg-eg} can be relaxed.
    Moreover, using \cite[Theorem 3.9]{LacRamWu-convergence} instead of \cite[Theorem 3.7]{LacRamWu-convergence}, one can obtain results analogous to Theorem \ref{thm:cvg-eg} for the local empirical measure of the connected component of $G_n$ containing a uniformly randomly chosen root $\myroot_n$ (and then discarding those vertices outside of that connected component).
\end{Remark}

\section{Interacting Markov chains on deterministic graphs}
\label{sec:MC_graph}

As already mentioned, a key ingredient in the derivation of the local-field equation is a conditional independence property of the particle trajectories. 
Specifically, we show in this section that 
 the trajectories of the particle system form a \textit{second-order Markov random fields}. 

Throughout the section fix a finite or countable connected locally finite graph $G=(V,E)$, a family of $\Nspace$-valued noises $\{\xi_v(k) : v \in V, k \in \N_0\}$, and an initial state configuration $\{X_v(0) : v \in V\}$. 
Recall that $N_v(G) = \{u \in V : (u,v) \in E\}$ denotes the neighborhood of a vertex $v$.
We make the following assumptions, which are notably   much more general than those of Section \ref{se:results} because here we work with a single fixed graph.

\begin{Condition}
	\phantomsection
	\label{cond:deterministic}
	{\ }
	\begin{enumerate}
	\item The  noises $\{\xi_v(k) : v \in V, k \in \N\}$ are  mutually independent and also independent of $X_V(0) \sim \mu_0$.
	\item For each $v \in V$ and $k \in \N_0$, $F_v^k \colon \PolishX^{k+1} \times (\PolishX^{k+1})^{N_v(G)} \times \Nspace \rightarrow \PolishX$ is measurable.
	\end{enumerate}	
\end{Condition}

Consider the interacting processes defined by
\begin{equation}
	X_v(k+1) = F_v^k(X_v[k],X_{N_v(G)}[k],\xi_v(k+1)), \qquad  v \in V, \ \ k \in \N_0, \label{eq:graph}
\end{equation}
which is more general than the dynamics in \eqref{eq:regular_tree} in that it allows for heterogeneous node-dependent transitions.
Recall from \eqref{def:boundaries} the notation $\partial A$ and $\partial^2A$ for the first and second boundaries of a set $A \subset V$.
As usual, we write $X_v[k]=(X_v(i))_{i=0}^k$ and $X_v=(X_v(i))_{i=0}^\infty$ for the finite and infinite trajectories.

\begin{Theorem}[Second-order Markov random field property]
	\label{thm:cond_independence}
	Suppose Condition \ref{cond:deterministic} holds, and let $A \subset V$. Suppose
	\begin{equation}
		\label{eq:cond-independence}
		{X_A}(0) \perp {X_{V \setminus (A\cup\partial^2A)}}(0) \mid {X_{\partial^2A}}(0). 
	\end{equation}
	Then ${X_A}[k] \perp {X_{V \setminus (A\cup\partial^2A)}}[k] \: | \: {X_{\partial^2A}}[k]$ for each $k \in \N_0$. 
\end{Theorem}

This result relies on the following two lemmas, which are also used in Section \ref{subs:props-part}, and whose proofs are deferred to Appendix \ref{se:appA}. 

\begin{Lemma}
	\label{lem:cond_indep_fact_1}
	Suppose $X$, $Y$, and $Z$ are random variables with values in some Polish spaces, and assume $X \perp Y \,\, | \,\, Z$.
	Then $(X,\phi(X,Z)) \perp (Y,\psi(Y,Z)) \mid Z$ for measurable functions $\phi$ and $\psi$.
\end{Lemma}

\begin{Lemma} \label{le:conditionalindependence1}
	Suppose $X$, $Y$, and $Z$ are random variables with values in some Polish spaces, and assume $X \perp Y \,\, | \,\, Z$.
	Then for measurable $\phi$ and $\psi$, we have $X \perp Y \,\, | \,\, (Z,\phi(X,Z),\psi(Y,Z))$. Moreover, we have
	\begin{align*}
		\Lmc(X \, | \, Z,\phi(X,Z),\psi(Y,Z)) &= \Lmc(X \, | \, Z,\phi(X,Z)), \\
		\Lmc(Y \, | \, Z,\phi(X,Z),\psi(Y,Z)) &= \Lmc(Y \, | \, Z,\psi(Y,Z)).
	\end{align*}
\end{Lemma}

We now apply these lemmas to prove Theorem \ref{thm:cond_independence}, 
and also to show that 
\begin{align}
	\Lmc&\left({X_A}[k],\xi_{A \cup \partial A}(k+1) \, | \, {X_{\partial^2A}}[k+1] = x_{\partial^2 A}[k+1]\right) \nonumber \\
	&= \Lmc\left({X_A}[k],\xi_{A \cup \partial A}(k+1) \, | \, X_{\partial A}(k+1) = x_{\partial A}(k+1),{X_{\partial^2 A}}[k] = x_{\partial^2 A}[k]\right), \label{pf:main1}
\end{align}
	for $\Lmc(X_{\partial^2A}[k+1])$-almost every $x_{\partial^2 A}[k+1] \in (\PolishX^{k+2})^{\partial^2 A}$, which will be used in 
	the proof of Theorem \ref{thm:cond_consistency}.

\begin{proof}[Proofs of Theorem \ref{thm:cond_independence} and relation \eqref{pf:main1}]
    Denote by $d(\cdot,\cdot)$ the graph distance on $G$.
	Let $C := V \setminus (A\cup\partial^2A) = \{v \in V : d(v,A) \ge 3\}$, and note that $\partial C = \{v \in V : d(v,A) = 2\} = \partial^2A \setminus \partial A$. 
	The claim is true for $k=0$ by \eqref{eq:cond-independence}, so we assume it holds for some $k \in \N_0$ and proceed by induction. 
	Since $\{\xi_v(k+1) : v \in V\}$ are independent of each other and of $\{X_v[k] : v \in V\}$, we have
	\begin{equation}
    	\label{eq-cond}
		({X_A}[k],\xi_{A \cup \partial A}(k+1)) \perp ({X_C}[k],\xi_{C \cup \partial C}(k+1)) \ | \ {X_{\partial^2A}}[k].
	\end{equation}
    Denoting $X := ({X_A}[k],\xi_{A \cup \partial A}(k+1))$, $Y := ({X_C}[k],\xi_{C \cup \partial C}(k+1))$ and $Z := X_{\partial^2A}[k]$. 
	 the form  \eqref{eq:graph} of the dynamics implies that  
     there exist measurable functions $\phi$ and $\psi$ such that $X_{\partial A}(k+1) = \phi(X,Z)$ and $X_{\partial C} (k+1)= \psi(X,Z)$. Since 
     $\partial^2 A = \partial A \cup \partial C$, 
     $X_{\partial^2 A}[k+1] = (Z, \phi(X,Z), \psi(Y,Z))$. Since \eqref{eq-cond} states that  $X \perp Y|Z$, the first assertion of Lemma \ref{le:conditionalindependence1} ensures that 
	\begin{align*}
	({X_A}[k],\xi_{A \cup \partial A}(k+1)) \perp ({X_C}[k],\xi_{C \cup \partial C}(k+1)) \: | \: {X_{\partial^2A}}[k+1]. 
	\end{align*}
    Also since $\partial^2 A = \partial A \cup \partial C$, relation \eqref{pf:main1} follows from the first display of Lemma \ref{le:conditionalindependence1}. 
    Continuing with the proof of Theorem \ref{thm:cond_independence}, note 
 that $X_{A}(k+1)$ is measurable with respect to $(X_A[k],\xi_A(k+1),X_{\partial A}[k])$ and hence, $(X,Z)$ and likewise,  $X_C(k+1)$ is measurable with respect to $(X_C[k],\xi_C(k+1),X_{\partial C}[k])$ and hence $(Y,Z).$
 The last display and Lemma \ref{lem:cond_indep_fact_1} 
  then imply that 
	\begin{equation*}
		{X_A}[k+1] \perp {X_{C}}[k+1] \: | \: {X_{\partial^2A}}[k+1].
	\end{equation*}
    Theorem \ref{thm:cond_independence} then follows by induction. 
\end{proof}

\begin{Corollary}
	\label{cor:cond_independence}
Under the assumptions of Theorem \ref{thm:cond_independence}, the same conclusion holds with $k = \infty$, that is,  ${X_A} \perp {X_{V \setminus (A\cup\partial^2A)}} \: | \: {X_{\partial^2A}}$.
\end{Corollary}

\begin{proof}
	Let $f,g,h$ be bounded continuous functions on $\PolishX^\infty$. Let $\Gmc_m := \sigma( X_{\partial^2A}[m])$ for each $m \in \N_0$, and $\Gmc_\infty := \sigma(\cup_m\Gmc_m) = \sigma(X_{\partial^2A})$.
	Then, using Theorem \ref{thm:cond_independence} followed by martingale convergence, we find 
	\begin{align*}
		\Emb & \left[ f(X_A) g(X_{\partial^2A}) h(X_{V \setminus (A\cup\partial^2A)}) \right] \\
		& = \lim_{k \to \infty} \lim_{m \to \infty} \Emb \left[ f(X_A[k]) g(X_{\partial^2A}[m]) h(X_{V \setminus (A\cup\partial^2A)}[k]) \right] \\
		& = \lim_{k \to \infty} \lim_{m \to \infty} \Emb \left[ \Emb \left[ f(X_A[k]) \mid \Gmc_m \right] g(X_{\partial^2A}[m]) \Emb \left[ h(X_{V \setminus (A\cup\partial^2A)}[k]) \mid \Gmc_m \right] \right] \\
		& = \lim_{k \to \infty} \Emb \left[ \Emb \left[ f(X_A[k]) \mid \Gmc_\infty \right] g(X_{\partial^2A}) \Emb \left[ h(X_{V \setminus (A\cup\partial^2A)}[k]) \mid \Gmc_\infty \right] \right] \\
		& = \Emb \left[ \Emb \left[ f(X_A) \mid \Gmc_\infty \right] g(X_{\partial^2A}) \Emb \left[ h(X_{V \setminus (A\cup\partial^2A)}) \mid \Gmc_\infty \right] \right].
	\end{align*}
\end{proof}

We say that $(X_v(0))_{v \in V}$ is a \emph{second-order Markov random field} if the conditional independence \eqref{eq:cond-independence} holds for every set $A \subset V$. The content of Theorem \ref{thm:cond_independence} and Corollary \ref{cor:cond_independence} is that (possibly non-Markovian) dynamics of the form \eqref{eq:graph} propagate the second-order Markov property on trajectories over time. 
Interestingly, for particle systems \eqref{eq:graph}  with a finite state space  $\PolishX$ 
and Markov dynamics that are reversible with respect to a stationary distribution, Dawson showed in \cite[Proposition 4.1]{dawson1973information} that  the stationary distribution  must be a second-order Markov random field. 
The paper \cite{goldstein1989pca} took an alternative perspective, studying for the lattice $G=\Z^d$ the problem of understanding the space-time random field $(X_v(k))_{(v,k) \in \Z^{d+1}}$ as a Gibbs measure on $\Z^{d+1}$.
See also \cite{LacRamWu-MRF,GangulyRamanan2022interacting} and references therein for developments analogous to Theorem \ref{thm:cond_independence} for systems of interacting diffusions and continuous-time pure jump processes.

\subsection{Proof of Theorem \ref{thm:localequations-reg}} \label{se:regtreeprf}

We will apply Theorem \ref{thm:cond_independence} to prove Theorem \ref{thm:localequations-reg}. Specifically, we will use sets of the form $A=T_1$, where $T_1$ and $T_2$ are disjoint subtrees separated by an edge $(u,v)$ as in Condition \ref{cond:regular_tree}(4a), so that $T_2=V\setminus(A\cup\partial^2A)$ and $\partial^2A=\{u,v\}$.

\begin{proof}[Proof of Theorem \ref{thm:localequations-reg}] 
At time $k=0$ we have $(\Xbar_v(0))_{v=0}^\kappa \stackrel{d}{=} (X_v(0))_{v=0}^\kappa$, by assumption. Let us assume now that $(\Xbar_v[k])_{v=0}^\kappa \stackrel{d}{=} (X_v[k])_{v=0}^\kappa$ for some $k \in \N_0$ and proceed by induction.

In line with the canonical tree labeling of Section \ref{subs:ulam-harris-neveu}, let us write $\{vj : j=1,\ldots,\kappa-1\}$ for the $\kappa-1$ neighbors of a vertex $v \in \{1,\ldots,\kappa\}$ other than the "root" neighbor $0$. Let $\Xbd[k] := (X_v[k])_{v=0}^\kappa$ and $\bar{\Xbd}[k] := (\Xbar_v[k])_{v=0}^\kappa$.
To show $\bar{\Xbd}[k+1] \stackrel{d}{=} \Xbd[k+1]$, from the evolution of $\Xbar_v$ and $X_v$ in Construction \ref{constr:regulartree_local}  and \eqref{eq:graph}, respectively, we see that it suffices to show that 
	\begin{align*}
		& \Lmc(\bar{\Xbd}[k], \Zbar^k_{ij} : v \in \{1,\dotsc,\kappa\}, j \in \{1,\dotsc,\kappa-1\}) \\
		& \quad = \Lmc(\Xbd[k], X_{ij}[k] : v \in \{1,\dotsc,\kappa\}, j \in \{1,\dotsc,\kappa-1\}), 
	\end{align*}
	which is equivalent to showing  that for bounded measurable functions $f_0 \colon (\PolishX^{k+1})^{\kappa+1} \to \R$ and $f_1,\ldots,f_\kappa \colon (\PolishX^{k+1})^{\kappa-1} \to \R$ we have
		\begin{equation*}
		\Emb \left[ f_0(\bar{\Xbd}[k]) \prod_{v=1}^\kappa f_v((\Zbar^k_{vj})_{j=1}^{\kappa-1}) \right] = \Emb \left[ f_0(\Xbd[k]) \prod_{v=1}^\kappa f_v((X_{vj}[k])_{j=1}^{\kappa-1}) \right].
	\end{equation*}
	To this end,  recalling that $(\Zbar^k_{vj})_{j=1}^{\kappa-1} \sim \bar{\cmeas}[k](\cdot\, | \Xbar_v[k], \Xbar_0[k])$, where 
\[ 
		\bar{\cmeas}[k](\cdot \,|\, x_0,x_1) := \Lmc\big((\Xbar_v[k])_{v=2}^{\kappa} \,|\, \Xbar_0[k] = x_0, \, \Xbar_1[k]=x_1 \big).
		\] 
    for $x_0, x_1 \in {\mathcal X}^{1+\kappa}$, define an analogous  kernel $\cmeas[k]$ by
	\begin{equation*}
	\cmeas[k](\cdot \,|\, x_0,x_1) := \Lmc\big((X_{1+j}[k])_{j=1}^{\kappa-1} \, | \, X_0[k]=x_0, \, X_1[k]=x_1 \big).
	\end{equation*}
	Due to the form of the dynamics  \eqref{eq:regular_tree} and the symmetry assumptions stated as Conditions \ref{cond:regular_tree}(2,3,4b), it is straightforward to check that the automorphism invariance of Condition \ref{cond:regular_tree}(4b) propagates through time; that is, $(X_{\varphi(v)}[k])_{v \in V} \stackrel{d}{=} (X_v[k])_{v \in V}$ for automorphisms $\varphi$ of $G$. It follows that for each $v=1,\ldots,\kappa$,
	\[
	\big((X_{vj})_{j=1}^{\kappa-1},X_v,X_0\big) \stackrel{d}{=} \big((X_{1+j})_{j=1}^{\kappa-1},X_0,X_1\big), 
	\]
which in turn implies 
	\begin{equation*}
	\cmeas[k](\cdot \,|\, X_v[k],X_0[k]) = \Lmc\big((X_{vj}[k])_{j=1}^{\kappa-1} \,|\, X_v[k],X_0[k]\big).
	\end{equation*}
	Next, use the conditional independence established in Theorem \ref{thm:cond_independence} to deduce that
	\begin{equation*}
	\cmeas[k](\cdot \,|\, X_v[k],X_0[k]) = \Lmc\big((X_{vj}[k])_{j=1}^{\kappa-1} \,|\, \Xbd[k]\big).
	\end{equation*}
	In particular, for Borel sets $B_1,\ldots,B_\kappa \subset (\PolishX^{k+1})^{\kappa-1}$, we have
\begin{equation}
	\PP\big((X_{vj}[k])_{j=1}^{\kappa-1} \in B_v, \ v=1,\ldots,\kappa \,\big|\, \Xbd[k]\big) = \prod_{v=1}^{\kappa}\cmeas[k](B_v \,|\, X_v[k], X_0[k]). \label{def:update-reg-2}
\end{equation}
Use \eqref{def:update-reg}, the induction hypothesis (which entails both $\Xbd[k] \stackrel{d}{=} \bar{\Xbd}[k]$ and $\cmeas[k] = \bar{\cmeas}[k]$),  then \eqref{def:update-reg-2} and finally the tower property to obtain  
	\begin{align*}
		 \Emb \left[ f_0(\bar{\Xbd}[k]) \prod_{v=1}^\kappa f_v((\Zbar^k_{vj})_{j=1}^{\kappa-1}) \right]  &= \Emb \left[ f_0(\bar{\Xbd}[k]) \Emb \left[  \prod_{v=1}^\kappa f_v((\Zbar^k_{vj})_{j=1}^{\kappa-1} \,\Big|\, \bar{\Xbd}[k] \right] \right] \\
		& = \Emb \left[ f_0(\bar{\Xbd}[k]) \prod_{v=1}^\kappa \int f_v \, d\bar{\cmeas}[k]( \cdot \,|\, \Xbar_v[k],\Xbar_0[k]) \right] \\
		& = \Emb \left[ f_0(\Xbd[k]) \prod_{v=1}^\kappa \int f_v \, d\cmeas[k]( \cdot \,|\, X_v[k],X_0[k])) \right] \\
		& = \Emb \left[ f_0(\Xbd[k]) \Emb \left[\prod_{v=1}^\kappa f_v((X_{vj}[k])_{j=1}^{\kappa-1}) \,\Big|\, \Xbd[k] \right] \right] \\
		& = \Emb \left[ f_0(\Xbd[k]) \prod_{v=1}^\kappa  f_v((X_{vj}[k])_{j=1}^{\kappa-1}) \right].
	\end{align*}
	This completes the proof by induction.
\end{proof}

\begin{Remark}
It is also possible (though much less elementary) to derive Theorem \ref{thm:localequations-reg} as a special case of Theorem \ref{thm:localequations-UGW}. Indeed, if $\rho_{\myroot}=\delta_{\kappa}$, then $\tree \sim \mathrm{UGW}(\rho_{\myroot})$ is precisely the $\kappa$-regular tree. One can check that Conditions \ref{cond:regular_tree} and \ref{cond:UGW} are equivalent in this case (up to extending the domain of $F^k$)  and also that Construction \ref{constr:UGWlocal} reduces to Construction \ref{constr:regulartree_local}. Indeed, the only point worth mentioning here is that  Condition \ref{cond:UGW}(1d) follows from Condition \ref{cond:regular_tree}(4b) because, for the non-random $\kappa$-regular tree $\tree=(V,E)$, we have that $(\tree,X_V(0))$ is unimodular if and only if $X_V(0)$ is invariant under automorphisms of $\tree$, by \cite[Theorem 3.2]{aldous-lyons}.
\end{Remark}

\subsection{Counterexamples} \label{se:counterexamples}

The conditional independence result in Theorem \ref{thm:cond_independence} is the best available in a sense.
To see this, we provide two simple counterexamples showing that in general conditional independence fails if one conditions only on
\begin{enumerate}
\item the current state of $\partial^2 A$ instead of the past of $\partial^2 A$ (Example \ref{eg:counter_current_of_two}), or
\item the past of $\partial A$ instead of $\partial^2 A$ (Example \ref{eg:counter_past_of_one}).
\end{enumerate}
In both examples, the graph is the line $V=\{1,\ldots,n\}$ with $E=\{(i,i+1) : i=1,\ldots,n-1\}$ for some $n \in \N$, and we check the conditional independence of Theorem \ref{thm:cond_independence} with the set $A=\{1\}$. We also work with state space $\X=\Rmb$ and standard Gaussian noises $\xi_v(k) \sim \mathcal{N}(0,1)$. We write $\Xbd(k)$ for the column vector $(X_1(k),\ldots,X_n(k))^\top$.

\begin{Example}
	\label{eg:counter_current_of_two}
	Consider $n=4$ in the above setup, with dynamics
	\begin{equation*}
		\Xbd(k+1) = B \Xbd(k) + \xibd(k), \quad \Xbd(0) = 0, \quad 
		B = \begin{pmatrix}
			1 & 1 & 0 & 0 \\
			1 & 1 & 1 & 0 \\
			0 & 1 & 1 & 1 \\
			0 & 0 & 1 & 1 \end{pmatrix}.
	\end{equation*}
	Then $\Xbd(2)=B\xibd(1)+\xibd(2)$ is centered Gaussian with covariance matrix
	\begin{equation*}
	B^2+I = \begin{pmatrix}
			3 & 2 & 1 & 0 \\
			2 & 4 & 2 & 1 \\
			1 & 2 & 4 & 2 \\
			0 & 1 & 2 & 3
	\end{pmatrix},
	\end{equation*}
	and we easily calculate $\mathrm{Cov}(X_1(2),X_4(2) \, | \,X_2(2),X_3(2)) = -1/2$.
Thus $X_1(2)$ is not independent of $X_4(2)$ given $(X_2(2),X_3(2))$.
	\qed
\end{Example}

\begin{Example}
	\label{eg:counter_past_of_one}
	Consider $n=3$ in the above setup, with dynamics
	\begin{equation*}
		\Xbd(k+1) = B \Xbd(k) + \xibd(k), \quad
		B = \begin{pmatrix}
			1 & 1 & 0  \\
			1 & 1 & 1  \\
			0 & 1 & 1  \end{pmatrix},
	\end{equation*}
	where $(X_1(0),X_2(0),X_3(0))$ are independent standard Gaussian.
	We find that $(X_1(1),X_2(1),X_3(1),X_2(0))$ is Gaussian with covariance matrix
	\begin{equation*}
	 \begin{pmatrix}
			3 & 2 & 1 & 1 \\
			2 & 4 & 2 & 1 \\
			1 & 2 & 3 & 1 \\
			1 & 1 & 1 & 1 \end{pmatrix},
	\end{equation*}
	and we easily calculate $\mathrm{Cov}(X_1(1),X_3(1) \,|\, X_2(1),X_2(0)) = -1/3$. 
	Thus $X_1(1)$ is not independent of $X_3(1)$ given $X_2[1]=(X_2(1),X_2(0))$.
	\qed
\end{Example}

\subsection{Gibbs-type uniqueness and a consistency property} \label{se:consist-Gibbs}

We close this section by presenting two additional results that are not used in the derivations of \local\ but that are of independent interest.
Throughout this section, we fix a graph $G$, along with   
noises and mappings $(F_v^k)_{v \in G, k \in \mathbb{N}_0}$ that satisfy Condition \ref{cond:deterministic},  and let $X^G$ satisfy  the dynamics \eqref{eq:graph}. 

We first establish a uniqueness property in the spirit of Gibbs measures. Namely, we show that the law $\mu$ of the system is fully characterized by its initial (time-zero) distribution 
and its ``specifications'' $\Lmc(X_A[k] \, | \, X_{\partial^2A}[k])$, for finite sets $A \subset V$.

\begin{Theorem}[Gibbs Uniqueness] \label{thm:Gibbs_uniqueness}
Given $X = X^G$ satisfying \eqref{eq:graph}, let $Y=(Y_v)_{v \in V}$ denote any other $(\PolishX^\infty)^V$-valued random variable, with $Y_v[k]=(Y_v(0),\ldots,Y_v(k))$ denoting the trajectories, $v \in V$, 
that satisfies the following three properties: 
\begin{enumerate}[(i)]
    \item $Y(0) \stackrel{d}{=} X(0)$;
    \item for each $k \in \N_0$ and each finite set $A \subset V$ the law of $Y_A[k]$ is mutually absolutely continuous with respect to that of $X_A[k]$;
    \item for each $k \in \N_0$, finite set $A \subset V$, and $\Lmc(X_{\partial A}[k])$-almost every $x_{\partial A}[k]$,
    \begin{equation} 
    \label{eq:Gibbs_measure_property}
    \Lmc\left(Y_A[k] \mid {Y_{\partial A}}[k] = x_{\partial A}[k] \right) = \Lmc\left({X_A}[k] \, | \, {X_{\partial A}}[k] = x_{\partial A}[k]\right).
    \end{equation}
\end{enumerate}
Then $Y \stackrel{d}{=} X$.
\end{Theorem}
\begin{proof}
If $G$ is finite, then taking $A=V$ (which has $\partial^2A=\emptyset$) in \eqref{eq:Gibbs_measure_property} proves the result. 

Now consider infinite $G$. 
We will prove by induction that $Y[k] \stackrel{d}{=} X[k]$ for each $k \in \N_0$.
By assumption (i), this holds for $k=0$.

Now suppose $Y[k] \stackrel{d}{=} X[k]$ for some $k \in \N_0$.
Fix $n \in \N$ and let $A_n := \{u \in V : d(u,\myroot) \le n\}$, where $d(\cdot,\cdot)$ denotes the graph metric on $G$.
In the following we will construct $W_{A_n}[k+1]$ such that 
\begin{equation}
    \label{eq:XYZ}
    Y_{A_n}[k+1] \stackrel{d}{=} W_{A_n}[k+1]  \stackrel{d}{=} X_{A_n}[k+1]. 
\end{equation}
For this, first let 
\begin{equation}
    \label{eq:YZ}
    W_{A_n}[k]:=Y_{A_n}[k], \quad W_{A_{n+1} \setminus A_n}[k+1]:=Y_{A_{n+1} \setminus A_n}[k+1], 
\end{equation}
and for $v \in A_n$, let $W_v(k+1)$ be defined in a manner similar to \eqref{eq:graph}, that is,
$$W_v(k+1) = F_v^k(W_v[k],W_{N_v(G)}[k],\eta_v(k+1)),$$
where $\Lmc(\eta_v(k+1)) = \Lmc(\xi_v(k+1))$ and the random variables $\{\eta_v(k+1) : v \in A_n\}$ are mutually independent and also independent of $W_{A_n}[k]$ and $W_{A_{n+1} \setminus A_n}[k+1]$.
Comparing this with \eqref{eq:graph}, we see that
\begin{align}
    & \Lmc(W_{A_n}(k+1) \,|\, W_{A_n}[k] = x_{A_n}[k], W_{A_{n+1} \setminus A_n}[k+1] = x_{A_{n+1} \setminus A_n}[k+1]) \notag \\
    & \qquad \qquad  = \Lmc(X_{A_n}(k+1) \,|\, X_{A_n}[k] = x_{A_n}[k], X_{A_{n+1} \setminus A_n}[k+1] = x_{A_{n+1} \setminus A_n}[k+1]) \label{eq:Z1}
\end{align}
and in fact they are also equal to
\begin{equation}
    \Lmc(W_{A_n}(k+1) \,|\, W_{A_{n+1}}[k] = x_{A_{n+1}}[k]) = \Lmc(X_{A_n}(k+1) \,|\, X_{A_{n+1}}[k] = x_{A_{n+1}}[k]) \label{eq:Z2}
\end{equation}
for $\Lmc(X_{A_{n+1}}[k+1])$-almost every $x_{A_{n+1}}[k+1]$.
Now fix  bounded measurable functions $f_1 \colon \PolishX^{A_n} \to \R$ and $f_2 \colon (\PolishX^{k+1})^{A_n} \times (\PolishX^{k+2})^{A_{n+1} \setminus A_n} \to \R$.
For $\Lmc(X_{A_{n+1}}[k+1])$-almost every $x_{A_{n+1}}[k+1]$, using \eqref{eq:Z1} we can write
\begin{align}
    & \ftil_1(x_{A_n}[k],x_{A_{n+1} \setminus A_n}[k+1]) \label{eq:ftil_1} \\
    & := \Emb[f_1(W_{A_n}(k+1)) \,|\, W_{A_n}[k] = x_{A_n}[k], W_{A_{n+1} \setminus A_n}[k+1] = x_{A_{n+1} \setminus A_n}[k+1]] \notag \\
    & = \Emb[f_1(X_{A_n}(k+1)) \,|\, X_{A_n}[k] = x_{A_n}[k], X_{A_{n+1} \setminus A_n}[k+1] = x_{A_{n+1} \setminus A_n}[k+1]], \notag
\end{align}
and therefore
\begin{align*}
    & \Emb[f_1(W_{A_n}(k+1)) f_2(W_{A_n}[k]) \,|\, W_{A_{n+1} \setminus A_n}[k+1] = x_{A_{n+1} \setminus A_n}[k+1]] \\
    & = \Emb [ \ftil_1(W_{A_n}[k],x_{A_{n+1} \setminus A_n}[k+1]) f_2(W_{A_n}[k])  \,|\, W_{A_{n+1} \setminus A_n}[k+1] = x_{A_{n+1} \setminus A_n} [k+1]] \\
    & = \Emb [ \ftil_1(Y_{A_n}[k],x_{A_{n+1} \setminus A_n}[k+1]) f_2(Y_{A_n}[k])  \,|\, Y_{A_{n+1} \setminus A_n}[k+1] = x_{A_{n+1} \setminus A_n} [k+1]] \\
    & = \Emb [ \ftil_1(X_{A_n}[k],x_{A_{n+1} \setminus A_n}[k+1]) f_2(X_{A_n}[k])  \,|\, X_{A_{n+1} \setminus A_n}[k+1] = x_{A_{n+1} \setminus A_n}[k+1] ] \\
    & = \Emb[f_1(X_{A_n}(k+1)) f_2(X_{A_n}[k]) \,|\, X_{A_{n+1} \setminus A_n}[k+1] = x_{A_{n+1} \setminus A_n}[k+1]] \\
    & = \Emb[f_1(Y_{A_n}(k+1)) f_2(Y_{A_n}[k]) \,|\, Y_{A_{n+1} \setminus A_n}[k+1] = x_{A_{n+1} \setminus A_n}[k+1]]
\end{align*}
where the first equality uses the tower property and \eqref{eq:ftil_1}, the second follows from \eqref{eq:YZ}, the third uses $A_{n+1} \setminus A_n = \partial A_n$ and \eqref{eq:Gibbs_measure_property} with $A=A_n$ at time $k+1$, 
the fourth uses again the tower property and \eqref{eq:ftil_1}, and the last uses again \eqref{eq:Gibbs_measure_property} with $A=A_n$ at time $k+1$.
Combining this with \eqref{eq:YZ}, we see that 
$$Y_{A_{n+1}}[k+1] \stackrel{d}{=} W_{A_{n+1}}[k+1].$$ 
Using the induction assumption $Y[k] \stackrel{d}{=} X[k]$, we have
$$X_{A_{n+1}}[k] \stackrel{d}{=} W_{A_{n+1}}[k].$$ 
Combining this with \eqref{eq:Z2}, we have
$$X_{A_n}[k+1] \stackrel{d}{=} W_{A_n}[k+1].$$ 
Therefore \eqref{eq:XYZ} holds.
Since $n \in \N$ is arbitrary, we have
$Y[k+1] \stackrel{d}{=} X[k+1]$
and this completes the proof.
\end{proof}

\begin{Remark}
Note that Theorem  \ref{thm:Gibbs_uniqueness} is really a distinct result from the Markov random field property of Theorem \ref{thm:cond_independence} despite some superficial 
similarity.  
In particular,  it is interesting and perhaps surprising to note that the above Gibbs-type uniqueness result does not have to assume any conditional independence property \`{a} la \eqref{eq:cond-independence}, or any Gibbs-type condition on the initial states.
    Moreover,  the specifications in  \eqref{eq:Gibbs_measure_property} involve conditioning on just a single boundary and  not on a double boundary (although the result would also hold for boundaries of any order $\partial^m$, as one only has to replace $A_{n+1}$ by $A_{n+m}$ in the proof).  
    In contrast, Theorem \ref{thm:cond_independence} and Example \ref{eg:counter_past_of_one} state that in general  $X$ is only a second-order Markov random field, and not a first-order one. 
\end{Remark}

Lastly we establish a certain consistency property.
The second-order Markov random field property proved in Theorem \ref{thm:cond_independence} shows that 
by conditioning on the dynamics on the double-boundary of a set $A$, the dynamics within $A$  is insensitive 
to the dynamics outside of $A \cup \partial^2 A$. 
Here we strengthen this property by showing that  
when conditioning on $X^G_{\partial^2 A}$ the 
dynamics inside $A$ is also insensitive in a suitable sense to the 
graph structure  outside  $A \cup \partial^2 A$. 
To make this precise, we introduce some notation.
Let $H=(V_H,E_H)$ be a subgraph of $G$, meaning $V_H \subset V$ and $E_H = \{(u,v) \in V_H^2 : (u,v) \in E\}$.
As usual, $N_v(H)$ and $N_v(G)$ denote the neighborhood sets relative to the respective graphs.

Let $F_v^{k,H} \colon \PolishX^{k+1} \times (\PolishX^{k+1})^{N_v(H)} \times \Nspace \rightarrow \PolishX$, $v \in H$, be a family of measurable functions satisfying the consistency condition
\begin{equation}
F^{k,H}_v \equiv F^k_v, \quad \text{ for } v \in V_H \text{ s.t. } N_v(G) \subset V_H. \label{def:F-consistency}
\end{equation}
This consistency condition demands that $F^{k,H}_v$ agrees with $F^k_v$ except possibly for $v$ lying on the ``edge'' of $H$ in the sense that one of its $G$-neighbors is outside of $H$.
Define a collection of interacting processes on $H$ as follows:
\begin{equation}
\label{XH}
	X^H_v(k+1) := F_v^{k,H}(X^H_v[k],X^H_{N_v(H)}[k],\xi_v(k+1)), \qquad k \in \N_0, \ \ v \in V_H.
\end{equation}

\begin{Proposition}[A consistency property] 
    \label{thm:cond_consistency}
    Suppose the functions $(F_v^{k,H})_{v \in H}$ satisfy the 
    the consistency condition 
    \eqref{def:F-consistency},  and let $X^H$ be the associated process defined in \eqref{XH}.  Fix   any set $A \subset V$ with $A \cup \partial^2 A \subset V_H$. If the initial conditions $X^G(0)$ satisfy 
    the conditional independence property \eqref{eq:cond-independence} for  $A$,  and also satisfy the identity 
    \begin{equation}
        \Lmc\left({X^H_A}(0) \, | \, X^H_{\partial^2A}(0) = x_{\partial^2 A}\right) = \Lmc\left({X^G_A}(0) \, | \, X^G_{\partial^2A}(0) = x_{\partial^2 A}\right) \label{consistency-time0}
    \end{equation}
    for $\Lmc(X^G_{\partial^2A}(0))$-almost every $x_{\partial^2 A}\in \PolishX^{\partial^2 A}$, 
    then it follows that 
    \begin{equation*}
        \Lmc\left({X^H_A}[k] \, | \, X^H_{\partial^2A}[k] = x_{\partial^2 A}[k]\right) = \Lmc\left({X^G_A}[k] \, | \, X^G_{\partial^2A}[k] = x_{\partial^2 A}[k]\right)
    \end{equation*}
    for each $k \in \N_0$ and $\Lmc(X^G_{\partial^2A}[k])$-almost every $x_{\partial^2 A}[k]\in (\PolishX^{k+1})^{\partial^2 A}$.
\end{Proposition}

The proof of Proposition \ref{thm:cond_consistency} is similar to, though more involved than, that of Proposition \ref{prop:properties-GW}, and is provided in Appendix \ref{sec:pf-cond_consistency} for completeness.

\section{Galton-Watson Trees}
\label{sec:GW}

This section contains the proofs of Theorem  \ref{thm:localequations-GW} and Theorem \ref{thm:localequations-UGW}. Throughout the section, $\tree \sim \mathrm{GW}(\rho_{\myroot},\rho)$, and $(X_v)_{v \in \V}$ is defined as in \eqref{eq:GW}. We also define as usual $Y_v[k] = (1_{\{v \in \tree\}},X_v[k])$ for $v \in \V$, $k \in \N_0$.

\subsection{Properties of the particle system}
\label{subs:props-part}

Along the way to proving Theorems \ref{thm:localequations-GW} and \ref{thm:localequations-UGW}, as before we derive two useful properties of the conditional distributions. 
Recall the sets $\V,\V^0,\V_{v+},\V_{v-}$ introduced in Section \ref{subs:ulam-harris-neveu}.

\begin{Proposition}[Insenstivity of certain conditional laws to the choice of base vertex]  \label{prop:properties-GW}
Suppose Condition \ref{cond:GW} holds.
Then the following hold for each $k \in \N_0$:
\begin{enumerate}[(i)]
\item 
For each  $v \in \V^0$, $Y_{\V_{v+}}[k]$ is conditionally independent of $Y_{\V_{v-}}[k]$ given $Y_{\{v,\mom_v\}}[k]$.
\item The conditional law of $(Y_{vu}[k])_{u \in \V^0}$ given $(Y_v[k],Y_{\mom_v}[k])$ does not depend on the choice of $v \in \V^0$. 
More precisely, there exists a measurable map $\Theta_k \colon (\{0,1\} \times \PolishX^{k+1})^2 \rightarrow \P((\{0,1\} \times \PolishX^{k+1})^{\V^0})$ such that, for every $v \in \V^0$ and every Borel set $B \subset (\{0,1\} \times \PolishX^{k+1})^{\V^0}$, we have 
\[
\Theta_k(Y_v[k], Y_{\mom_v}[k])(B) = \Pmb\left(  (Y_{vu}[k])_{u \in \V^0} \in B \, | \, Y_v[k], Y_{\mom_v}[k] \right) \ \ a.s.
\]
\end{enumerate}
\end{Proposition}

\begin{proof}
Both claims are true for $k=0$ by Condition \ref{cond:GW}(1). For $k \ge 1$ we proceed by induction.

(i) Assume the claim is true for some $k \in \N_0$. 
	Since $\{\xi_v(k+1) : v \in \V\}$ are independent, and since $\{\xi_v(k+1) : v \in \V\}$ are independent of $\{X_v[k] : v \in \V\}$ and $\tree$, we have
	\begin{equation*}
		({Y_{\V_{v+}}}[k],\xi_{{\V_{v+}} \cup \{v\}}(k+1)) \perp ({Y_{\V_{v-}}}[k],\xi_{{\V_{v-}} \cup \{\mom_v\}}(k+1)) \: | \: {Y_{\{v,\mom_v\}}}[k].
	\end{equation*}
	Next, notice that $Y_{v}(k+1)$ is measurable with respect to $({Y_{\V_{v+}}}[k],\xi_{v}(k+1),{Y_{\{v,\mom_v\}}}[k])$, and $Y_{\mom_v}(k+1)$ is measurable with respect to $({Y_{\V_{v-}}}[k],\xi_{\mom_v}(k+1),{Y_{\{v,\mom_v\}}}[k])$. Observing that  $Y_{{v,\pi_v}} (k+1) = (Y_{(v,\pi_v)}(k),Y_{v}(k+1), Y_{\pi_v}(k+1))$,  use Lemma \ref{le:conditionalindependence1} to deduce that 
	\begin{equation}
	({Y_{\V_{v+}}}[k],\xi_{{\V_{v+}} \cup \{v\}}(k+1)) \perp ({Y_{\V_{v-}}}[k],\xi_{{\V_{v-}} \cup \{\mom_v\}}(k+1)) \: | \: {Y_{\{v,\mom_v\}}}[k+1]. \label{pf:GW_main0}
	\end{equation}
	For use in part (ii), note that  Lemma \ref{le:conditionalindependence1} also implies 
	\begin{align}
		\Lmc&\left({Y_{\V_{v+}}}[k],\xi_{{\V_{v+}} \cup \{v\}}(k+1) \, | \, {Y_{\{v,\mom_v\}}}[k+1] = y_{\{v,\mom_v\}}[k+1]\right) \nonumber \\
		&= \Lmc\left({Y_{\V_{v+}}}[k],\xi_{{\V_{v+}} \cup \{v\}}(k+1) \, | \, Y_{v}(k+1) = y_{v}(k+1),{Y_{\{v,\mom_v\}}}[k] = y_{\{v,\mom_v\}}[k]\right), \label{pf:GW_main1}
	\end{align}
	for $\Lmc(Y)$-almost every $y_{\{v,\mom_v\}}[k+1] \in (\{0,1\}\times\PolishX^{k+2})^{|\{v,\mom_v\}|}$.  Finally, observing that $Y_{\V_{v+}}(k+1)$ is measurable with respect to $(Y_{\V_{v+}}[k],\xi_{\V_{v+}}(k+1),Y_{v}[k])$ and that $Y_{\V_{v-}}(k+1)$ is measurable with respect to $(Y_{\V_{v-}}[k],\xi_{\V_{v-}}(k+1),Y_{\mom_v}[k])$, 
	 Lemma~\ref{lem:cond_indep_fact_1} and \eqref{pf:GW_main0} together imply that 
	\begin{equation*}
		{Y_{\V_{v+}}}[k+1] \perp {Y_{\V_{v-}}}[k+1] \: | \: {Y_{\{v,\mom_v\}}}[k+1].
	\end{equation*}

(ii) Assume the claim is true for some $k \in \N_0$. 
	The idea of the induction is to build $$\Lmc((Y_{vu}[k+1])_{u \in \V^0} \, | \, Y_v[k+1] = y_v[k+1], Y_{\mom_v}[k+1] = y_{\mom_v}[k+1])$$ from $$\Lmc((Y_{vu}[k])_{u \in \V^0} \, | \, Y_v[k] = y_v[k], Y_{\mom_v}[k] = y_{\mom_v}[k])$$ in several steps and to check at each step that the operations involved are independent of the particular choice of $v \in \V^0$. 
	\begin{enumerate}[(a)]
	\item 
		For induction, we are assuming that the kernel
		\begin{equation}
		\Lmc((Y_{vu}[k])_{u \in \V^0} \, | \, Y_v[k] = y_v[k], Y_{\mom_v}[k] = y_{\mom_v}[k]) \label{pf:GW_conditionalspecification-step0}
		\end{equation}
		does not depend on $v \in \V^0$.
	\item 
		We claim that then 
		\begin{equation}
			\Lmc\left((Y_{vu}[k])_{u \in \V^0}, (\xi_{vu}(k+1))_{u \in \V} \, | \, Y_v[k] = y_v[k], Y_{\mom_v}[k] = y_{\mom_v}[k]\right) \label{pf:GW_conditionalspecification-step1}
		\end{equation}
		does not depend on $v \in \V^0$. 
		For this, since $(\xi_{vu}(k+1))_{u \in \V}$ and $((Y_{vu}[k])_{u \in \V^0},Y_v[k],Y_{\mom_v}[k])$ are independent, it follows that the measure \eqref{pf:GW_conditionalspecification-step1} is simply the product of the conditional measure in \eqref{pf:GW_conditionalspecification-step0} with the measure $\Lmc((\xi_{vu}(k+1))_{u \in \V})$.
	\item 
		We next claim that 
		\begin{equation}
			\Lmc\left((Y_{vu}[k])_{u \in \V^0}, (\xi_{vu}(k+1))_{u \in \V}, X_v(k+1) \, | \, Y_v[k] = y_v[k], Y_{\mom_v}[k] = y_{\mom_v}[k]\right) \label{pf:GW_conditionalspecification-step2}
		\end{equation}
		does not depend on $v \in \V^0$. This follows from the fact that there exists a measurable function $\Psi \colon (\{0,1\} \times \PolishX^{k+1})^\infty \times \PolishX \rightarrow \PolishX$, \emph{independent of the choice of $v \in \V^0$}, such that
		\begin{equation*}
			X_v(k+1) = \Psi(Y_{\mom_v}[k],Y_v[k],(Y_{vu}[k])_{u \in \V^0}, \xi_v(k+1)).
		\end{equation*}
		Using this map, we then obtain the measure \eqref{pf:GW_conditionalspecification-step2} as an image measure from \eqref{pf:GW_conditionalspecification-step1}.
	\item 
		We next claim that 
		\begin{equation}
			\Lmc\left(\left. (Y_{vu}[k])_{u \in \V^0}, (\xi_{vu}(k+1))_{u \in \V} \right| X_v(k+1)=x_v(k+1), Y_v[k] = y_v[k], Y_{\mom_v}[k] = y_{\mom_v}[k]\right) \label{pf:GW_conditionalspecification-step3}
		\end{equation}
		does not depend on $v \in \V^0$.
		This follows from step (c) simply by conditioning the (random) measure in \eqref{pf:GW_conditionalspecification-step2} on its third argument. 
	\item 
		By marginalizing the measure in \eqref{pf:GW_conditionalspecification-step3} and noting that $Y_v[k+1]=(Y_v[k],X_v(k+1))$, we find that
		\begin{equation}
			\Lmc\left(\left. (Y_{vu}[k], \xi_{vu}(k+1))_{u \in \V^0} \right| Y_v[k+1] = y_v[k+1], Y_{\mom_v}[k] = y_{\mom_v}[k]\right) \label{pf:GW_conditionalspecification-step4}
		\end{equation}
		does not depend on $v \in \V^0$.
	\item 
		We saw in \eqref{pf:GW_main1} that 
		\begin{equation}
			\Lmc\left(\left. (Y_{vu}[k], \xi_{vu}(k+1))_{u \in \V^0} \right|  Y_v[k+1] = y_v[k+1], Y_{\mom_v}[k+1] = y_{\mom_v}[k+1]\right) \label{pf:GW_conditionalspecification-step5}
		\end{equation}
		is equal to the kernel in \eqref{pf:GW_conditionalspecification-step4}, and thus is still independent of $v \in \V^0$.
	\item 
		Finally, we claim that 
		\begin{equation}
			\Lmc\left(\left. (Y_{vu}[k+1])_{u \in \V^0} \right|  Y_v[k+1] = y_v[k+1], Y_{\mom_v}[k+1] = y_{\mom_v}[k+1]\right) \label{pf:GW_conditionalspecification-step6}
		\end{equation}
		does not depend on $v \in \V^0$. This follows from the fact that there exists a measurable function $\Phi \colon (\{0,1\} \times \PolishX^{k+2})^\infty \times \{0,1\} \times \PolishX \rightarrow (\{0,1\} \times \PolishX^{k+2})^\infty$, independent of the choice of $v \in \V^0$, such that
		\begin{equation*}
			(Y_{vu}[k+1])_{u \in \V^0} = \Phi((Y_{vu}[k], \xi_{vu}(k+1))_{u \in \V^0},Y_v[k]).
		\end{equation*}
		Using this map, the measure in \eqref{pf:GW_conditionalspecification-step6} is then obtained as an image measure from \eqref{pf:GW_conditionalspecification-step5}.
	\end{enumerate}
\end{proof}

\subsection{Proof of Theorem \ref{thm:localequations-GW}}

We are now ready to prove Theorem \ref{thm:localequations-GW}, following an induction argument similar to the proof of Theorem \ref{thm:localequations-reg}. Recall that $Y_v[k] := (1_{\{v \in \tree\}}, X_v[k])$, and define $\Ybar_v[k] := (1_{\{v \in \bar{\tree}\}}, \Xbar_v[k])$ similarly.

\begin{proof}[Proof of Theorem \ref{thm:localequations-GW}]
Clearly the result holds for $k=0$ by construction, i.e., $\bar{Y}_{\V_2}(0) \stackrel{d}{=} Y_{\V_2}(0)$. So we assume that $\bar{Y}_{\V_2}[k] \stackrel{d}{=} Y_{\V_2}[k]$ holds for some $k \in \N_0$ and proceed by induction. Let us write $C_v = |\tree \cap \{vk : k \in \N\}|$ for the number of children of a vertex $v \in \V$ in the tree $\tree$.
	
Write $\Ybd = Y_{\V_2}$, $\bar{\Ybd} = \Ybar_{\V_2}$.
To show $\bar{\Ybd} \stackrel{d}{=} \Ybd$, from the evolution of $\Xbar_v(k+1)$ and $X_v(k+1)$ we see that it suffices to show
\begin{equation*}
\big(\bar{\Ybd}[k],  \Zbar^k_v\big)_{v \in \V_2 \setminus \V_1} \stackrel{d}{=} \big(\Ybd[k], \lan X_{N_v(\tree)}[k] \ran\big)_{v \in \V_2 \setminus \V_1},
\end{equation*}
where we define $\Zbar_v^k :=\emptyset \in \mathcal{S}^0(\PolishX^{k+1}) \subset \SQ(\PolishX^{k+1})$ for $v \in (\V_2 \setminus \tree ) \setminus \V_1$.
It then suffices to show
\begin{equation*}
\big(\bar{\Ybd}[k], \Zbar^k_v\big)_{v \in \Umb_n} \stackrel{d}{=} \big(\Ybd[k], \lan X_{N_v(\tree)}[k] \ran\big)_{v \in \Umb_n}.
\end{equation*}
for each fixed $n \in \N$, where $\Umb_n := \{1,\dotsc,n\}^2$, since $\cup_n \Umb_n =\V_2 \setminus \V_1$.
It then suffices to show that for each bounded and continuous functions $f_0,f_v$, $v \in \Umb_n$,
\begin{equation}
\Emb \left[ f_0(\bar{\Ybd}[k]) \prod_{v \in \Umb_n} f_v(\Zbar^k_v) \right] = \Emb \left[ f_0(\Ybd[k]) \prod_{v \in \Umb_n} f_v(\lan X_{N_v(\tree)}[k] \ran) \right]. \label{pf:GWlocal-1}
\end{equation}
Now define the kernel $\cmeas[k]$ by
\[
\cmeas[k](\cdot \,|\, x_1,x_{\myroot}) := \Lmc(\lan X_{N_1(\tree)}[k] \ran \,|\, X_1[k]=x_1,X_{\myroot[k]}=x_{\myroot}).
\]
Recalling the analogous definition of $\bar{\cmeas}[k]$ from part (3) of Construction \ref{constr:UGWlocal} of the local-field equation, and the induction hypothesis $\bar{Y}_{\V_2}[k] \stackrel{d}{=} Y_{\V_2}[k]$ it follows that $\bar{\cmeas}[k] = \cmeas[k]$.
Note that Proposition \ref{prop:properties-GW}(ii) and the fact that $Y_v[k]$ is $X_v[k]$-measurable imply  that, almost surely,
\begin{equation*}
\cmeas[k](\cdot \,|\, X_v[k],X_{\mom_v}[k]) = \Lmc(\lan X_{N_v(\tree)}[k] \ran \,|\, Y_v[k],Y_{\mom_v}[k]), \quad v \in \V^0.
\end{equation*}
Combining this with Proposition \ref{prop:properties-GW}(i), we have for Borel sets $B_v \subset \SQ(\PolishX^{k+1})$, $v \in \V_2 \setminus \V_1$,
\begin{equation}
\PP\Big(\lan X_{N_v(\tree)}[k] \ran \in B_v, \, v \in \V_2 \setminus \V_1 \,\big|\, Y_{\V_2}[k]\Big) = \prod_{v \in \V_2 \setminus \V_1}\cmeas[k](B_v \,|\, X_v[k], X_{\mom_v}[k]). \label{def:updaterule2}
\end{equation}
Note that $\tree \cap \V_2$ is $Y_{\V_2}[k]$-measurable, and each term in the product for $v \in (\V_2 \setminus \tree ) \setminus \V_1$ equals $1_{B_v}(\emptyset)$ (recalling our convention that $N_v(\tree)=\emptyset$ if $v \notin \tree$).
Similarly, noting that $\Ybar_v[k]$ is $\Xbar_v[k]$-measurable, we may deduce from the conditional joint distribution of the "phantom particles" of the local-field equation specified in \eqref{def:updaterule} and our above definition $\Zbar^k_v := \emptyset$ for $v \in (\V_2 \setminus \tree ) \setminus \V_1$ that \eqref{def:updaterule} can be written as
\begin{equation}
\PP\Big(\Zbar^k_v \in B_v, \, v \in \V_2 \setminus \V_1 \,\big|\, \Ybar_{\V_2}[k]\Big) = \prod_{v \in \V_2 \setminus \V_1}\bar{\cmeas}[k](B_v \,|\, \Xbar_v[k], \Xbar_{\mom_v}[k]). \label{def:updaterule3}
\end{equation}
Hence, we may use \eqref{def:updaterule3}, $\cmeas[k]=\bar{\cmeas}[k]$, $\bar{Y}_{\V_2}[k] \stackrel{d}{=} Y_{\V_2}[k]$, and then \eqref{def:updaterule2} to find 
\begin{align*}
\Emb \left[ f_0(\bar{\Ybd}[k]) \prod_{v \in \Umb_n} f_v(\Zbar^k_v) \right] 
		& = \Emb \left[ f_0(\bar{\Ybd}[k]) \Emb \left[ \prod_{v \in \Umb_n} f_v(\Zbar^k_v) \,\Big|\, \bar{\Ybd}[k] \right] \right] \\
		& = \Emb \left[ f_0(\bar{\Ybd}[k]) \prod_{v \in \Umb_n} \int f_v  \, d\bar{\cmeas}[k](\cdot \,|\, \Xbar_v[k],\Xbar_{\mom_v}[k]) \right] \\
		& = \Emb \left[ f_0(\Ybd[k]) \prod_{v \in \Umb_n} \int f_v  \, d\cmeas[k](\cdot \,|\, X_v[k],X_{\mom_v}[k]) \right] \\
		& = \Emb \left[ f_0(\Ybd[k]) \Emb \left[\prod_{v \in \Umb_n}  f_v( \lan X_{N_v(\tree)}[k] \ran) \,\Big|\, \Ybd[k] \right] \right] \\
		& = \Emb \left[ f_0(\Ybd[k]) \prod_{v \in \Umb_n} f_v( \lan X_{N_v(\tree)}[k] \ran) \right].
\end{align*}
This establishes \eqref{pf:GWlocal-1}, completing the proof.
\end{proof}

\subsection{Additional properties of the particle system}

In Proposition \ref{prop:properties-GW}, we showed that the dynamics \eqref{eq:GW} propagate Condition \ref{cond:GW}(1), or equivalently Conditions \ref{cond:UGW}(1a,1b), over time. Similarly, to prove Theorem \ref{thm:localequations-UGW}, we will need to show that Conditions \ref{cond:UGW}(1c,1d) are preserved by the dynamics \eqref{eq:GW}. This is the role of the next two lemmas.

\begin{Lemma}[Leaf exchangeability]  \label{lem:exch}
Suppose Condition \ref{cond:UGW} holds. Then, for each $n \in \N$ and each permutation $\pi$ of $\{1,\ldots,n\}$, the conditional law of $(Y_{\myroot}[k],(Y_{\pi(1)v}[k])_{v \in \V},\dotsc,(Y_{\pi(n)v}[k])_{v \in \V})$ given $|N_{\myroot}(\tree)|=n$ is the same as that of $(Y_{\myroot}[k],(Y_{1v}[k])_{v \in \V},\dotsc,Y_{nv}[k])_{v \in \V})$.
\end{Lemma}
\begin{proof}
Let $\pi$ be a $|N_{\myroot}(\tree)|$-measurable random bijection on $\N$ such that $\pi(i)=i$ for all $i > |N_{\myroot}(\tree)|$. We may lift $\pi$ to act on $\V$ by setting $\pi(\myroot)=\myroot$ and $\pi(iv)=\pi(i)v$ for all $i \in \Nmb$ and $v \in \V$. Let $\widetilde{X}_v(k)=X_{\pi(v)}(k)$ and $\widetilde{Y}_v(k) = Y_{\pi(v)}(k)$, for $v \in \V$ and $k \in \N_0$.
Condition \ref{cond:UGW}(1c) can clearly be recast as saying that $(\widetilde{Y}_v(0))_{v \in \V} \stackrel{d}{=} (Y_v(0))_{v \in \V}$. (Indeed, each has the same conditional distribution given $|N_{\myroot}(\tree)|$.) It is clear from the symmetry of $F^k$ with respect to the neighbors that $(\widetilde{X}_v(k))_{v \in \V}$ for $k \in \N_0$ obeys the same dynamics \eqref{cond:GW} as $(X_v(k))_{v \in \V}$. Since they start from the same initial distribution, it must therefore hold that $(\widetilde{Y}_v[k])_{v \in \V} \stackrel{d}{=} (Y_v[k])_{v \in \V}$ for all $k$. As $\pi$ was arbitrary, this is equivalent to the claim.
\end{proof}

\begin{Lemma}[Unimodularity] \label{lem:unimodular}
Suppose Condition \ref{cond:UGW} holds. Then the $\Gmc_*[\{0,1\}\times\PolishX^{k+1}]$-valued random variable $(\tree,(Y_v[k])_{v \in \tree})$ is unimodular for each $k \in \N_0$.
\end{Lemma}

\begin{proof}
By Condition \ref{cond:UGW}(1d), the claim is true for $k=0$. Now assume the claim is true for some $k \in \N_0$.
Let us write $\theta \in \P(\Nspace)$ for the common law of the i.i.d.\ $(\xi_v(k))_{v \in \V, k \in \N}$.
Fix a nonnegative bounded measurable function $F \colon \Gmc_{**}[\PolishX] \rightarrow \R$.
Note that $Y_v[k+1]=(Y_v[k],X_v(k+1)) = (1_{\{v \in \tree\}}, X_v[k+1])$.
Then, using the dynamics \eqref{eq:GW} and integrating out the independent noises $(\xi_v(k+1))_{v \in \tree}$, we have
\begin{align*}
\Emb & \left[ \sum_{o \in \tree} F(\tree,\{Y_v[k+1] : v \in \tree\},\myroot,o) \right] \\
& = \Emb \left[ \sum_{o \in \tree} F(\tree,\{Y_v[k], F^k(X_v[k], X_{N_v(\tree)}[k] , \xi_v(k+1)) : v \in \tree\},\myroot,o) \right] \\
& = \Emb \left[ \sum_{o \in \tree} \int_{\PolishX^\tree} F(\tree,\{Y_v[k], F^k(X_v[k], X_{N_v(\tree)}[k] , e_v) : v \in \tree\},\myroot,o) \, \prod_{v \in \tree} \theta(de_v) \right].
\end{align*}
Using the defining property \eqref{def:unimodularity} of unimodularity of $(\tree,(Y_v[k])_{v \in \tree})$, this becomes
\begin{align*}
\Emb & \left[ \sum_{o \in \tree} \int_{\PolishX^\tree} F(\tree,\{Y_v[k], F^k(X_v[k], X_{N_v(\tree)}[k] , e_v) : v \in \tree\},o,\myroot) \, \prod_{v \in \tree} \theta(de_v) \right] \\
& = \Emb \left[ \sum_{o \in \tree} F(\tree,\{Y_v[k], F^k(X_v[k], X_{N_v(\tree)}[k] , \xi_v(k+1)) : v \in \tree\},o,\myroot) \right] \\
& = \Emb \left[ \sum_{o \in \tree} F(\tree,\{Y_v[k+1] : v \in \tree\},o,\myroot) \right].
\end{align*}
Thus $(\tree,(Y_v[k+1])_{v \in \tree})$ is unimodular, and the proof is complete by induction.
\end{proof}

It is convenient now to adopt the notation $\lan \cdot \ran$ when we wish to stress that we are dealing with an \emph{unordered} sequence. That is, for $\polish_1,\ldots,\polish_n \in \Polish$ for some space $\Polish$, we may write $\lan (\polish_i)_{i=1}^n\ran$ or $\lan (\polish_1,\ldots,\polish_n)\ran$ to denote the element of $\SQ(\Polish)$ obtained by forgetting the order of the vector. Similarly, for $\polish_A=(\polish_v)_{v \in A} \in \Polish^A$, we write $\lan \polish_A\ran$ or $\lan (\polish_v)_{v \in A}\ran$ for the corresponding element of $\SQ(\Polish)$.

The last step before embarking on the main line of the proof of Theorem \ref{thm:localequations-UGW} is to combine the exchangeability and unimodularity properties from the previous two lemmas to derive a relationship between the conditional laws $\Lmc(\lan X_{N_{\myroot}(\tree)}[k] \ran \,|\, X_{\myroot}[k],X_1[k])$ and $\Lmc(\lan X_{N_1(\tree)}[k] \ran \,|\, X_1[k],X_{\myroot}[k])$. To understand properly the role of the indicator in statement, recall that $\{N_{\myroot}(\tree)=\emptyset\}=\{1 \notin \tree\}=\{X_1[k]=\cem\}$ is $X_1[k]$-measurable.

\begin{Proposition}[Transformation of conditional laws under rerooting] \label{prop:property-UGW}
Suppose Condition \ref{cond:UGW} holds.
Let $k \in \N_0$ and let $h \colon (\PolishX^{k+1})^2 \times \SQ(\PolishX^{k+1}) \to \Rmb$ be bounded and measurable.
Define $\Xi_k \colon (\{0,1\} \times \PolishX^{k+1})^2 \rightarrow \Rmb$ by 
\begin{equation*}
\Xi_k(X_{\myroot}[k], X_1[k]) := 1_{\{N_{\myroot}(\tree)\ne\emptyset\}} \frac{\Emb\left[\frac{|N_{\myroot}(\tree)|}{|N_1(\tree)|} h(X_{\myroot}[k],X_1[k],\lan X_{N_{\myroot}(\tree)}[k]\ran) \,\big|\, X_{\myroot}[k],X_1[k]\right]}{\Emb\left[\frac{|N_{\myroot}(\tree)|}{|N_1(\tree)|} \,\big|\, X_{\myroot}[k],X_1[k]\right]}.
\end{equation*}
Then, for each $v \in \N$, 
\begin{equation} \label{eq:property-UGW1}
\Xi_k(X_v[k], X_{\myroot}[k]) = \Emb\left[ h(X_v[k],X_{\myroot}[k],\lan X_{N_v(\tree)}[k]\ran) \,|\, X_{\myroot}[k],X_v[k] \right], \mbox{ on } \{v \in \tree\}.
\end{equation} 	
\end{Proposition}

The proof of Proposition \ref{prop:property-UGW} is preceded by another preparatory lemma, which is a fairly straightforward consequence of Proposition \ref{prop:properties-GW}(ii) and Lemma \ref{lem:exch}. 
Its proof, along with that of Proposition \ref{prop:property-UGW}, are essentially identical to the proofs of Lemma 7.1 and Proposition 3.18 in \cite{LacRamWu-diffusion}, but are less involved.
Therefore these are provided in Appendix \ref{sec:pf-property-UGW} for the sake of completeness and ease of exposition.

\subsection{Proof of Theorem \ref{thm:localequations-UGW}}
\label{sec:pf-UGW}

We are now ready to prove Theorem \ref{thm:localequations-UGW}, following an induction argument with the help of Proposition \ref{prop:property-UGW}. Recall that $Y_v[k] := (1_{\{v \in \tree\}}, X_v[k])$, and define $\Ybar_v[k] := (1_{\{v \in \bar{\tree}\}}, \Xbar_v[k])$ similarly.

\begin{proof}
Our goal is to prove
\begin{equation} \label{eq:localeqn-UGW-goal}
\Lmc(Y_{\V_1}[k],|N_1(\tree)|) = \Lmc(\Ybar_{\V_1}[k],\Nhat(k)), \quad k \in \N_0.
\end{equation}
This is true for $k=0$ by construction, so let us assume \eqref{eq:localeqn-UGW-goal} holds for some $k \in \N_0$ and proceed by induction.
Write $\Ybd = Y_{\V_1}$, $\bar{\Ybd} = \Ybar_{\V_1}$.
Fix a tree $t \subset \V_1$.
Recalling that $\bar{\tree} \stackrel{d}{=} \tree \cap \V_1$ by construction, it suffices to show
\begin{equation*}
\Lmc(\bar{\Ybd}[k+1],\Nhat(k+1) \,|\, \bar{\tree} = t) = \Lmc(\Ybd[k+1],|N_1(\tree)| \,|\, \tree \cap \V_1 = t).
\end{equation*}
From the evolution of $\Xbar_v(k+1)$ and $X_v(k+1)$ we see that it suffices to show 
\begin{equation*}
\Lmc( \bar{\Ybd}[k], (\Zbar^k_v)_{v \in t \cap \N}, \Nhat(k+1) \,|\, \bar{\tree}=t) 
	= \Lmc(\Ybd[k], (\lan X_{N_v(\tree)}[k] \ran)_{v \in t \cap \N}, |N_1(\tree)| \mid \tree \cap \V_1=t).
\end{equation*}
Note that $|N_1(\tree)| = |\lan X_{N_1(\tree)}\ran|$ by Condition \ref{cond:UGW}(5), and $\Nhat(k+1)=|\Zbar^k_1|$. Hence, it suffices to show
\begin{equation*}
	\Lmc\big(\bar{\Ybd}[k], (\Zbar^k_v)_{v \in t \cap \N} \,|\, \bar{\tree}=t\big) = \Lmc\big(\Ybd[k], (\lan X_{N_v(\tree)}[k] \ran)_{v \in t \cap \N} \,|\, \tree \cap \V_1=t\big).
\end{equation*}
It then suffices to show that for bounded measurable functions $f_{\myroot}$ on $(\PolishX^{k+1})^{\V_1}$ and $(f_v)_{v \in t \cap \N}$ on $\SQ(\PolishX^{k+1})$,
\begin{equation*}
	\Emb \Big[ f_{\myroot}(\bar{\Ybd}[k]) \prod_{v \in t} f_v(\Zbar^k_v) \,\Big|\, \bar{\tree}=t \Big] = \Emb \Big[ f_{\myroot}(\Ybd[k]) \prod_{v \in t} f_v(\lan X_{N_v(\tree)}[k] \ran) \,\Big|\, \tree \cap \V_1=t \Big]. 
\end{equation*}

Now, for bounded measurable functions $h$ on $\SQ(\PolishX^{k+1})$, we may use the definition of $\bar{\cmeas}[k]$ along with Proposition \ref{prop:property-UGW} to get, for $x_{\myroot},x_1 \in \PolishX^{k+1}$ with $x_{\myroot} \neq \cem$,
\begin{align*}
\int h \,d\bar{\cmeas}[k](\cdot \,|\, x_{\myroot},x_1) &= \frac{\Emb\left[\frac{|N_{\myroot}(\bar{\tree})|}{\Nhat(k)} h(\lan \Xbar_{N_{\myroot}(\tree)}\ran) \,\big|\, \Xbar_{\myroot}[k]=x_{\myroot},\Xbar_1[k]=x_1\right]}{\Emb\left[\frac{|N_{\myroot}(\bar{\tree})|}{\Nhat(k)} \,\big|\, \Xbar_{\myroot}[k]=x_{\myroot},\Xbar_1[k]=x_1\right]} \\
	&= \frac{\Emb\left[\frac{|N_{\myroot}(\tree)|}{|N_1(\tree)|} h(\lan X_{N_{\myroot}(\tree)}[k]\ran) \,\big|\, X_{\myroot}[k]=x_{\myroot},X_1[k]=x_1\right]}{\Emb\left[\frac{|N_{\myroot}(\tree)|}{|N_1(\tree)|} \,\big|\, X_{\myroot}[k]=x_{\myroot},X_1[k]=x_1\right]} \\
	&= \E\big[h(\lan X_{N_v(\tree)}[k]\ran) \,|\, X_{\myroot}[k]=x_1,X_v[k]=x_{\myroot}\big].
\end{align*}
In other words, it holds a.s.\ on the event $\{v \in \tree\}=\{X_v[k] \neq \cem\}$ that
\begin{equation}
	\int h \,d\bar{\cmeas}[k](\cdot \,|\, X_v[k],X_{\myroot}[k]) = \E\big[h(\lan X_{N_v(\tree)}[k]\ran) \,|\, X_{\myroot}[k],X_v[k]\big]. \label{eq:localeqn-UGW-2}
\end{equation}
Now let $f_{\myroot}$ and $(f_v)_{v \in t \cap \N}$ be as before, bounded measurable functions  on $(\PolishX^{k+1})^{\V_1}$ and  $\SQ(\PolishX^{k+1})$, respectively. Then, since $\{\bar{\tree}=t\}$ is $\bar{\Ybd}[k]$-measurable,
\begin{align*}
 \Emb &\left[ f_{\myroot}(\bar{\Ybd}[k]) \prod_{v \in t \cap \N} f_v( \Zbar^k_v) \,\Big|\, \bar{\tree}=t \right] \\
		& = \Emb \left[ f_{\myroot}(\bar{\Ybd}[k]) \Emb \left[ \prod_{v \in t \cap \N} f_v(\Zbar^k_v) \,\Big|\, \bar{\tree}=t, \bar{\Ybd}[k] \right] \,\Big|\, \bar{\tree}=t \right] \\
		& = \Emb \left[ f_{\myroot}(\bar{\Ybd}[k]) \prod_{v \in t \cap \N} \int \tilde{f}_v \,d\bar{\cmeas}[k]( \cdot \,|\, \Xbar_v[k],\Xbar_{\myroot}[k]) \,\Big|\, \bar{\tree}=t \right] \\
		& = \Emb \left[ f_{\myroot}(\Ybd[k]) \prod_{v \in t \cap \N} \int \tilde{f}_v \,d\bar{\cmeas}[k]( \cdot \,|\, X_v[k],X_{\myroot}[k]) \,\Big|\, \tree \cap \V_1 =t \right] \\
		& = \Emb \left[ f_{\myroot}(\Ybd[k]) \prod_{v \in t \cap \N} \Emb \left[ f_v(\lan X_{N_v(\tree)}[k] \ran) \mid X_v[k],X_{\myroot}[k] \right] \,\Big|\, \tree \cap \V_1 = t \right] \\
		& = \Emb \left[ f_{\myroot}(\Ybd[k]) \Emb \left[ \prod_{v \in t \cap \N} f_v(\lan X_{N_v(\tree)}[k] \ran) \,\Big|\, \tree \cap \V_1 = t, \Ybd[k] \right] \,\Big|\, \tree \cap \V_1 = t \right] \\
		& = \Emb \left[ f_{\myroot}(\Ybd[k]) \prod_{v \in t \cap \N} f_v(\lan X_{N_v(\tree)}[k] \ran) \,\Big|\, \tree \cap \V_1 = t \right].
	\end{align*}
	Here the second equality uses joint conditional distribution of $(\Zbar^k_{vu})_{u\in\N}$ given by \eqref{def:ZbardistUGW} and some trivial notational changes,
	the third follows from induction, 
	the fourth uses \eqref{eq:localeqn-UGW-2},
	and the last two use Proposition \ref{prop:properties-GW}(i) (and the fact that $Y_v[k]$ is $X_v[k]$-measurable).
	This gives \eqref{eq:localeqn-UGW-goal} for $k+1$ and completes the proof by induction.
\end{proof}

\appendix

\section{Proofs of lemmas on conditional independence}
\label{se:appA}

Here we provide the 
proofs of Lemma \ref{le:conditionalindependence1}. 
and Lemma \ref{lem:cond_indep_fact_1}. 
Recall that Lemma \ref{lem:cond_indep_fact_1} states that if 
$X \perp Y \mid Z$, then 
\[ (X,\phi(X,Z)) \perp (Y,\psi(Y,Z)) \mid Z \]
for any measurable functions $\phi$ and $\psi$. 

\begin{proof}[Proof of Lemma \ref{lem:cond_indep_fact_1}]
	To show the last display it suffices to prove that for every bounded and measurable functions $f,g,h$, one has
	\begin{equation}
		\label{eq:cond_indep_fact_1}
		\Emb [f(X,\phi(X,Z)) g(Y,\psi(Y,Z)) h(Z)] = \Emb[ \Emb[ f(X,\phi(X,Z)) \mid Z] g(Y,\psi(Y,Z)) h(Z)].
	\end{equation}
	In turn, by standard approximation arguments, it  suffices to prove \eqref{eq:cond_indep_fact_1} when $f(X,\phi(X,Z)) = \ftil_1(X) \ftil_2(Z)$ for some bounded  measurable functions $\ftil_1,\ftil_2$.
	To this end, note that since $X \perp Y \mid Z$, we have 
	\begin{align*}
		\mbox{ RHS of } \eqref{eq:cond_indep_fact_1} & = \Emb[ \Emb[ \ftil_1(X) \ftil_2(Z) \mid Z] g(Y,\psi(Y,Z)) h(Z)] \\
		& = \Emb[ \Emb[ \ftil_1(X) \mid Z] \ftil_2(Z) g(Y,\psi(Y,Z)) h(Z)] \\
		& = \Emb[ \Emb[ \ftil_1(X) \mid Y,Z] \ftil_2(Z) g(Y,\psi(Y,Z)) h(Z)] \\
		& = \Emb[ \ftil_1(X) \ftil_2(Z) g(Y,\psi(Y,Z)) h(Z)] = \mbox{ LHS of \eqref{eq:cond_indep_fact_1}}.
	\end{align*}
	This proves the result. 
\end{proof}

We now establish an intermediate result in preparation of the proof of Lemma \ref{le:conditionalindependence1}. 

\begin{Lemma}
	\label{lem:cond_indep_fact_2}
	If $X \perp Y \mid Z$, then $X \perp Y \mid (Z,\psi(Y,Z))$ for every measurable function $\psi$.
	Furthermore, $X \perp Y \mid (Z,\psi(Y,Z),\phi(X,Z))$ for any measurable functions $\phi$ and $\psi$.
\end{Lemma}

\begin{proof}
	 For any bounded  measurable functions $f,g_1,g_2,g_3$, using first $X \perp Y \mid Z$,  then the  tower property,  we have 
	\begin{align*}
		\Ebf[ \Ebf[ f(X) \mid Z] g_1(Y) g_2(Z) g_3(\psi(Y,Z))] & = \Ebf [ \Ebf[f(X) \mid Y,Z] g_1(Y) g_2(Z) g_3(\psi(Y,Z))] \\
		& = \Ebf [f(X) g_1(Y) g_2(Z) g_3(\psi(Y,Z))] \\ 
        &= \Ebf [\Ebf [f(X) \mid Y, Z, \phi(Y,Z)] g_1(Y) g_2(Z) g_3(\psi(Y,Z)].
	\end{align*}
	By standard approximation arguments, it follows that  $\Ebf[f(X) \mid Z] = \Ebf[f(X) \mid Y,Z,\psi(Y,Z)]$ a.s. for each bounded measurable function $f$.
	By the tower property, the last equality also implies that $\Ebf[f(X) \mid Z] = \Ebf[f(X) \mid Y,Z,\psi(Y,Z)] = \Ebf[f(X) \mid Z,\psi(Y,Z)]$ a.s. for each bounded measurable function $f$.
	Therefore given any  bounded  measurable functions $f,g,h$, we have
	\begin{align*}
		\Ebf[ \Ebf[ f(X) \mid Z,\psi(Y,Z)] g(Y) h(Z,\psi(Y,Z))] & = \Ebf[ \Ebf[ f(X) \mid Y,Z,\psi(Y,Z)] g(Y) h(Z,\psi(Y,Z))] \\
		& = \Ebf [f(X) g(Y) h(Z,\psi(Y,Z))].
	\end{align*}
	This proves the first assertion of the lemma. 
    
	The second assertion follows from the first on  observing that $(Z,\phi(X,Z),\psi(Y,Z)) = (\Ztil,\phitil(X,\Ztil))$, where $\Ztil = (Z,\psi(Y,Z))$ for some measurable function $\phitil$.
\end{proof}

We can now wrap up the proof of Lemma \ref{le:conditionalindependence1}. 

\begin{proof}[Proof of Lemma \ref{le:conditionalindependence1}] 
	The fact that 
    $X \perp Y|(Z, \psi(Y,Z), \phi(X,Z))$ 
    follows immediately 
    from 
Lemma \ref{lem:cond_indep_fact_2}. 
Setting $\psi$ to be a constant and $\tilde{Z} := (Z, \phi(X,Z))$ this in particular implies that 
$X \perp Y |\tilde{Z}$. 
Lemma \ref{lem:cond_indep_fact_1} 
then implies that for any measurable $\tilde{\psi},$ 
$X \perp (Y, \tilde{\psi}(Y, \tilde{Z})| \tilde{Z}.$
Since for any measurable $\psi$, there exists 
another measurable $\tilde{\psi}$  such that  
$\psi(Y,Z) = \tilde{\psi}(Y, \tilde{Z})$, it follows that $X \perp \psi (Y,Z)| (Z, \phi(X,Z)),$ which is  equivalent to the first relation in  the display in Lemma \ref{le:conditionalindependence1}. 
The second relation in the display follows in an exactly analogous fashion. 
\end{proof}

\section{Proof of Proposition \ref{thm:cond_consistency}}
\label{sec:pf-cond_consistency}

In this section we prove Proposition  \ref{thm:cond_consistency} using an inductive argument,  with the base case covered by \eqref{consistency-time0} of Proposition \ref{thm:cond_consistency}. 
The idea is to build $\Lmc(X^H_A[k+1] \, | \, X^H_{\partial^2A}[k+1] = x_{\partial^2 A}[k+1])$ from $\Lmc(X^H_A[k] \, | \, X^H_{\partial^2A}[k] = x_{\partial^2 A}[k])$ in several steps and to check at each step that the operations involved are independent of the particular choice of subgraph $H$. In the following, when we write ``does not depend on $H$", we mean subject to the restriction that $A \cup \partial^2A \subset V_H \subset V$ and \eqref{def:F-consistency} holds.
	\begin{enumerate}[(a)]
	\item 
		As our induction hypothesis, we  assume that the kernel
		\begin{equation}
		\Lmc\left(X^H_A[k] \, | \, X^H_{\partial^2A}[k] = x_{\partial^2 A}[k]\right) \label{pf:conditionalspecification-step0}
		\end{equation}
		does not depend on $H$.
	\item 
		We claim that then 
		\begin{equation}
			\Lmc\left(X^H_A[k],\xi_{A \cup \partial A}(k+1) \, | \, X^H_{\partial^2A}[k] = x_{\partial^2 A}[k]\right) \label{pf:conditionalspecification-step1}
		\end{equation}
		does not depend on $H$. 
		This holds because $\xi_{A \cup \partial A}(k+1)$ and $({X^H_A}[k],{X^H_{\partial^2A}}[k])$ are independent, and thus   \eqref{pf:conditionalspecification-step1} is simply the product of the measure \eqref{pf:conditionalspecification-step0} with $\Lmc(\xi_{A \cup \partial A}(k+1))$.
	\item 
		We next claim that 
		\begin{equation}
			\Lmc\left( {X^H_A}[k],\xi_{A \cup \partial A}(k+1),X^H_{\partial A}(k+1) \,|\, {X^H_{\partial^2A}}[k] = x_{\partial^2 A}[k]\right) \label{pf:conditionalspecification-step2}
		\end{equation}
		does not depend on $H$. This follows from the fact that there exists a measurable function $\Psi \colon (\PolishX^{k+1})^{\partial^2A} \times (\PolishX^{k+1})^{A} \times \Nspace^{\partial A} \rightarrow \PolishX^{\partial A}$, independent of the choice of $H$, such that
		\begin{equation*}
			X^H_{\partial A}(k+1) = \Psi(X^H_{\partial^2 A}[k],X^H_{A}[k],\xi_{\partial A}(k+1)).
		\end{equation*}
		Using this map, we then obtain the measure \eqref{pf:conditionalspecification-step2} as an image measure from \eqref{pf:conditionalspecification-step1}.
	\item By conditioning the (random joint) measure of \eqref{pf:conditionalspecification-step2}, which was shown to be insensitive to $H$, on $X^H_{\partial A}(k+1)$, we find that
		\begin{equation}
			\Lmc\left( {X^H_A}[k],\xi_{A \cup \partial A}(k+1) \,|\, X^H_{\partial A}(k+1) = x_{\partial A}(k+1), {X^H_{\partial^2A}}[k] = x_{\partial^2 A}[k]\right) \label{pf:conditionalspecification-step3}
		\end{equation}
		does not depend on $H$.
	\item 
		Marginalizing the measure in \eqref{pf:conditionalspecification-step3}, we find that the following law does not depend on $H$: 
		\begin{equation}
			\Lmc\left( {X^H_A}[k],\xi_A(k+1) \,|\, X^H_{\partial A}(k+1) = x_{\partial A}(k+1), {X^H_{\partial^2A}}[k] = x_{\partial^2 A}[k]\right). \label{pf:conditionalspecification-step4}
		\end{equation} 
	\item It follows from \eqref{pf:main1}  that the conditional law 
		\begin{equation}
			\Lmc\left( {X^H_A}[k],\xi_A(k+1) \,|\,  {X^H_{\partial^2A}}[k+1] = x_{\partial^2 A}[k+1]\right) \label{pf:conditionalspecification-step5}
		\end{equation}
		is equal to the kernel in \eqref{pf:conditionalspecification-step4}, and thus is still insensitive to the choice of $H$.
	\item Finally, we claim that the conditional law 
		\begin{equation}
			\Lmc\left( {X^H_A}[k+1] \,|\,  {X^H_{\partial^2A}}[k+1] = x_{\partial^2 A}[k+1]\right) \label{pf:conditionalspecification-step6}
		\end{equation}
		does not depend on $H$. This follows from the fact that there exists a measurable function $\Phi \colon (\PolishX^{k+1})^{A} \times \Nspace^{A} \times (\PolishX^{k+1})^{\partial A} \rightarrow (\PolishX^{k+2})^{A}$, independent of the choice of subgraph $H$, such that
		\begin{equation*}
			{X^H_A}[k+1] = \Phi({X^H_A}[k], \xi_A(k+1), X^H_{\partial A}[k]).
		\end{equation*}
		Using this map, the measure in \eqref{pf:conditionalspecification-step6} is then obtained as an image measure from \eqref{pf:conditionalspecification-step5}.
	\end{enumerate}
	Therefore the statement holds by induction. \qed

\section{Proof of Proposition \ref{prop:property-UGW}}
\label{sec:pf-property-UGW}

In this section we prove Proposition \ref{prop:property-UGW}.

\begin{Lemma} \label{lem:GW-exchangeable}
Suppose Condition \ref{cond:UGW} holds.
For each $k \in \N_0$ and each bounded measurable function $h \colon (\PolishX^{k+1})^2 \times \SQ(\PolishX^{k+1})^2 \to \Rmb$, it holds a.s.\ on the event $\{N_{\myroot}(\tree) \neq \emptyset\}$ that 
	\begin{align*}
		\E & \left[\left. \frac{1}{|N_{\myroot}(\tree)|} \sum_{u \in N_{\myroot}(\tree)} h(X_{\myroot}[k],X_u[k],\lan X_{N_{\myroot}(\tree)}[k] \ran,\lan X_{N_u(\tree)}[k] \ran) \, \right| \, X_{\myroot}[k], \, \lan X_{N_{\myroot}(\tree)}[k]\ran\right] \\
		& = \E\left[\left. h(X_{\myroot}[k],X_1[k],\lan X_{N_{\myroot}(\tree)}[k] \ran,\lan X_{N_1(\tree)}[k] \ran) \, \right| \, X_{\myroot}[k], \, \lan X_{N_{\myroot}(\tree)}[k]\ran\right].
	\end{align*}
\end{Lemma}
\begin{proof}
	Fix $k \in \N_0$. 
	Throughout this proof we will omit the argument $[k]$ for the sake of readability, with the understanding that every appearance of $X_v$ below should be written more precisely as $X_v[k]$.
	We first prove the claim assuming that $h$ has the following form: there exists a bounded measurable function $f \colon (\PolishX^{k+1})^2 \to \Rmb$ such that
	\begin{equation}
		\label{eq:GW-exchangeable-form1}
		h(z_1,z_2,x_1,x_2) = f(z_1,z_2), \quad z_1,z_2 \in \PolishX^{k+1}, \quad x_1,x_2 \in \SQ(\PolishX^{k+1}).
	\end{equation}	
	Fix $n \in \N$.
	By Lemma \ref{lem:exch},
	\[
	\Lmc((X_{\myroot},X_1,\ldots,X_n) \,|\, |N_{\myroot}(\tree)|=n) = \Lmc((X_{\myroot},X_{\pi(1)},\ldots,X_{\pi(n)}) \,|\, |N_{\myroot}(\tree)|=n)
	\]
	for any permutation $\pi$ of $\{1,\dotsc,n\}$.
	Hence
	\[
	\Lmc((X_{\myroot},X_1) \,|\, X_{\myroot} , \, \lan X_{N_{\myroot}(\tree)} \ran, \, |N_{\myroot}(\tree)|=n) = \Lmc((X_{\myroot},X_u) \,|\, X_{\myroot} , \, \lan X_{N_{\myroot}(\tree)} \ran, \, |N_{\myroot}(\tree)|=n)
	\]
	for each $u=1,\dotsc,n$, and we have
	\begin{align*}
		\frac{1}{n} \sum_{u=1}^n f(X_{\myroot} ,X_u )
		& = \E\Bigg[ \frac{1}{n} \sum_{u =1}^nf(X_{\myroot} ,X_u ) \, \Bigg| \, X_{\myroot} , \, \lan X_{N_{\myroot}(\tree)} \ran, \, |N_{\myroot}(\tree)|=n\Bigg] \\
		& = \E\big[ f(X_{\myroot} ,X_1 ) \, \big| \, X_{\myroot} , \, \lan X_{N_{\myroot}(\tree)} \ran, \, |N_{\myroot}(\tree)|=n\big].
	\end{align*}
	In other words, it holds a.s.\ on $\{1 \in \tree\} = \{N_{\myroot}(\tree) \ne \emptyset\}$ that
	\begin{equation*}
		\frac{1}{|N_{\myroot}(\tree)|} \sum_{u \in N_{\myroot}(\tree)}f(X_{\myroot} ,X_u ) = \E\big[ f(X_{\myroot},X_1) \, \big| \, X_{\myroot} , \, \lan X_{N_{\myroot}(\tree)} \ran, \, |N_{\myroot}(\tree)| \big].
	\end{equation*}
	Because $|N_{\myroot}(\tree)|$ is a.s.\ $\lan X_{N_{\myroot}(\tree)} \ran$-measurable (recall Remark \ref{re:treemeasurable}), this implies
	\begin{equation*}
		\frac{1}{|N_{\myroot}(\tree)|} \sum_{u \in N_{\myroot}(\tree)}f(X_{\myroot} ,X_u ) = \E\big[ f(X_{\myroot},X_1) \, \big| \, X_{\myroot} , \, \lan X_{N_{\myroot}(\tree)} \ran \big],
	\end{equation*}
	again on the event $\{1 \in \tree\}$.
	Thus, the proof is complete for $h$ of the form \eqref{eq:GW-exchangeable-form1}.

	We now prove the claim for general $h$. 
 First note that a consequence of Proposition  \ref{prop:properties-GW}(ii) is that there exists a measurable function $\Phi_k \colon ( \PolishX^{k+1})^2 \rightarrow \P(\SQ(\PolishX^{k+1}))$ such that, for every $v \in \N$,
	\[
	\Phi_k(X_u , X_{\myroot} ) = \Lmc\left( \lan X_{N_u(\tree)} \ran \, | \, X_u , X_{\myroot}  \right), \quad \mbox{ a.s., on } \{u \in \tree\}.
	\]
	Using the conditional independence properties established in Proposition \ref{prop:properties-GW}(i), we have also 
	\begin{equation} \label{eq:GW-exchangeable-3}
		\Phi_k(X_u , X_{\myroot} ) = \Lmc\left(  \lan X_{N_u(\tree)} \ran \, | \, Y_{\V_1}  \right), \quad \mbox{ a.s., on } \{u \in \tree\},
	\end{equation}
	where as usual $Y_u=Y_u[k]=(1_{\{u \in \tree\}},X_u[k])$.
	Noting that $|N_{\myroot}(\tree)|$ is a.s.\ $X_{\V_1}$-measurable, we may use the tower property of conditional expectation to get, on $\{1 \in \tree\}$,
	\begin{align*}
		\E&\Bigg[ \frac{1}{|N_{\myroot}(\tree)|}\sum_{u \in N_{\myroot}(\tree)} h(X_{\myroot}, X_u,  \lan X_{N_{\myroot}(\tree)}  \ran,\lan X_{N_u(\tree)}  \ran) \, \Bigg| \, X_{\myroot} , \, \lan X_{N_{\myroot}(\tree)} \ran\Bigg] \\
		&= \E\Bigg[ \frac{1}{|N_{\myroot}(\tree)|}\sum_{u \in N_{\myroot}(\tree)} \E\big[h(X_{\myroot}, X_u,  \lan X_{N_{\myroot}(\tree)}  \ran,\lan X_{N_u(\tree)}  \ran) \, | \, X_{\V_1} \big] \,  \, \Bigg| \, X_{\myroot} , \, \lan X_{N_{\myroot}(\tree)} \ran\Bigg] \\
		&= \E\Bigg[ \frac{1}{|N_{\myroot}(\tree)|}\sum_{u \in N_{\myroot}(\tree)} \Big\lan \Phi_k(X_u ,X_{\myroot} ), \, h(X_{\myroot}, X_u,  \lan X_{N_{\myroot}(\tree)}  \ran, \, \cdot \,  )\Big\ran  \, \Bigg| \, X_{\myroot} , \, \lan X_{N_{\myroot}(\tree)} \ran\Bigg] \\
		&= \E\big[ \lan \Phi_k(X_1 ,X_{\myroot} ), \, h(X_{\myroot}, X_1,  \lan X_{N_{\myroot}(\tree)}  \ran, \, \cdot \,  )\ran  \, \big| \, X_{\myroot} , \, \lan X_{N_{\myroot}(\tree)} \ran\big],
	\end{align*}
	where the second equality used \eqref{eq:GW-exchangeable-3}, and the last equality used the result of the first part of the proof. Now, apply \eqref{eq:GW-exchangeable-3} once again to write this as 
	\begin{align*}
	\E &\Big[   \E\big[h(X_{\myroot}, X_1,  \lan X_{N_{\myroot}(\tree)}  \ran, \lan X_{N_1(\tree) } \ran ) \, | \, X_{\V_1} \big]  \, \Big| \, X_{\myroot} , \, \lan X_{N_{\myroot}(\tree)} \ran\Big] \\
		&= \E\left[\left.  h(X_{\myroot}, X_1,  \lan X_{N_{\myroot}(\tree)}  \ran, \lan X_{N_1(\tree) } \ran )  \, \right| \, X_{\myroot} , \, \lan X_{N_{\myroot}(\tree)} \ran\right],  \ \ \ a.s., \text{ on } \{1 \in \tree\}.
	\end{align*}
	This completes the proof of the lemma. 
\end{proof}

\begin{proof}[Proof of Proposition \ref{prop:property-UGW}]
	Fix $k \in \N_0$, and let $g \colon (\PolishX^{k+1})^2 \to \Rmb$ be any bounded measurable function.
	Assume without loss of generality that $h$ and $g$ are nonnegative.
	Because $k$ is fixed, throughout this proof we again omit the argument $[k]$ for the sake of readability.
We saw in 	Lemma \ref{lem:unimodular} that $(\tree,Y_\tree)$ is unimodular, and thus so is $(\tree,X_\tree)$. Applying the defining property of unimodularity \eqref{def:unimodularity} to the map $H_k \colon \Gmc_{**}[\PolishX^{k+1}] \rightarrow \R$ given by
	\begin{equation*}
	H_k(G,x,\myroot,u) := g(x_{\myroot},x_u) h(x_u,x_{\myroot},\lan x_{N_u(G)} \ran) 1_{\{u \in N_{\myroot}(G)\}}\frac{1}{|N_{\myroot}(G)|},
	\end{equation*}
we derive
	\begin{align}
		\E&\left[g(X_{\myroot},X_1)h(X_1,X_{\myroot},\lan X_{N_1(\tree)}\ran)1_{\{1 \in \tree\}} \right]  \nonumber \\
		&= \E\left[\frac{1}{|N_{\myroot}(\tree)|}\sum_{v \in N_{\myroot}(\tree)}g(X_{\myroot},X_v)h(X_v,X_{\myroot},\lan X_{N_v(\tree)}\ran) \right] \nonumber \\
		&= \E\left[\sum_{v \in \tree}g(X_{\myroot},X_v)h(X_v,X_{\myroot},\lan X_{N_v(\tree)}\ran)1_{\{v \in N_{\myroot}(\tree)\}}\frac{1}{|N_{\myroot}(\tree)|} \right] \nonumber \\
		&= \E\left[\sum_{v \in \tree}g(X_v,X_{\myroot})h(X_{\myroot},X_v,\lan X_{N_{\myroot}(\tree)}\ran)1_{\{\myroot \in N_{v}(\tree)\}}\frac{1}{|N_{v}(\tree)|} \right] \nonumber \\
		&= \E\left[\frac{1}{|N_{\myroot}(\tree)|}\sum_{v \in N_{\myroot}(\tree)}g(X_v,X_{\myroot})h(X_{\myroot},X_v,\lan X_{N_{\myroot}(\tree)}\ran) \frac{|N_{\myroot}(\tree)|}{|N_v(\tree)|} \right] \nonumber \\
		&= \E\left[ g(X_1,X_{\myroot})h(X_{\myroot},X_1,\lan X_{N_{\myroot}(\tree)}\ran) \frac{|N_{\myroot}(\tree)|}{|N_1(\tree)|} 1_{\{1 \in \tree\}}\right], \label{eq:unimod-key1}
	\end{align}
	where the first equality uses the measurability of $\{1 \in \tree\}$ with respect to $\lan X_{N_{\myroot}(\tree)}\ran$ and Lemma \ref{lem:GW-exchangeable}, the third equality uses unimodularity with the function $H_k$ defined above,
	the fourth equality uses the fact that $\myroot \in N_v(\tree)$ if and only if $v \in N_{\myroot}(\tree)$, and the last equality uses again Lemma \ref{lem:GW-exchangeable} and the measurability of $|N_{\myroot}(\tree)|$ and $|N_v(\tree)|$ with respect to $\lan X_{N_{\myroot}(\tree)}\ran$ and $\lan X_{N_v(\tree)}\ran$, respectively.
	
	If we define $\varphi_h \colon (\PolishX^{k+1})^2 \to \Rmb$ by
	$$\varphi_h(X_{\myroot},X_1) = \Emb \left[ \frac{|N_{\myroot}(\tree)|}{|N_1(\tree)|} h(X_{\myroot},X_1,\lan X_{N_{\myroot}(\tree)}\ran) \Big| X_{\myroot},X_1 \right]1_{\{1 \in \tree\}}$$
	(which makes sense because $\{1 \in \tree\}=\{X_1\neq\cem\}$ is $X_1$-measurable),
	then \eqref{eq:unimod-key1} implies
	\begin{equation}
	\E\left[g(X_{\myroot},X_1)h(X_1,X_{\myroot},\lan X_{N_1(\tree)}\ran)1_{\{1 \in \tree\}} \right] = \E\left[ g(X_1,X_{\myroot})\varphi_h(X_{\myroot},X_1) 1_{\{1 \in \tree\}}\right]. \label{eq:unimod-key1.5}
	\end{equation}
	We will similarly define $\varphi_1 \colon (\PolishX^{k+1})^2 \to \Rmb$ by
	\begin{equation*}
	\varphi_1(X_{\myroot},X_1) = \E\left[\left. \frac{|N_{\myroot}(\tree)|}{|N_1(\tree)|} \, \right| \, X_{\myroot}, X_1\right] 1_{\{1 \in \tree\}}.
	\end{equation*}
	Apply the procedure \eqref{eq:unimod-key1.5}, with $h$ replaced by the constant function $1$ and with $g(x_{\myroot},x_1)$ replaced by $g(x_1,x_{\myroot})\varphi_h(x_{\myroot},x_1)$, to get
	\begin{equation*}
	\E \left[ g(X_1,X_{\myroot})\varphi_h(X_{\myroot},X_1) 1_{\{1 \in \tree\}}\right] = \E\left[ g(X_{\myroot},X_1)\varphi_h(X_1,X_{\myroot})\varphi_1(X_{\myroot},X_1) 1_{\{1 \in \tree\}}\right].
	\end{equation*}
	Combining this with \eqref{eq:unimod-key1.5}, we have
	\begin{equation}	
	\E \left[g(X_{\myroot},X_1)h(X_1,X_{\myroot},\lan X_{N_1(\tree)}\ran)1_{\{1 \in \tree\}} \right] = \E\left[ g(X_{\myroot},X_1)\varphi_h(X_1,X_{\myroot})\varphi_1(X_{\myroot},X_1) 1_{\{1 \in \tree\}}\right]. \label{eq:unimod-key2}
	\end{equation}
	As this holds for any non-negative $g$, we deduce that, a.s.\ on $\{1 \in \tree\}$,
	\begin{equation*}
	\E\left[\left. h(X_1,X_{\myroot},\lan X_{N_1(\tree)}\ran) \, \right| \, X_{\myroot}, X_1\right] = \varphi_h(X_1,X_{\myroot})\varphi_1(X_{\myroot},X_1).
	\end{equation*}
	On the other hand, taking $h \equiv 1$ in \eqref{eq:unimod-key2}, we deduce that $\varphi_1(X_1,X_{\myroot})\varphi_1(X_{\myroot},X_1) = 1$ a.s.\ on $\{1 \in \tree\}$, and so
	\begin{equation*}
	\E\left[\left. h(X_1,X_{\myroot},\lan X_{N_1(\tree)}\ran) \, \right| \, X_{\myroot}, X_1\right] = \frac{\varphi_h(X_1,X_{\myroot})}{\varphi_1(X_1,X_{\myroot})}.
	\end{equation*}
	Now recall that by definition (still omitting $[k]$ from the notation)
	\begin{equation*}
		\Xi_k(X_{\myroot}[k], X_1[k]) = 1_{\{1 \in \tree\}} \frac{\varphi_h(X_{\myroot},X_1)}{\varphi_1(X_{\myroot},X_1)},
	\end{equation*}
	and note that \eqref{eq:property-UGW1} for $v=1$ follows.
	In light of the symmetry provided by Proposition  \ref{prop:properties-GW}(ii), this is enough to complete the proof.
\end{proof}

\section{A Gaussian model on a regular trees}
\label{sec:char_regular_tree}

In this section we present a simple AR(1) model on regular trees whose local equations can be simulated efficiently. 

\subsection{Example of a $\kappa$-regular tree}

Suppose $G$ is the $\kappa$-regular tree for some $\kappa \ge 2$.
With Gaussian initial states, Gaussian noises and affine update rules, the solution $(X_i[k])_{i=0}^{\kappa}$ is Gaussian at each time $k$. 
The means are easy to compute, but the covariance matrix is more difficult. 
For $(X_i[k])_{i=0}^{\kappa}$, the covariance matrix $\Sigma_k$ is $(k+1)(\kappa+1) \times (k+1)(\kappa+1)$ and the conditional law $\cmeas[k]$ appearing in Construction \ref{constr:regulartree_local} is again Gaussian, the parameters of which are obtained by inverting the (sub-) covariance matrix $\Upsilon_k$ of $(X_0[k],X_1[k])$, which is $2(k+1) \times 2(k+1)$. 
This matrix inversion may seem prohibitively expensive, but the key point is that $\Upsilon_k$ can be obtained efficiently from $\Upsilon_{k-1}$ via block matrix inversion.

In the following, for a random (column) vector $\Xbd$, we write $\Var(\Xbd) := \E[\Xbd\Xbd^\top] -\E[\Xbd]\E[\Xbd^\top]$ to denote the covariance matrix. Similarly, for a list of random vectors $X_1,\ldots,X_m$, we interpret $(X_1,\ldots,X_m)^\top$ as the (column) vector obtained by concatenation, and $\Var(X_1,\ldots,X_m)$ is the corresponding covariance matrix. On the other hand, for two vectors $\Xbd$ and $\Ybd$, we write $\Cov(\Xbd,\Ybd) := \E[\Xbd\Ybd^\top] -\E[\Xbd]\E[\Ybd^\top]$ for the matrix of covariances.

Consider the setup in Section \ref{se:regresults}. 
Again we denote by $1,\dotsc,\kappa$ the neighbors of a root vertex $0$, and vertices $ij$ are neighbors of node $i$, for $j=1,\ldots,\kappa-1$ and $i=1,\dotsc,\kappa$.
	Suppose $\{ \xi_v(k) : v \in V, k \in \N_0 \}$ are independent standard Gaussians.
	Consider real-valued $\Xbd(k) := \{ X_v(k) : v \in V, k \in \N\}$ given by the following affine system
	\begin{equation}
		\label{eq:affine_regular_tree}
 X_v(k+1) = aX_v(k) + b \sum_{u \in N_v(G)} X_u(k) + c + \xi_v(k+1),
	\end{equation}
	for some constants $a,b,c \in \Rmb$, initialized from i.i.d.\ standard Gaussians $(X_v(0))_{v \in V}$.
	
	By symmetry, we have $\Lmc(X_v[k])=\Lmc(X_u[k])$ for all $u,v \in V$ and $k \in \N_0$.
	Thus $(X_i[k])_{i=0}^\kappa$ is Gaussian with mean vector $(m_0,\dotsc,m_k) \in \R^{k+1}$ and some covariance matrix which we denote $\Sigma_k \in \R^{(k+1)\times(k+1)}$.
Taking expectations in \eqref{eq:affine_regular_tree} gives
	\begin{equation*}
		m_{k+1} = (a+\kappa b)m_k + c, \quad m_0=0,
	\end{equation*}
	which immediately gives
	\begin{equation*}
m_k =  c\frac{(a+\kappa b)^k - 1}{a+\kappa b -1}, \quad k \in \N_0.
	\end{equation*}

The calculation of $\Sigma_k$ needs some notations.
Noting that by symmetry properties (automorphism invariance) of the system \eqref{eq:affine_regular_tree}, for fixed time instants $k$ and $l$, the covariance $\Cov(X_u(k),X_v(l))$ depends on vertices $u$ and $v$ only through their distance.
Hence, letting
$$\Omega_{k,i} := \Cov(X_u[k],X_v[k]), \quad d(u,v)=i, \quad i \in \N_0,$$
we can write the following block matrix decomposition
$$\Sigma_k = \Var((X_i[k])_{i=0}^\kappa) = (\Sigma_{k}(i,j))_{i,j=0}^\kappa,$$
with
$$\Sigma_k(i,j) = \begin{cases} 
	\Omega_{k,0}, & i=j, \\
	\Omega_{k,1}, & i=0 \mbox{ or } j=0, i \ne j, \\
	\Omega_{k,2}, & \mbox{otherwise}. \end{cases}$$
We claim that one can calculate $\Omega_{k+1,i}, i=0,1,2$ in an inductive manner as follows, using the dynamics \eqref{eq:affine_regular_tree} and the conditional independence property in Theorem \ref{thm:cond_independence}, once we have $\Omega_{k,i}, i=0,1,2$.
The detailed justification is deferred to Appendix \ref{sec:pf-example} and the analysis of computational complexity is provided in Appendix \ref{sec:complexity}.

For $i=0,1,2,3,4$, $k \in \Nmb_0$, and any vertices $u,v$ such that $d(u,v)=i$, let
\begin{align*}
	A_{k+1,i} & := \Cov(X_u(k+1),X_v(k+1)), \\ 
	B_{k+1,i} & := \Cov(X_u[k],X_v(k+1)), \\
	C_{k+1,i} & := \Cov(X_u[k+1],X_v(k+1)) = \begin{bmatrix}
		B_{k+1,i} \\
		A_{k+1,i} 
		\end{bmatrix}
\end{align*}
so that
\begin{equation*}
	\Omega_{k+1,i} = \begin{bmatrix}
	\Cov(X_u[k],X_v[k]) & \Cov(X_u[k],X_v(k+1)) \\
	\Cov(X_u(k+1),X_v[k]) & \Cov(X_u(k+1),X_v(k+1))
	\end{bmatrix} = \begin{bmatrix}
	\Omega_{k,i} & B_{k+1,i} \\
	B_{k+1,i}^\top & A_{k+1,i}
	\end{bmatrix}.
\end{equation*}
Let $\Omega_{0,i} = A_{0,i} = 1_{\{i=0\}}$, $B_{0,i}=0$, and $C_{0,i}=\begin{bmatrix}
    0 \\ 0
\end{bmatrix}$.
Let
\begin{equation*}
	\Upsilon_{k} := \Var((X_0[k],X_1[k])) = \begin{bmatrix}
	\Omega_{k,0} & \Omega_{k,1} \\
	\Omega_{k,1} & \Omega_{k,0}
	\end{bmatrix}.
\end{equation*}
Then one can show that
\begin{align*}
	A_{k+1,0} & = (a^2+\kappa b^2) A_{k,0} + 2ab\kappa A_{k,1} + \kappa(\kappa-1)b^2 A_{k,2} + 1, \\
	B_{k+1,0} & = a C_{k,0} + \kappa b C_{k,1}, \\	
	A_{k+1,1} & = 2ab A_{k,0} + (a^2+\kappa b^2+ab(\kappa-1)) A_{k,1} + 2ab(\kappa-1) A_{k,2} + b^2(\kappa-1)^2 A_{k,3}, \\
	B_{k+1,1} & = a C_{k,1} + b C_{k,0} + b(\kappa-1) C_{k,2}, \\
	A_{k+1,2} & = b^2 A_{k,0} + 2ab A_{k,1} + (a^2+b^2(\kappa-1)+ab(\kappa-1)) A_{k,2} + 2ab(\kappa-1) A_{k,3} + b^2(\kappa-1)^2 A_{k,4}, \\
	B_{k+1,2} & = a C_{k,2} + b C_{k,1} + b(\kappa-1) C_{k,3},	\\
	C_{k,3} & = 		
		\begin{bmatrix} 
		\Omega_{k,1} & \Omega_{k,2} 
		\end{bmatrix} 
		\Upsilon_{k}^{-1}
		\begin{bmatrix} 
		C_{k,2} \\ C_{k,1}
		\end{bmatrix}, \\
	A_{k,4} & = 		
		\begin{bmatrix} 
		C_{k,1}^\top & C_{k,2}^\top 
		\end{bmatrix} 
		\Upsilon_{k}^{-1}
		\begin{bmatrix} 
		\Omega_{k,2} & \Omega_{k,1} \\
		\Omega_{k,1} & \Omega_{k,0}
		\end{bmatrix}
		\Upsilon_{k}^{-1}
		\begin{bmatrix} 
		C_{k,1} \\ C_{k,2} 
		\end{bmatrix},
\end{align*}
and $\Upsilon_{k}^{-1}$ is calculated as follows with a reorder of rows and columns:
Write
\begin{equation*}
	\widetilde{\Upsilon}_{k} = \Var(X_0(0),X_1(0),\dotsc,X_0(k),X_1(k)) 
	= \begin{pmatrix}
	\widetilde{\Upsilon}_{k-1} & \widetilde{B}_{k} \\
	\widetilde{B}_{k}^\top & \widetilde{A}_{k}
	\end{pmatrix},
\end{equation*}
where 
\begin{align*}
\widetilde{A}_{k} &:= \Var((X_0(k),X_1(k))), \\
\widetilde{B}_{k} &:= \Cov((X_0(0),X_1(0),\dotsc,X_0(k-1),X_1(k-1)),(X_0(k),X_1(k)))
\end{align*}
are obtained from $\Omega_{k,0}$ and $\Omega_{k,1}$ after a reorder of rows and columns.
Then
\begin{equation*}
	\widetilde{\Upsilon}_{k}^{-1} 
	= \begin{pmatrix}
			\widetilde{\Upsilon}_{k-1}^{-1}+\widetilde{\Upsilon}_{k-1}^{-1}\widetilde{B}_{k}\widetilde{C}_{k}^{-1}\widetilde{B}_{k}^\top \widetilde{\Upsilon}_{k-1}^{-1} & -\widetilde{\Upsilon}_{k-1}^{-1}\widetilde{B}_{k}\widetilde{C}_{k}^{-1} \\
			-\widetilde{C}_{k}^{-1}\widetilde{B}_{k}^\top \widetilde{\Upsilon}_{k-1}^{-1} & \widetilde{C}_{k}^{-1}
		\end{pmatrix},
\end{equation*}
where $\widetilde{C}_{k}:=\widetilde{A}_{k}-\widetilde{B}_{k}^\top \widetilde{\Upsilon}_{k-1}^{-1}\widetilde{B}_{k}$ is the Schur complement of $\widetilde{\Upsilon}_{k-1}$ in $\widetilde{\Upsilon}_{k}$.

\subsubsection{Justification of the method}
\label{sec:pf-example}

For $A_{k+1,0}$ and $B_{k+1,0}$, from the dynamics \eqref{eq:affine_regular_tree} we have
\begin{align*}
	A_{k+1,0} & = \Cov(X_0(k+1),X_0(k+1)) \\
	& = \Var\left(aX_0(k) + b \sum_{j=1}^\kappa X_j(k) + c + \xi_0(k+1)\right) \\
	& = a^2\Var(X_0(k)) + b^2\left(\kappa \Var(X_0(k)) + \kappa(\kappa-1) \Cov(X_1(k),X_2(k))\right) + 1 \\
	& \qquad + 2ab\kappa\Cov(X_0(k),X_1(k)) \\
	& = (a^2+\kappa b^2) A_{k,0} + 2ab\kappa A_{k,1} + \kappa(\kappa-1)b^2 A_{k,2} + 1,
\end{align*}
and
\begin{align*}
	B_{k+1,0} & = \Cov(X_0[k],X_0(k+1)) \\
	& = \Cov\left(X_0[k], aX_0(k) + b \sum_{j=1}^\kappa X_j(k) + c + \xi_0(k+1)\right) \\	
	& = a\Cov\left(X_0[k], X_0(k)\right) + \kappa b\Cov\left(X_0[k], X_1(k)\right) \\
	& = a C_{k,0} + \kappa b C_{k,1}.	
\end{align*}

For $A_{k+1,1}$ and $B_{k+1,1}$, from the dynamics \eqref{eq:affine_regular_tree} we have
\begin{align*}
	A_{k+1,1} & = \Cov(X_0(k+1),X_1(k+1)) \\
	& = \Cov\left(aX_0(k) + bX_1(k) + b\sum_{j=2}^\kappa X_j(k) + c + \xi_0(k+1), \right. \\
    & \qquad \left. aX_1(k) + bX_0(k) + b \sum_{j=1}^{\kappa-1} X_{1j}(k) + c + \xi_1(k+1)\right) \\
	& = a^2\Cov(X_0(k),X_1(k)) + ab\Var(X_0(k)) + ab(\kappa-1)\Cov(X_0(k),X_{11}(k))  \\
	& \qquad + ba\Var(X_1(k)) + b^2\Cov(X_1(k),X_0(k)) + b^2(\kappa-1)\Cov(X_1(k),X_{11}(k)) \\
	& \qquad + ba(\kappa-1)\Cov(X_2(k),X_1(k)) + ba(\kappa-1)\Cov(X_2(k),X_0(k)) \\
    & \qquad + b^2(\kappa-1)^2\Cov(X_2(k),X_{11}(k)) \\
	& = 2ab A_{k,0} + (a^2+\kappa b^2+ab(\kappa-1)) A_{k,1} + 2ab(\kappa-1) A_{k,2} + b^2(\kappa-1)^2 A_{k,3},
\end{align*}
and
\begin{align*}
	B_{k+1,1} & = \Cov(X_0[k],X_1(k+1)) \\
	& = \Cov\left(X_0[k], aX_1(k) + bX_0(k) + b \sum_{j=1}^{\kappa-1} X_{1j}(k) + c + \xi_1(k+1)\right) \\	
	& = a \Cov\left(X_0[k], X_1(k)\right) + b \Cov\left(X_0[k],X_0(k)\right) + b(\kappa-1)\Cov\left(X_0[k], X_{11}(k)\right) \\
	& = a C_{k,1} + b C_{k,0} + b(\kappa-1) C_{k,2}.	
\end{align*}

For $A_{k+1,2}$ and $B_{k+1,2}$, from the dynamics \eqref{eq:affine_regular_tree} we have
\begin{align*}
	A_{k+1,2} & = \Cov(X_1(k+1),X_2(k+1)) \\
	& = \Cov\left(aX_1(k) + bX_0(k) + b \sum_{j=1}^{\kappa-1} X_{1j}(k) + c + \xi_1(k+1), \right. \\
    & \qquad \left. aX_2(k) + bX_0(k) + b \sum_{j=1}^{\kappa-1} X_{2j}(k) + c + \xi_2(k+1)\right) \\
	& = a^2\Cov(X_1(k),X_2(k)) + ab\Cov(X_1(k),X_0(k)) + ab(\kappa-1)\Cov(X_1(k),X_{21}(k))  \\
	& \qquad + ba\Cov(X_0(k),X_2(k)) + b^2\Var(X_0(k)) + b^2(\kappa-1)\Cov(X_0(k),X_{21}(k)) \\
	& \qquad + ba(\kappa-1)\Cov(X_{11}(k),X_2(k)) + ba(\kappa-1)\Cov(X_{11}(k),X_0(k)) \\
    & \qquad + b^2(\kappa-1)^2\Cov(X_{11}(k),X_{21}(k)) \\
	& = b^2 A_{k,0} + 2ab A_{k,1} + (a^2+b^2(\kappa-1)+ab(\kappa-1)) A_{k,2} + 2ab(\kappa-1) A_{k,3} + b^2(\kappa-1)^2 A_{k,4},
\end{align*}
and
\begin{align*}
	B_{k+1,2} & = \Cov(X_1[k],X_2(k+1)) \\
	& = \Cov\left(X_1[k], aX_2(k) + bX_0(k) + b \sum_{j=1}^{\kappa-1} X_{2j}(k) + c + \xi_2(k+1)\right) \\	
	& = a \Cov\left(X_1[k], X_2(k)\right) + b \Cov\left(X_1[k], X_0(k)\right) + b(\kappa-1)\Cov\left(X_1[k], X_{21}(k)\right) \\
	& = a C_{k,2} + b C_{k,1} + b(\kappa-1) C_{k,3}.	
\end{align*}

For $A_{k,3}$, $B_{k,3}$ and $A_{k,4}$ appearing in the above expressions, first note that the conditional independence proven in Theorem \ref{thm:cond_independence}, by taking $A=\Vmb_{2+}$ therein, ensures that $X_{21}(k)$ and $X_1[k]$ are conditionally independent given $(X_0[k],X_2[k])$.
Using this and the law of total covariance, we have
\begin{align*}
	C_{k,3} & = \Cov(X_1[k],X_{21}(k)) \\
	& = \Cov\Big( \Emb[X_1[k] \mid (X_0[k],X_2[k],X_{21}(k))], \Emb[X_{21}(k) \mid (X_0[k],X_2[k],X_{21}(k))] \Big) \\
	& = \Cov\Big( \Emb[X_1[k] \mid (X_0[k],X_2[k])], X_{21}(k) \Big).
\end{align*}
Since
$$\Var(X_1[k],X_0[k],X_2[k])=\begin{bmatrix}
\Omega_{k,0} & \Omega_{k,1} & \Omega_{k,2} \\
\Omega_{k,1} & \Omega_{k,0} & \Omega_{k,1} \\
\Omega_{k,2} & \Omega_{k,1} & \Omega_{k,0}
\end{bmatrix},$$
by a standard property of multivariate Gaussian distributions, the conditional law
\begin{equation*}
	\Lmc(X_1[k] \,|\, X_0[k],X_2[k]) 
\end{equation*}
is Gaussian with mean
\begin{equation*}
	\Emb[X_1[k]] + 
	\begin{bmatrix} 
	\Omega_{k,1} & \Omega_{k,2} 
	\end{bmatrix} 
	\begin{bmatrix}
	\Omega_{k,0} & \Omega_{k,1} \\
	\Omega_{k,1} & \Omega_{k,0}
	\end{bmatrix}^{-1} 
	\left(
	\begin{bmatrix} 
	X_0[k] \\ X_2[k] 
	\end{bmatrix} 
	-
	\begin{bmatrix} 
	\Emb[X_0[k]] \\ \Emb[X_2[k]] 
	\end{bmatrix}\right)
\end{equation*}
and variance
\begin{equation*}
	\Omega_{k,0} - 	
	\begin{bmatrix} 
	\Omega_{k,1} & \Omega_{k,2} 
	\end{bmatrix} 
	\begin{bmatrix}
	\Omega_{k,0} & \Omega_{k,1} \\
	\Omega_{k,1} & \Omega_{k,0}
	\end{bmatrix}^{-1} 
	\begin{bmatrix} 
	\Omega_{k,1} \\ \Omega_{k,2} 
	\end{bmatrix}.
\end{equation*}	
Therefore,
\begin{align*}
	C_{k,3} 
	& = \Cov\left( 	
		\begin{bmatrix} 
		\Omega_{k,1} & \Omega_{k,2} 
		\end{bmatrix} 
		\begin{bmatrix}
		\Omega_{k,0} & \Omega_{k,1} \\
		\Omega_{k,1} & \Omega_{k,0}
		\end{bmatrix}^{-1}
		\begin{bmatrix} 
		X_0[k] \\ X_2[k] 
		\end{bmatrix},
		X_{21}(k)		
		\right) \\
	& = 		
		\begin{bmatrix} 
		\Omega_{k,1} & \Omega_{k,2} 
		\end{bmatrix} 
		\begin{bmatrix}
		\Omega_{k,0} & \Omega_{k,1} \\
		\Omega_{k,1} & \Omega_{k,0}
		\end{bmatrix}^{-1}
		\begin{bmatrix} 
		C_{k,2} \\ C_{k,1}
		\end{bmatrix}.
\end{align*}

Similarly, using the conditional independence proven in Theorem \ref{thm:cond_independence} and the law of total covariance, we have
\begin{align*}
	A_{k,4} & = \Cov(X_{11}(k),X_{21}(k)) \\
	& = \Cov\Big( \Emb[X_{11}(k) \mid (X_0[k],X_1[k],X_2[k])], \Emb[X_{21}(k) \mid (X_0[k],X_1[k],X_2[k])] \Big) \\
	& = \Cov\Big( \Emb[X_{11}(k) \mid (X_0[k],X_1[k])], \Emb[X_{21}(k) \mid (X_0[k],X_2[k])] \Big).
\end{align*}
Since
$$\Var(X_{11}(k),X_1[k],X_0[k])=\Var(X_{21}(k),X_2[k],X_0[k])=\begin{bmatrix}
A_{k,0} & C_{k,1}^\top & C_{k,2}^\top \\
C_{k,1} & \Omega_{k,0} & \Omega_{k,1} \\
C_{k,2} & \Omega_{k,1} & \Omega_{k,0}
\end{bmatrix},$$
we have
\begin{align*}
	A_{k,4} 
	& = \Cov\left( 	
		\begin{bmatrix} 
		C_{k,1}^\top & C_{k,2}^\top 
		\end{bmatrix} 
		\begin{bmatrix}
		\Omega_{k,0} & \Omega_{k,1} \\
		\Omega_{k,1} & \Omega_{k,0}
		\end{bmatrix}^{-1}
		\begin{bmatrix} 
		X_1[k] \\ X_0[k] 
		\end{bmatrix},
		\begin{bmatrix} 
		C_{k,1}^\top & C_{k,2}^\top 
		\end{bmatrix} 
		\begin{bmatrix}
		\Omega_{k,0} & \Omega_{k,1} \\
		\Omega_{k,1} & \Omega_{k,0}
		\end{bmatrix}^{-1}
		\begin{bmatrix} 
		X_2[k] \\ X_0[k] 
		\end{bmatrix}		
		\right) \\
	& = 		
		\begin{bmatrix} 
		C_{k,1}^\top & C_{k,2}^\top 
		\end{bmatrix} 
		\begin{bmatrix}
		\Omega_{k,0} & \Omega_{k,1} \\
		\Omega_{k,1} & \Omega_{k,0}
		\end{bmatrix}^{-1}
		\begin{bmatrix} 
		\Omega_{k,2} & \Omega_{k,1} \\
		\Omega_{k,1} & \Omega_{k,0}
		\end{bmatrix}
		\begin{bmatrix}
		\Omega_{k,0} & \Omega_{k,1} \\
		\Omega_{k,1} & \Omega_{k,0}
		\end{bmatrix}^{-1}
		\begin{bmatrix} 
		C_{k,1} \\ C_{k,2} 
		\end{bmatrix}.
\end{align*}

Lastly, the inductive formula for $\widetilde{\Upsilon}_{k}^{-1}$ follows directly from the block matrix inversion.

\subsubsection{Analysis of computational complexity and comparison with naive method}
\label{sec:complexity}

In the inductive formula from $k$ to $k+1$, the computational complexity of calculations of $A_{k+1,i}$ and $B_{k+1,i}$, $i=0,1,2$, is $O(k^3)$, provided that $\Upsilon_k^{-1}$ is known.
The computational complexity of calculations of $\Upsilon_k^{-1}$, including computing $\widetilde{\Upsilon}_k^{-1}$ and reordering rows and columns, is also $O(k^3)$.
Here computing the multiplication of matrices of size $O(k)$ is the most time-consuming. 
As a result, the computational complexity of the inductive method from time $0$ to time $k$ is $O(k^4)$.

In contrast, a naive approach would forsake symmetry and conditional independence, instead tracking the entire tree of dependence. 
That is, to simulate a single particle up to time $k$, one would simulate each of its $\kappa$ neighbors up to time $k-1$, each of its neighbors' $\kappa-1$ remaining neighbors up to time $k-2$, and so on, resulting in the enormous total of $1+\kappa+\kappa(\kappa-1)+\dotsb+\kappa(\kappa-1)^k = O((\kappa-1)^{k+1})$ particles.
Therefore the computational complexity grows exponentially in $k$.

\begin{Remark}
	\label{rmk:naive}
	It must be noted that in this simple setting there is an improved naive method that does not use any conditional independence but rather uses just the symmetry of the tree and performs quite well. 
    Define $C_n(k) := \Cov(X_u(k),X_v(k))$ where $u$ and $v$ are vertices of distance $n$. There is a simple recurrence for $C_n(k+1)$ in terms of $\{C_\ell (k) : \ell \in \N_0, |n-\ell| \le 2\}$, and so $C_0(k)$ can be computed in $O(k^2)$ steps.
	However, this approach will fail if there is no global symmetry in the system, as illustrated in the next example. 
\end{Remark}

\subsection{Example of a $\kappa$-regular-ish tree}
In this section we consider an example of a $\kappa$-regular-ish tree where one does not have the global symmetry.
As a result, the above naive method has exponentially growing computational complexity and the improved naive method in Remark \ref{rmk:naive} would not work.
In contrast, we will show that our method still has a  computational complexity that grows only polynomially in the size of the system. 

Fix $\kappa \ge 3$ and $\tilde\kappa \ge 1$.
Recall that in a $\kappa$-regular tree, the root vertex $0$ has $\kappa$ children and every non-root vertex has $\kappa-1$ children.
Suppose the tree $G$ is almost $\kappa$-regular in that for every vertex, one of the children (e.g.\ the first) has $\tilde\kappa \ne \kappa-1$ children.
In the Ulam-Harris-Neveu labeling, this means that for any non-root vertex $v=(v_1,\ldots,v_k) \in \N^k$,
\begin{align*}
	v \in G & \iff v_1 \in \{1,\dotsc,\kappa\} \mbox{ and for each } j=2,\dotsc,k, \\
	& \qquad v_j \in \{1,\dotsc,\kappa-1\} \mbox{ whenever } v_{j-1} \ne 1, v_j \in \{1,\dotsc,\tilde\kappa\} \mbox{ whenever } v_{j-1} = 1.
\end{align*}
We note that, different from the previous example, the graph $G$ does not have  global symmetry.
However, it has the following self-similarity:
For $v \in G$, letting $C(v)$ denote the subtree rooted at $v$, namely $C(v):=\{u \in G : v \le u\}$, we have that 
\begin{equation}
	\label{eq:self-similarity}
	\mbox{$C(v1)$ and $C(1)$ are isomorphic and that $C(v2)$ and $C(2)$ are isomorphic,}
\end{equation} 
for each $v \in G$.

Consider real-valued $\Xbd(k) := \{ X_v(k) : v \in V, k \in \N\}$ given by the affine system \eqref{eq:affine_regular_tree}, namely
\begin{equation*}
X_v(k+1) = aX_v(k) + b \sum_{u \in N_v(G)} X_u(k) + c + \xi_v(k+1),
\end{equation*}
for some constants $a,b,c \in \Rmb$, where $\{ \xi_v(k+1) : v \in V, k \in \N_0 \}$ and $(X_v(0))_{v \in V}$ are independent standard Gaussians.
Our goal is to simulate the Gaussian random variable $X_0(k)$.
Assume for simplicity that $c=0$.
Then $\Emb[X_v[k]]=\zero$, where $\zero$ is a vector of $0$'s of proper dimension.
Due to the lack of symmetry, the naive method of simulating $X_0(k)$ or calculating $\Var(X_0(k))$ in the previous section has computational complexity $O((\max\{\kappa-1,\tilde\kappa\})^{k+1})$ that is exponential in $k$.
Next we will briefly show that our method still has computational complexity $O(k^4)$ that is polynomial in $k$.

For each $k \in \Nmb_0$, consider the trajectory of representatives from the first two generations of the root 
$$(X_0[k],X_1[k],X_{11}[k],X_{12}[k],X_2[k],X_{21}[k],X_{22}[k]),$$
which is Gaussian with mean vector $\zero$ and some covariance matrix denoted by $\Sigma_k$.
Suppose we have calculated $\Sigma_k$ and would like to calculate $\Sigma_{k+1}$.
We need to calculate $\Cov(X_u(k+1),X_v[k+1])$ for $u,v$ in the first two generations.
By the linear evolution of $X_u$ and $X_v$, this reduces to the calculation of $\Cov(X_u[k],X_v[k])$ for $u,v$ in the first three generations, which can be done with the following two key ingredients.
One is the conditional independence property in Theorem \ref{thm:cond_independence} as used in the previous example.
The other is a combination of the conditional consistency property in Theorem \ref{thm:cond_consistency} and the self-similarity property in \eqref{eq:self-similarity}.

To illustrate the idea, we take the calculation of $\Cov(X_{111}(k),X_{222}(k))$ for example and the others can be treated in a similar manner. 
From \eqref{eq:self-similarity} and Theorem \ref{thm:cond_consistency} we have
\begin{align*}
	\begin{aligned}
	\Lmc\left(X_{C(i1)}[k] \mid (X_{i1}[k], X_{i}[k]) = (x_{1}[k], x_{0}[k]) \right) & = \Lmc\left(X_{C(1)}[k] \mid (X_{1}[k], X_{0}[k]) = (x_{1}[k], x_{0}[k]) \right),  \\
	\Lmc\left(X_{C(i2)}[k] \mid (X_{i2}[k], X_{i}[k]) = (x_{2}[k], x_{0}[k]) \right) & = \Lmc\left(X_{C(2)}[k] \mid (X_{2}[k], X_{0}[k]) = (x_{2}[k], x_{0}[k]) \right),
	\end{aligned}
\end{align*}
for each $k \in \N_0$ and $\Lmc(X_2[k],X_1[k],X_0[k])$-almost every $(x_2[k],x_1[k],x_0[k])\in (\R^{k+1})^{3}$.
In particular,
\begin{align}
	\label{eq:eg-2}
	\begin{aligned}
	\Lmc(X_{111}(k) \mid (X_{11}[k], X_{1}[k]) = (x_{1}[k], x_{0}[k]) ) & = \Lmc\left(X_{11}[k] \mid (X_{1}[k], X_{0}[k]) = (x_{1}[k], x_{0}[k]) \right), \\
	\Lmc(X_{222}(k) \mid (X_{22}[k], X_{2}[k]) = (x_{2}[k], x_{0}[k]) ) & = \Lmc\left(X_{22}[k] \mid (X_{2}[k], X_{0}[k]) = (x_{2}[k], x_{0}[k]) \right),
	\end{aligned}	
\end{align}
for $\Lmc(X_{2,1,0}[k])$-almost every $x_{2,1,0}[k]\in (\R^{k+1})^{2}$.
Write
\begin{equation*}
	\Var(X_{11}[k], X_1[k],X_0[k]) = 	
	\begin{bmatrix}
	A_k & B_k^T \\
	B_k & C_k
	\end{bmatrix}, \quad 
	\Var(X_{22}[k], X_2[k],X_0[k]) = 	
	\begin{bmatrix}
	\Atil_k & \Btil_k^T \\
	\Btil_k & \Ctil_k
	\end{bmatrix},
\end{equation*}
where
\begin{align*}
	A_k & := \Var(X_{11}[k]), \quad B_k := \Cov((X_1[k],X_0[k]),X_{11}[k]), \quad C_k := \Var((X_{11}[k], X_1[k],X_0[k])), \\
	\Atil_k & := \Var(X_{22}[k]), \quad \Btil_k := \Cov((X_2[k],X_0[k]),X_{22}[k]), \quad \Ctil_k := \Var((X_{22}[k], X_2[k],X_0[k])),
\end{align*}
which are all available in $\Sigma_k$.
Using this and a standard property of multivariate Gaussian distributions, we have that the conditional laws in \eqref{eq:eg-2} are Gaussian with mean
\begin{equation*}
	\zero + B_k^T C_k^{-1} 
	\left(
	\begin{bmatrix} 
	x_1[k] \\ x_0[k] 
	\end{bmatrix} 
	-
	\zero\right), \quad
	\zero + \Btil_k^T \Ctil_k^{-1} 
	\left(
	\begin{bmatrix} 
	x_2[k] \\ x_0[k] 
	\end{bmatrix} 
	-
	\zero\right)
\end{equation*}
and variance
\begin{equation*}
	A_k - B_k^T C_k^{-1} B_k, \quad \Atil_k - \Btil_k^T \Ctil_k^{-1} \Btil_k.
\end{equation*}	
By the law of total covariance and the conditional independence property proven in Theorem \ref{thm:cond_independence}, we have
\begin{align*}
	& \Cov(X_{111}(k),X_{222}(k)) \\
	& = \Cov\Big( \Emb[X_{111}(k) \mid (X_{11}[k],X_1[k],X_{22}[k],X_2[k])], \Emb[X_{222}(k) \mid (X_{11}[k],X_1[k],X_{22}[k],X_2[k])] \Big) \\
	& = \Cov\Big( \Emb[X_{111}(k) \mid (X_{11}[k],X_1[k])], \Emb[X_{222}(k) \mid (X_{22}[k],X_2[k])] \Big).
\end{align*}
Plugging in the above expressions of conditional distributions in \eqref{eq:eg-2}, we have
\begin{align*}
	\Cov(X_{111}(k),X_{222}(k))
	& = \Cov\left( 	
		B_k^T C_k^{-1} 
		\begin{bmatrix} 
		X_{11}[k] \\ X_1[k] 
		\end{bmatrix}		
		,
		\Btil_k^T \Ctil_k^{-1}
		\begin{bmatrix} 
		X_{22}[k] \\ X_2[k] 
		\end{bmatrix}
		\right) \\
	& = B_k^T C_k^{-1} \Cov((X_{11}[k], X_1[k]),(X_{22}[k], X_2[k])) \Ctil_k^{-1} \Btil_k.
\end{align*}
Note that here every matrix is available in $\Sigma_k$.
The inverse of $C_k$ and $\Ctil_k$ could be calculated in a similar manner as in the previous example. 
As these matrices have dimensions $O(k) \times O(k)$, the total computational complexity from $k$ to $k+1$ is still $O(k^3)$.

To sum up, compared with the previous example, due to the lack of  global symmetry, we will consider vertices from the first two generations and the notation is more involved.
But using the conditional independence and conditional consistency properties, we could still design an inductive algorithm that has a polynomial computational complexity of $O(k^4)$.

\bibliographystyle{plain}


\end{document}